\documentclass[11pt,a4paper]{article}
\usepackage[utf8x]{inputenc}
\usepackage[T1]{fontenc}
\usepackage[english]{babel} %
\usepackage[pdftex]{hyperref} 
\usepackage{enumitem} 
\usepackage[a4paper, lmargin=0.1666\paperwidth, rmargin=0.1666\paperwidth, tmargin=0.1111\paperheight, bmargin=0.1111\paperheight]{geometry} 
\hypersetup{colorlinks=true,  linkcolor=red, citecolor=red, urlcolor=cyan}

\usepackage{thmtools}
\usepackage{amsmath, amssymb, amsfonts}
\usepackage{csquotes}
\usepackage{amsthm}
\usepackage{cleveref}
\usepackage{xcolor}

\usepackage{graphicx}
\usepackage{diagbox}
\usepackage{hhline}

\makeatletter
\def\blfootnote{\gdef\@thefnmark{}\@footnotetext}
\makeatother

\date{}
\title{\textbf{Classification complexity of homeomorphism group actions}}

\author{Michal Hevessy\footnote{email: hevessy@karlin.mff.cuni.cz; ORCID: 0009-0003-4192-2407. This author was supported by the grant GA UK no. 336326.
}\\
Department of Mathematical Analysis\\
Faculty of Mathematics and Physics, Charles University\\
Prague, Czechia\\
\and 
Benjamin Vejnar\footnote{email: vejnar@karlin.mff.cuni.cz; ORCID: 0000-0002-2833-5385. 
}\\
Department of Mathematical Analysis\\
Faculty of Mathematics and Physics, Charles University\\
Prague, Czechia\\}

\theoremstyle{plain}
\newtheorem*{Theorem*}{Theorem}
\newtheorem*{Lemma*}{Lemma}
\newtheorem*{Proposition*}{Proposition}
\newtheorem*{Question*}{Question}
\newtheorem*{Conjecture*}{Conjecture}

\theoremstyle{definition}
\newtheorem*{Definition*}{Definition}

\theoremstyle{remark}
\newtheorem*{Remark*}{Remark}
\newtheorem*{Notation}{Notation}
\newtheorem*{Example*}{Example}

\theoremstyle{plain}
\newtheorem{Question}{Question}
\newtheorem{Theorem}{Theorem}[section]
\newtheorem{Lemma}[Theorem]{Lemma}
\newtheorem{Proposition}[Theorem]{Proposition}
\newtheorem{Corollary}[Theorem]{Corollary}

\newtheorem{Definition}[Theorem]{Definition}
\theoremstyle{remark}
\newtheorem{Remark}[Theorem]{Remark}

\AddToHook{env/Lemma/begin}{\crefalias{Theorem}{Lemma}}
\AddToHook{env/Proposition/begin}{\crefalias{Theorem}{Proposition}}
\AddToHook{env/Corollary/begin}{\crefalias{Theorem}{Corollary}}
\AddToHook{env/Definition/begin}{\crefalias{Theorem}{Definition}}
\AddToHook{env/Remark/begin}{\crefalias{Theorem}{Remark}}
\AddToHook{env/Example/begin}{\crefalias{Theorem}{Example}}

\crefname{Theorem}{theorem}{theorems}
\Crefname{Theorem}{Theorem}{Theorems}

\crefname{Lemma}{lemma}{lemmas}
\Crefname{Lemma}{Lemma}{Lemmas}

\crefname{Proposition}{proposition}{propositions}
\Crefname{Proposition}{Proposition}{Propositions}

\crefname{Corollary}{corollary}{corollaries}
\Crefname{Corollary}{Corollary}{Corollaries}

\crefname{Definition}{definition}{definitions}
\Crefname{Definition}{Definition}{Definitions}

\crefname{Remark}{remark}{remarks}
\Crefname{Remark}{Remark}{Remarks}

\crefname{Example}{example}{examples}
\Crefname{Example}{Example}{Examples}

\newenvironment{claimproof}{
\begin{proof}[Proof of claim] 
}
{\end{proof}}

\newcounter{proofdepth} 
\AtBeginEnvironment{proof}{ 
    \ifnum\value{proofdepth}=0 
        \setcounter{innerclaim}{0}
    \fi 
    \stepcounter{proofdepth}
} 
\AtEndEnvironment{proof}{
    \addtocounter{proofdepth}{-1}
} 
\newtheorem{innerclaim}{Claim}

\newenvironment{Claim}{
    \ifnum\value{proofdepth}=0 
        \PackageError{claim}{The claim environment can only be used inside a proof} {Move this claim inside a proof environment.}
    \fi 
    \begin{innerclaim}%
    }
{\end{innerclaim}}

\newcommand{\Bleq}{\leq_B}
\newcommand{\Bgeq}{\geq_B}
\newcommand{\Bsim}{\sim_B}

\DeclareMathOperator{\id}{Id}
\newcommand{\abs}[1]{\left| #1 \right|}

\DeclareMathOperator{\cl}{cl}
\newcommand{\closure}[1]{\cl\left(#1\right)}
\DeclareMathOperator{\Fix}{Fix}
\DeclareMathOperator{\diam}{diam}

\newcommand{\dac}{d_{\AC}}
\newcommand{\ddif}{d_{\Dif}}

\DeclareMathOperator{\Dif}{Diff}
\DeclareMathOperator{\AC}{AC}
\DeclareMathOperator{\Lip}{Lip}

\newcommand{\Z}{\mathbb{Z}}
\newcommand{\Q}{\mathbb{Q}}
\newcommand{\R}{\mathbb{R}}
\newcommand{\N}{\mathbb{N}}

\newcommand{\ignore}[1]{}
\newcommand{\der}{\mathrm{d}}

\begin{document}
\maketitle
\begin{abstract}
    In this paper, we study how the classification complexity of natural orbit equivalence relations changes
    when the full homeomorphism group of a compact metrizable space is replaced by
    a dense non-closed subgroup. For a compact space $X$ and a subgroup $G \leq \mathcal{H}(X)$, we consider three canonical actions: the left shift action on $\mathcal{H}(X)$, the induced
    hyperspace action on $\mathcal{F}(X)$, and the conjugation action on $G$. We first analyze
    subgroups of the group $\mathcal{H}^+([0,1])$ of increasing interval homeomorphisms,
    focusing on bi-Lipschitz homeomorphisms, diffeomorphisms, and bi-absolutely
    continuous homeomorphisms. We show that, in contrast to the behavior of
    closed subgroups, passing to these subgroups strictly increases
    the complexities of the associated classification problems or makes them incomparable
    with the corresponding full-group relations. In the second part, we
    investigate hyperspace actions of bi-absolutely continuous homeomorphisms on the
    Cantor space and the Hilbert cube with respect to some Borel probability measure and show that a similar behavior occurs on these spaces as well.
\end{abstract}

\section{Introduction}
A central theme of invariant descriptive set theory is the measurement of the complexity of classification problems arising from natural mathematical structures \cite{The_complexity_of_the_classification_problems_of_finite-dimensional_continua, Computing_the_complexity_of_the_relation_of_isometry_between_separable_Banach_spaces, The_complexity_of_the_homeomorphism_relation_between_compact_metric_spaces}. One of the most natural ways to classify objects is to consider them up to some symmetries (compact spaces up to homeomorphism, graphs up to isomorphism, Banach spaces up to linear isometry, etc.). This is formalized by encoding the symmetries into a group structure and considering the so-called orbit  equivalence relations. Formally, let \(X\) be a Polish space, and let \(G\) be a group acting on \(X\), the 
associated \textit{orbit equivalence relation} is defined by
\[x E^X_G y \iff \exists g \in G (g \cdot x = y).\]
This allows us to use the tools of invariant descriptive set theory, mainly the notion of Borel reductions, to compare classification problems from different parts of mathematics. We will discuss this in more detail in the following section, where we also provide some important results that will be used throughout this paper. The classification is usually done by considering so-called benchmark equivalence relations 
and placing the studied problem among those. The notion of benchmark equivalence relation is rather informal. These are equivalence relations that arise very often in classification problems or have some undeniable theoretical significance. We will list some in the following sections. 

The aim of this paper is to investigate how passing from the full group to subgroups changes the complexity of the induced orbit equivalence relation. We focus on homeomorphism groups of compact spaces. All spaces under consideration are assumed to be metrizable. Let \(X\) be a compact space, \(\mathcal{H}(X)\) be the homeomorphism group considered with the supremum metric, and \(G\) an algebraic subgroup of \(\mathcal{H}(X)\). There are three very natural actions and orbit equivalences induced by \(G\). The first is the \textit{left shift action} of \(G\) on \(\mathcal{H}(X)\), i.e., \(g \cdot h = g \circ h\), and the \textit{shift equivalence} on \(\mathcal{H}(X)\)
\[h_1 E_G^{S} h_2 \iff \exists g \in G \ g \circ h_1 = h_2 \iff h_2 \circ h_1^{-1} \in G.\] 
It is also possible to consider the right shift action. However, we only focus on the left shift. The second is the \textit{hyperspace action}, i.e., the natural action of \(G\) on \(\mathcal{F}(X)\), the space of non-empty closed subsets or, equivalently, compact subsets of \(X\) 
\[g \cdot F = g(F)\] 
for \(g \in G\), \(F \in \mathcal{F}(X)\). The induced \textit{hyperspace equivalence relation}
\[F_1 E_{G}^{H} F_2 \iff \exists g \in G  \ g(F_1) = F_2\] 
can be regarded as classifying subsets of \(X\) up to \(G\) homeomorphisms. The third and final one is the \textit{conjugation action} of \(G\) on itself 
\(g \cdot f = g \circ f \circ g^{-1}\). The induced \textit{conjugation equivalence relation}
\[f_1 E_{G}^C f_2 \iff \exists g \in G \ g \circ f_1 \circ  g^{-1} = f_2 \iff \exists g \in G \ g \circ f_1 = f_2 \circ g\]
can be seen as classifying the dynamical systems arising from \(G\).
These relations are natural from both dynamical and descriptive-set-theoretic points of view: they ask whether two closed sets, two transformations, or two dynamical systems are equivalent under a prescribed regularity class of coordinate changes.

The first naive conclusion would be that passing to a subgroup would decrease the complexity of the equivalence relations; after all, we only need to consider a smaller set of transformations. For closed subgroups, the intuition is correct in the following sense:
\begin{Theorem}[{\cite[Theorem 3.5.2]{Gao}}]
    Let \(G\) be a Polish group and \(H\) a closed subgroup of \(G\). Suppose \(H\) acts on a Polish space \(X\) by a continuous action \(a\). Then there exists a Polish space \(Y\) and a continuous action \(b\) of \(G\) on \(Y\) such that:
    \begin{enumerate}
        \item \(X\) is a closed subset of \(Y\),
        \item \(a(h,x) = b(h,x)\) for all \(h \in H\) and \(x \in X\),
        \item every \(G\)-orbit in \(Y\) contains exactly one \(H\)-orbit in \(X\).
    \end{enumerate}
\end{Theorem}

However, we will work with subgroups that are not closed. It turns out that in this case, and if the subgroup is rich enough, the complete opposite behavior appears, and the complexity of the equivalences increases, or their complexities become incomparable. A deeper consideration of the matter reveals that this behavior is, in fact, natural. When considering the orbit equivalence relation induced by a subgroup, we are not only asking whether two objects are equivalent but also whether this equivalence can be realized by an element of the subgroup, which is possibly a more complicated problem. 

The paper is split into two parts. In the first part, we consider three subgroups of the group of increasing interval homeomorphisms \(\mathcal{H}^+([0,1])\). Namely, the group of bi-Lipschitz homeomorphisms, diffeomorphisms, and bi-absolutely continuous homeomorphisms. The obtained results are summarized in \Cref{Tabular: interval results}. We show that all of these subgroups demonstrate the expected behavior, and all of the considered equivalences either become strictly more complicated or are incomparable. It is also worth mentioning that all of the subgroups exhibit markedly different behaviors, which shows that the different regularity conditions carry substantially different complexities with them.

The case of bi-Lipschitz homeomorphisms fits with earlier work of Rosendal \cite{Cofinal_families_of_Borel_equivalence_relations_and_quasiorders}, who studied the complexity of the bi-Lipschitz classification of compact metric spaces. One of the striking consequences of the present paper is that the same complexity already appears when one restricts attention to closed subsets of the interval. Thus, from the point of view of classification complexity, the interval contains sufficiently rich configurations to realize the full difficulty of bi-Lipschitz classification.

The group of bi-absolutely continuous homeomorphisms of the interval was studied by Solecki \cite{Polish_group_topologies}. He has shown that it can be made into a Polish group by considering a certain stronger metric. He also studied the shift equivalence relation of this group. Some algebraic properties for the group of bi-absolutely continuous homeomorphisms of the interval were studied in \cite{generation_and_genericity_for_AC}.

Perhaps the most investigated case in the literature is that of diffeomorphisms; see \cite{The_complexity_and_the_structure_and_classification_of_dynamical_systems} for a survey of results. However, there has not been much attention dedicated to the interval. 

The situation of the shift actions can be compared to a similar construction for \(\R^\N\) and linear subsets that become Banach spaces under a different topology \cite{Actions_by_the_classical_Banach_spaces}. In the shift action, the equivalence classes are, in fact, the right cosets of the considered subgroup, and thus all of them are homeomorphic to the considered subgroup.

In the second part of the paper, we investigate the hyperspace action of bi-absolutely continuous homeomorphisms on the Hilbert cube and the Cantor space. See \Cref{Tabular: AC results} for a summary of results. By a bi-absolutely continuous homeomorphism, we mean a homeomorphism that preserves null sets of some Borel measure. This makes the group depend on the choice of the measure. However, we show that, under some mild assumptions on the measure, the complexity of the equivalence relation is independent of the choice of the measure. For this, we rely on the results of \cite{Homeomorphic_measures_in_the_Hilbert_cube} and some parallels on the Cantor space. 

{
\Crefname{Theorem}{Thm.}{Thms.}
\Crefname{Corollary}{Cor.}{Cors.}
\Crefname{Proposition}{Prop.}{Props.}
{\renewcommand{\arraystretch}{1.66}
\begin{table}
    \centering
    \caption{Summary of results for subgroups of \(\mathcal{H}^+([0,1])\)}
    \label{Tabular: interval results}
\begin{tabular}{c||c|c|c|c}
    \diagbox[width=6.5em]{Action}{Group}  & \(\mathcal{H}^+([0,1])\) &\(\Lip([0,1])\) &\(\Dif([0,1])\) & \(\AC([0,1])\) \\ 
    \hhline{=\|=|=|=|=}
    
    Hyperspace & \(E_{S_\infty}\) \cite{Hjorth} & \(\ell_\infty\) \Cref{Corollary: Hypespace and shift for Lipschitz are ell infinity}  & \begin{tabular}{c}
        \(\Bgeq c_0\) \Cref{Theorem: Dif classification on interval contains c_0} \\
        \(\not\Bgeq E_{S_\infty}\) \Cref{Corollary: hyperspace and conjugation for diff are not above E_S_infty}
    \end{tabular} & \(>_B E_{S_\infty}\) \Cref{Corollary: AC equivalence on interval is above E_{S_ifinity}}\\
    
    Shift & trivial &  \(\ell_\infty\) \Cref{Corollary: Hypespace and shift for Lipschitz are ell infinity} & Turbulent \Cref{Proposition: Shift of Diff is turbulent} &  Turbulent \cite{Polish_group_topologies} \\
    
    Conjugation & \(E_{S_\infty}\) \cite{Hjorth} & \(\ell_\infty\) \Cref{Corollary: conjugatio for Lip is l infinity}& 
    \begin{tabular}{c}
         \(\Bgeq c_0\) \Cref{Theorem: Conjugation of Dif is above c_0}  \\
         \(\not \Bgeq E_{S_\infty}\) \Cref{Corollary: hyperspace and conjugation for diff are not above E_S_infty}
    \end{tabular}& \(>_B E_{S_\infty}\) \Cref{Corollary: AC cojugation and hyperspace are the same},  \Cref{Corollary: AC equivalence on interval is above E_{S_ifinity}}\\
\end{tabular}
\end{table}}

\begin{table}
    \centering
    \caption{Summary of results for bi-absolutely continuous homeomorphisms on other spaces}
    \label{Tabular: AC results}
{\renewcommand{\arraystretch}{1.66}
\begin{tabular}{c||c|c|c}
     \diagbox[width = 5em]{Group}{Space} & \(2^\N\) &\([0,1]\) & \([0,1]^\N\) \\ 
      \hhline{=\|=|=|=}
    \(\mathcal{H}(X)\) & \(E_{S_\infty}\) \cite{Hjorth} & \(E_{S_\infty}\) \cite{Hjorth}& \(E_{G_\infty}\) \cite{The_complexity_of_the_homeomorphism_relation_between_compact_metric_spaces} \\
        \hline
    \(\AC(X)\) & \(>_B E_{S_\infty}\) \Cref{Corollary: AC equivalence on Cantor is strictly above S_infinity} & \(>_B E_{S_\infty}\) \Cref{Corollary: AC equivalence on interval is above E_{S_ifinity}} & \( E_{G_\infty}\) \Cref{Theorem: AC Hyperspace on Hilbert cube is complete}
\end{tabular}}
\end{table}}

\section{Preliminary results and definitions}
In this section, we recall some standard notions and results from descriptive set theory and invariant descriptive set theory. For further reading on these topics, we refer the reader to \cite{Gao, Classical_descriptive_set_theory}. A \textit{Polish space} is a completely metrizable separable space. A \textit{Polish group} is a topological group with a completely metrizable separable topology. For two Polish spaces \(X,Y\) we write \(X \approx Y\) if \(X\) and \(Y\) are homeomorphic. A measurable space \((X,\Sigma)\) is called a \textit{standard Borel space} if \(\Sigma\) is the Borel \(\sigma\)-algebra of some Polish topology on \(X\). Any Borel subset of a Polish space with the subspace Borel structure is a standard Borel space.
The central notion of invariant descriptive set theory that allows the comparison of classification problems is the notion of Borel reducibility.
\begin{Definition}
    Let \(X,Y\) be standard Borel spaces and \(E,F\) be two equivalence relations on \(X,Y\) respectively. We say that \(E\) is \textit{Borel reducible} to \(F\), denoted by \(E \Bleq F\), if there exists a Borel mapping \(f: X \to Y\) called a \textit{reduction} such that
    \[x E y \iff f(x) F f(y)\]
    for all \(x,y \in X\). We say that \(E\) is \textit{Borel bireducible} with \(F\), denoted by \(E \Bsim F\), if \(E \Bleq F\) and \(F \Bleq E\). We write \(E <_B F\) if \(E \Bleq F\) and \(F \not \Bleq E\).
\end{Definition}
Informally, \(E \Bleq F\) means that the structure of \(E\) can be realized in a Borel way in the structure of \(F\). In this case, we can regard  \(E\) as being at most as complex as \(F\) and \(F\) as being at least as complex as \(E\). When the strict reduction \(E <_B F\) occurs, we can regard \(F\) as strictly more complex than \(E\) and \(E\) as strictly less complex than \(F\). This notion of relative complexity is sometimes referred to as Borel complexity on Borel equivalence relations. Two important points should be noted. First, for equivalence relations \(E\) and \(F\), it can occur that neither \(E \Bleq F\) nor \(F \Bleq E\). In this case, we say that \(E\) and \(F\) are incomparable. Second, if \(E\) and \(F\) are equivalence relations and \(E \Bsim F\), both reductions can be captured by very different mappings and thus, even if \(E\) and \(F\) have the same complexity, they can be very different. The intuitive notion of \enquote{being the same} is captured by isomorphism.
\begin{Definition}
    Let \(X,Y\) be standard Borel spaces and \(E,F\) be two equivalence relations on \(X,Y\) respectively. We say that \(E\) and \(F\) are \textit{Borel isomorphic} if there is a Borel bijective reduction \(f: X \to Y\).
\end{Definition}
\begin{Remark*}
    By the classical result on injective Borel mappings {\cite[Theorem 15.1]{Classical_descriptive_set_theory}} a Borel bijection is automatically a Borel isomorphism.
\end{Remark*}

A satisfying classification result is obtained when the studied equivalence relation is shown to be Borel bireducible with some of the aforementioned benchmark equivalence relations. We list some of these benchmark equivalence relations. This is far from a complete list; we focus only on those that are important to our investigation. For a more comprehensive treatment, see \cite{Gao}. The space \(c_0(\N)\) is the space of infinite sequences of real numbers converging to \(0\) and the space \(\ell_\infty(\N)\) is the space of bounded infinite sequences of real numbers.
\begin{Definition}
    The equivalence relation \(c_0\) is the coset equivalence relation of \(c_0(\N)\) on \(\R^\N\), that is
    \[(x_n) c_0 (y_n) \iff (x_n-y_n) \in c_0(\N)  \iff x_n -y_n \overset{n \to \infty}{\to} 0\]
    for \((x_n), (y_n) \in \R^\N\).
\end{Definition}
\begin{Remark*}
    By an abuse of notation we will use \(c_0\) to denote the space \(c_0(\N)\) as well as the equivalence relation. This will not cause any confusion and it will always be clear from the context which one we mean.
\end{Remark*}
Although this may not be apparent, we will show that this equivalence relation is natural to consider in the case of diffeomorphisms. A similar construction can be considered for the space of bounded sequences.
\begin{Definition}
    The equivalence relation \(\ell_\infty\) is the coset equivalence relation of \(\ell_\infty(\N)\) on \(\R^\N\) that is
    \[(x_n) \ell_\infty (y_n) \iff (x_n-y_n) \in \ell_\infty(\N) \iff \exists K > 0 \ \sup_{n \in \N}|x_n-y_n| <K\]
    for \((x_n), (y_n) \in \R^\N\).
\end{Definition}
\begin{Remark*}
    We will again abuse the notation and denote by \(\ell_\infty\) both the space and the equivalence relation. Again, it will always be apparent from the context which one we mean.
\end{Remark*}
The equivalence \(\ell_\infty\) has an important theoretical property, and this is one of the main reasons why it is considered to be a benchmark equivalence relation. A subset of a Polish space is \(K_\sigma\) if it can be written as a countable union of compact sets.
\begin{Theorem}[{\cite[Theorem 8.4.2]{Gao}}]
\label{Theorem: ell infinity is universal K sigma}
    If \(E\) is a \(K_\sigma\) equivalence relation on a Polish space \(X\) then \(E \Bleq \ell_\infty\).
\end{Theorem}
\begin{Remark*}
    Note that \(\ell_\infty\) itself is not a \(K_\sigma\) equivalence relation. However, there are equivalence relations that are Borel bireducible with \(\ell_\infty\) that are \(K_\sigma\).
\end{Remark*}

Universal equivalence relations often play an important role, and when an equivalence relation is shown to be universal, it is often regarded as a benchmark equivalence relation. There are many natural classes of equivalence relations that have a universal element, as well as classes that lack this property. One of the classes that enjoys this property is the class of orbit equivalence relations induced by Borel actions of Polish groups. A universal relation for this class can be constructed from an action of a universal Polish group, but the construction is not important for our purposes. 
\begin{Theorem}[{\cite[Theorem 5.1.9]{Gao}}]
    There exists an orbit equivalence relation \(E_{G_\infty}\) induced by a Borel action of a Polish group such that whenever \(E\) is an orbit equivalence relation induced by a Polish group acting in a Borel way on a Polish space \(X\) then \(E \Bleq E_{G_\infty}\).
\end{Theorem}
This orbit equivalence relation is usually referred to as the \textit{complete orbit equivalence relation}.
Zieliński \cite{The_complexity_of_the_homeomorphism_relation_between_compact_metric_spaces} has shown that this can be realized by a very concrete equivalence relation.
\begin{Theorem}[\cite{The_complexity_of_the_homeomorphism_relation_between_compact_metric_spaces}]
\label{Theorem: Hyperspace relation on Hilber cube is universal}
    The equivalence relation 
    \[\{(A,B) \in \mathcal{F}([0,1]^\N)^2; A \approx B\}\]
    on \(\mathcal{F}([0,1]^\N)\) is Borel bireducible with \(E_{G_\infty}\).
\end{Theorem}
Although classification complexity is a different notion from Borel complexity, some information can still be gained from investigating the Borel complexity of a given equivalence relation. For instance, it can be shown that the equivalence \(E_{G_\infty}\) is a complete analytic set in the product space.

A complete classification of a given equivalence relation is often an intractable problem. However, results showing that certain reductions are impossible, i.e., that a given equivalence relation cannot be reduced to some benchmark equivalence relation, also offer valuable insight. These are usually regarded as non-classification or anti-classification results. This can sometimes be even more difficult than showing that a reduction is possible. There is no general theory developed that could be used to show anti-classification for a variety of different benchmark equivalence relations. However, there are some tools available.
\begin{Theorem}[{\cite[Section 8.4, Section 10.6]{Gao}}]
\label{Theorem: ell infty is not group action}
    We have
    \[\ell_\infty \not \Bleq E_{G_\infty}.\] 
    In other words, whenever a Polish group \(G\) acts by a Borel action on a Polish space \(X\) then 
    \[\ell_\infty \not\Bleq E_{G}^{X}.\]
\end{Theorem}
\begin{Remark*}
    This result can be used to show that some problems are not reducible to any orbit equivalence relation induced by a Borel action of a Polish group.
\end{Remark*}

An important level that is often used when showing anti-classification results is the so-called classification by countable structures. This complexity level can be described by model-theoretic means that justify the terminology. We describe this notion purely in the terminology of descriptive set theory, as the model-theoretic perspective is beyond the scope of this text. We refer the interested reader to \cite{Hjorth} for a model-theoretic description and a far more detailed investigation of the matter.  We will denote by \(S_\infty\) the permutation group of \(\N\) with the topology inherited from \(\N^\N\).
\begin{Theorem}[\cite{Hjorth}]
    There exists an orbit equivalence relation \(E_{S_\infty}\) induced by a Borel action of \(S_\infty\) such that whenever \(S_\infty\) acts by a Borel action on a Polish space \(X\) we have
    \[E_{S_\infty}^X \Bleq E_{S_\infty}.\]
\end{Theorem}
\begin{Definition}
    Let \(E\) be an equivalence relation on a Polish space \(X\). If 
    \[E \Bleq E_{S_\infty}\]
    we say that \(E\) is \textit{classifiable by countable structures}. If 
    \[E \not\Bleq E_{S_\infty}\]
    we say that \(E\) is \textit{not classifiable by countable structures}.
\end{Definition}
It can again be shown that the equivalence relation \(E_{S_\infty}\) is a complete analytic set in the product space.

There are several natural classification problems that are Borel bireducible with \(E_{S_\infty}\). We state this explicitly for those that are most relevant to our investigation.
\begin{Theorem}[{\cite[Theorem 4.10]{Hjorth}}]
\label{Theorem: order isomorphisms on Q is E S infinity}
    Let \(E\) be the following equivalence relation on subsets of \(\Q\):  
    \[A E B \iff \exists \varphi \text{ an order isomorphism of \(\Q\) such that } \varphi(A) = B\]
    for \(A,B \subset \Q\). Then \(E \Bsim E_{S_\infty}\).
\end{Theorem}
\begin{Theorem}[{\cite[Theorem 4.6, Corollary 4.11, Exercise 4.13]{Hjorth}}]
\label{Theorem: equivalences of homemorphisms ar E S infinity}
    We have
    \[E^H_{\mathcal{H}^+([0,1])} \Bsim E_{S_\infty}\]
    and 
    \[E^C_{\mathcal{H}^+([0,1])} \Bsim E_{S_\infty}.\]
\end{Theorem}
By Stone duality, the homeomorphism group of the Cantor space can be realized as a closed subgroup of the permutation group. It follows that any orbit equivalence relation induced by a Borel action of \(\mathcal{H}(2^\N)\) is classifiable by countable structures. In fact, even more is true for the Cantor space. 
\begin{Theorem}[{\cite[Theorem 1]{Completness_of_Boolean_algebras}}]
    \label{Theorem: Hyperspace relation on Cantor space is universal S_infty}
    The equivalence relation 
    \[\{(A,B) \in \mathcal{F}(2^\N)^2; A \approx B\}\]
    on \(\mathcal{F}(2^\N)\) is Borel bireducible with \(E_{S_\infty}\).

\end{Theorem}
Classification by countable structures is one of the only classification levels where a general theory to show anti-classification has been developed. This has been done by Hjorth \cite{Hjorth} and it is usually referred to as Hjorth's turbulence theory. 
\begin{Definition}
    Let \(G\) be a Polish group acting continuously on a Polish space \(X\). For \(x \in X\), \(U \subset X\) open with \(x \in U\) and \(V \subset G\) open with \(1_G \in V\), the \textit{local \(U \text{-} V\)-orbit of} \(x\), denoted \(\mathcal{O}(x,U,V)\), is the set of \(y \in U\) for which there exist \(l \in \N\), \(x = x_1\in U, x_2\in U, \dots ,x_l = y \in U\) and \(g_1, \dots ,g_{l-1} \in V\) such that \(x_{i+1} = g_i \cdot x_i\) for \(i < l\).
\end{Definition}
\begin{Definition}
    Let \(G\) be a Polish group acting continuously on a Polish space \(X\). The action of \(G\) on \(X\) is \textit{turbulent} if
    \begin{enumerate}
        \item[(T1)] every orbit is meager,
        \item[(T2)] every orbit is dense, and
        \item[(T3)] every local orbit is somewhere dense, that is, for any open \(U \subset X\), \( x \in U\) and open \(1_G \in V \subset G\), \(\mathcal{O}(x,U,V)\) is somewhere dense.
    \end{enumerate}
\end{Definition}
The remarkable result of Hjorth states that the mostly dynamical notion of turbulence can capture anti-classification by countable structures.
\begin{Theorem}[Hjorth's turbulence theorem {\cite[Corollary 3.19]{Hjorth}}]
\label{Thoerem: Hjorth turbulenc theorem}
    Let \(G\) be a Polish group acting by a continuous action on a Polish space \(X\). If the action of \(G\) on \(X\) is turbulent then \(E_G^X\) is not classifiable by countable structures.
\end{Theorem}
Hjorth's turbulence theorem can be used to show that \(c_0\) is not classifiable by countable structures.
\begin{Proposition}[{\cite[Corollary 10.5.3]{Gao}}]
\label{Proposition: c_0 is not classifiable}
    We have 
    \[c_0 \not \Bleq  E_{S_\infty}.\]
\end{Proposition}

Since our work focuses on the actions of Polish groups and their subgroups, we shall need several basic facts about such groups. A subgroup of a Polish group need not itself be Polish when equipped with the subspace topology; in fact, this occurs precisely for closed subgroups.
\begin{Theorem}[{\cite[Proposition 2.2.1]{Gao}}]
    Let \(G\) be a Polish group and \(H \leq G\). Then \(H\) equipped with the subspace topology is a Polish group if and only if \(H\) is closed in \(G\).
\end{Theorem}
All of the groups that we consider are dense but not equal to the whole group and thus not Polish with the subspace topology. Most of the theory developed can be used only for Borel actions of Polish groups. This issue can be resolved by changing the topology on the subgroup while preserving its Borel structure. This leads to the notion of a Polishable subgroup.
\begin{Definition}
    Let \(G\) be a Polish group and \(H\) a Borel subgroup of \(G\). We say that \(H\) is \textit{Polishable} if there is a Polish group topology \(\tau\) on \(H\) inducing the same Borel structure on \(H\) as the subspace topology.
\end{Definition}
It turns out that proper dense Borel subgroups of Polish groups are actually meager. This is very useful for employing turbulence theory for the shift action. Since all of the equivalence classes of a shift action are translates of the subgroup, if the subgroup is dense and meager, all the equivalence classes are dense and meager as well. This leaves us only with the property \((T3)\) to be checked to show the turbulence of the shift action.
\begin{Theorem}[{\cite[Theorem 9.9]{Classical_descriptive_set_theory}}]
\label{Theorem: proper subsets of Polish groups are meager}
    Let \(G\) be a Polish group and \(H\) be a proper dense Borel subgroup of \(G\). Then \(H\) is meager.
\end{Theorem}

Since the groups considered here are quite different, we will need a variety of results from different areas. We will introduce these results at the appropriate places.

Later on we will work with measures on compact spaces. Let us fix some notation and terminology.
\begin{Definition}
    Let \(X\) be a compact space. By a Borel measure on \(X\) we mean a measure on \(X\) such that the \(\sigma\)-algebra of measurable sets contains the Borel sets of \(X\). For a Borel measure \(\mu\) on \(X\) we will denote \(\|\mu\|\) the total variation of \(\mu\). We say that a Borel measure \(\mu\) on \(X\) is a probability measure if \(\mu(X) = 1\).
\end{Definition}
\begin{Definition}
    Let \(X\) be a compact space and \(\mu\) a Borel measure on \(X\). We say that \(\mu\) is locally positive if \(\mu(U) > 0\) for every non-empty \(U \subset X\) open. We say that \(\mu\) is non-atomic if \(\mu(\{x\}) = 0\) for every \(x \in X\).
\end{Definition}

\section{Homeomorphism groups of the interval}
In this section, we study the hyperspace, shift, and conjugation actions and equivalences of three different subgroups of the interval. The three groups considered are bi-Lipschitz homeomorphisms, diffeomorphisms, and bi-absolutely continuous homeomorphisms. We show that the behavior of the actions changes significantly for the different groups.  Before coming to the results of this section, let us fix some terminology and notation.
\begin{Definition}
    Let 
    \[\Lip([0,1]) = \{f \in \mathcal{H}^+([0,1]); f \text{ is bi-Lipschitz}\}\]
    and for \(K > 0\) let
    \[\Lip_K([0,1]) = \left\{f \in \mathcal{H}^+([0,1]); \begin{array}{c}
        \sup_{x,y \in [0,1], x \neq y}\frac{|f(x)-f(y)|}{|x-y|}\leq K,  \\[6pt]
        \sup_{x,y \in [0,1], x \neq y}\frac{|f^{-1}(x)-f^{-1}(y)|}{|x-y|} \leq K
    \end{array} \right\}.\] 
 For \(f \in \Lip([0,1])\), we will denote by \(\Lip(f)\) the Lipschitz constant of \(f\).

    Let \[\Dif([0,1]) =\{f \in \mathcal{H}^+([0,1]); f' \in \mathcal{C}([0,1]), (f^{-1})' \in \mathcal{C}([0,1])\}.\] We will consider the space \(\Dif([0,1])\) with the following metric \[\ddif(f,g) = \sup_{t \in [0,1]}|f(t)-g(t)|+\sup_{t \in [0,1]}|f'(t)-g'(t)|,\] for \(f,g \in \Dif([0,1])\).

    Let 
    \[\AC([0,1]) = \{f \in \mathcal{H}^+([0,1]); f,f^{-1} \text{ are absolutely continuous}\}.\]
    We will consider the space \(\AC([0,1])\) with the following metric
    \[\dac(f,g) = \sup_{t \in [0,1]}|f(t)-g(t)| + \int_{0}^{1}|f'(t)-g'(t)|\der t\]    
\end{Definition}
\begin{Notation}
    We will denote by \(B_{\infty}\), respectively \(B_{\Dif}\), respectively \(B_{\AC}\) open balls in the supremum metric, respectively \(d_{\Dif}\), respectively \(d_{\AC}\).
\end{Notation}
Solecki {\cite[Lemma 2.4]{Polish_group_topologies}} has shown that the space \(\AC([0,1])\) with the metric \(\dac\) is a Polish group with the same Borel structure as the one generated by the supremum metric. Note, however, that the metric \(\dac\) is not complete. 
It can be easily shown that \(\Dif([0,1])\) equipped with the topology induced by the metric \(\ddif\) becomes a Polish group with the same Borel structure as the one generated by the supremum metric. However, again, the metric \(\ddif\) is not complete. We do not specify a special metric to be used for the space of bi-Lipschitz homeomorphisms. There are two reasons for this. The first is that to classify the considered equivalences, it is enough to consider the supremum metric, and the second is that there is actually no metric that would make it into a Polish group with the same Borel structure. The second fact was shown in \cite{Subgroups_of_Homeo_without_polish_topology} and we will comment on it later.

All of these groups are dense in \(\mathcal{H}^+([0,1])\). It can also be easily shown that all of these groups are Borel. This, among other things, implies that the shift equivalence is a Borel equivalence relation for all of these groups. This already gives us some information on the complexities of the shift equivalence relations.

\subsection{Lipschitz}
In this section, we focus on the group of bi-Lipschitz homeomorphisms. The case of bi-Lipschitz homeomorphisms is very different compared to the other subgroups. One of the main reasons for this, and the main difference compared to the other two groups, is the topological properties of the space \(\Lip([0,1])\). The following lemma is a well known property of the space of Lipschitz functions. We include it for the sake of completeness.
\begin{Lemma}
\label{Lemma: Lipcshitz functions are compact}
    The space \(\Lip_K([0,1])\) is compact for every \(K > 0\).
\end{Lemma}
\begin{proof}
    This is an easy consequence of the Arzelà-Ascoli theorem \cite[Chapter 7, Theorem 17]{Kelley_General_topology}.
\end{proof}
As a consequence, the space \(\Lip([0,1])\) is \(K_\sigma\) in \(\mathcal{H}^+([0,1])\). Thus, there is an upper bound for the complexity of any induced orbit equivalence relation on \(K_\sigma\) spaces.
\begin{Proposition}
\label{Proposition: equivalences induced by Lipschitz are K_sigma}
    Suppose \(\Lip([0,1])\) acts on a \(K_\sigma\) Polish space \(X\) by a continuous action. Then \[E_{\Lip([0,1])}^{X}\Bleq \ell_\infty.\]
\end{Proposition}
\begin{proof}
    For \(n \in \N\) let \(G_n = \Lip_n([0,1])\). By \Cref{Lemma: Lipcshitz functions are compact} we have that for every \(n \in \N\) the set \(G_n\) is compact. Note that we have \(G_n \subset {G_{n+1}}\). Let \(X = \bigcup_{n \in \N}K_n\), where \(K_n\) are compact and such that \(K_{n} \subset K_{n+1}\). For \(n \in \N\), let 
    \[E_n = \{(x,y) \in K_n \times K_n ; \exists f \in G_n \ f\cdot x = y\}.\] 
    Since \(G_n\) and \(K_n\) are compact and the action is continuous, we get that \(E_n\) is compact as well. We will show that 
    \[\bigcup_{n \in \N}E_n = E_{\Lip([0,1])}^X.\]
    The inclusion 
    \[\bigcup_{n \in \N}E_n \subset E_{\Lip([0,1])}^X\]
    is clear. To show the other inclusion, let \((x,y) \in E^{X}_{\Lip([0,1])}\). Let \(f \in \Lip([0,1])\) be such that \(f\cdot x = y\). There exist \(n_1,n_2 \in \N\) such that \(x,y \in K_{n_1}\) and \(f \in G_{n_2}\). Then \((x,y) \in E_{\max\{n_1,n_2\}}\).
    Thus, \(E^X_{\Lip([0,1])}\) is a \(K_\sigma\) equivalence relation and by \Cref{Theorem: ell infinity is universal K sigma} we get \(E^X_{\Lip([0,1])} \Bleq \ell_\infty\).
\end{proof}
We can use this to bound the complexity of the hyperspace equivalence relation.
\begin{Corollary}
\label{Corollary: hyperspace and conjugation ar below ell_infty}
    We have
    \[E^H_{\Lip([0,1])} \Bleq \ell_\infty.\]
\end{Corollary}
\begin{proof}
    Since \(\mathcal{F}([0,1])\) is compact and the action is continuous, we can apply \Cref{Proposition: equivalences induced by Lipschitz are K_sigma}.
\end{proof}
To bound the complexity of the remaining two equivalences, we need to use a different strategy. The space \(\mathcal{H}^+([0,1])\) is not \(K_\sigma\) and, although \(\Lip([0,1])\) is a \(K_\sigma\) metric space, it is not Polish. However, we can still reduce both equivalence relations to \(K_\sigma\) equivalence relations.
\begin{Proposition}
\label{Proposition: shift is below ell_infty}
    We have 
    \[E^S_{\Lip([0,1])} \Bleq \ell_\infty\]
    and
    \[E^C_{\Lip([0,1])} \Bleq \ell_\infty.\]
\end{Proposition}
\begin{proof}
    For the first reduction, consider an action of \(\Lip([0,1])\) on \(\mathcal{F}([0,1]^2)\) defined as follows
    \[f \cdot F = \{(f(x),y); (x,y) \in F\}\]
    for \(F \in \mathcal{F}([0,1]^2)\) and \(f \in \Lip([0,1])\). Denote by \(E\) the induced orbit equivalence relation. By \Cref{Proposition: equivalences induced by Lipschitz are K_sigma} we get \(E \Bleq \ell_\infty\). Consider the map 
    \[\mathcal{H}^+([0,1]) \ni f \mapsto \{(f(x),x); x \in [0,1]\} \in \mathcal{F}([0,1]^2).\] 
    By the definition of the shift action, we can see that this map is a Borel reduction from \(E^S_{\Lip([0,1])}\) to \(E\). Together, we get \(E^S_{\Lip([0,1])} \Bleq \ell_\infty\).
    
    For the second reduction, consider an action of \(\Lip([0,1])\) on \(\mathcal{F}([0,1]^2)\) defined as follows
    \[f \cdot F = \{(f(x),f(y)); (x,y) \in F\}\]
    for \(F \in \mathcal{F}([0,1]^2)\) and \(f \in \Lip([0,1])\). Denote by \(H\) the induced orbit equivalence relation. We can again use \Cref{Proposition: equivalences induced by Lipschitz are K_sigma} to get \(H \Bleq \ell_\infty\). From the definition of conjugation we can easily see that the same map
    \[\Lip([0,1]) \ni f \mapsto \{(f(x),x); x \in [0,1]\} \in \mathcal{F}([0,1]^2)\] 
    is a reduction from \(E^C_{\Lip([0,1])}\) to \(H\). Thus, we get \(E^C_{\Lip([0,1])} \Bleq \ell_\infty\).
\end{proof}

We show that for all the orbit equivalence relations considered, the upper bound is attained. To do this, we will use a multiplicative version of the \(\ell_\infty\) equivalence relation. This is natural since the Lipschitz condition is defined via bounds on quotients. 

\begin{Definition}
    Let \(\ell_\infty^\times\) be the equivalence relation on \((0,1)^\N\) defined by 
    \[(x_n) \ell_\infty^\times (y_n) \iff \exists K > 1 \ \forall n : \left|\frac{x_n}{y_n}\right| \in (\frac{1}{K},K).\]
\end{Definition}
\begin{Remark*}
    It can be easily shown that \(\R^\N/\ell_\infty \Bsim (-\infty,0)^\N/\ell_{\infty}\).
\end{Remark*}
\begin{Lemma}
\label{Lemma: multiplicative ell_infty}
    We have \[\ell_\infty^\times \Bsim \ell_\infty.\]
\end{Lemma}
\begin{proof}
    By the preceding remark it is enough to show \(\ell_\infty^\times \Bsim (-\infty,0)^\N/\ell_\infty\). The mapping \(\log: (0,1)^\N \to (-\infty,0)^\N\) defined by \((x_n) \mapsto (\log(x_n))\) is a Borel isomorphism of \(\ell_\infty^\times\) and \((-\infty, 0)^\N / \ell_\infty\).
\end{proof}

We employ a straightforward strategy for the reductions. We embed sequences into the interval in such a way that the Lipschitz condition exactly corresponds to the quotient of the sequences. 

\begin{Theorem}
\label{Theorem: Hyperspace and shift for lipschitz is above ell infty}
    We have
    \[\ell_\infty^\times \Bleq E_{\Lip([0,1])}^H\]
    and 
    \[\ell^\times_\infty \Bleq E^S_{\Lip([0,1])}.\]
\end{Theorem}
\begin{proof}
    We start with the first reduction. Let \(F: (0,1)^\N \to \mathcal{F}([0,1])\) be defined as: \[F((x_n)) = \{0\} \cup \{\frac{1}{2^n}; n \in \N\} \cup \{\frac{1}{2^n}+\frac{x_n}{2^{n+1}}; n \in \N\}.\] This is clearly a Borel map. We will show that this is a reduction. Let \((x_n),(y_n) \in (0,1)^\N\).
    
     First, suppose that \((x_n) \ell_\infty^\times (y_n)\). For every \(n \in \N\) let 
    \[f_n(x) = 
    \begin{cases}
        \frac{y_n}{x_n}x + (1-\frac{y_n}{x_n})\frac{1}{2^n}, \quad &x \in (\frac{1}{2^n},\frac{1}{2^n}+\frac{x_n}{2^{n+1}}] \\
        \frac{y_n-2}{x_n-2}x+(1-\frac{y_n-2}{x_n-2})\frac{1}{2^{n-1}}, \quad &x \in (\frac{1}{2^n}+\frac{x_n}{2^{n+1}},\frac{1}{2^{n-1}}]
    \end{cases}\]
    for \(x \in (\frac{1}{2^n},\frac{1}{2^{n-1}}]\). Then for every \(n \in \N\) we have that \(f_n\) is an increasing homeomorphism from \((\frac{1}{2^n},\frac{1}{2^{n-1}}]\) onto \((\frac{1}{2^n},\frac{1}{2^{n-1}}]\) such that \(f_n(\frac{1}{2^n}+\frac{x_n}{2^{n+1}}) = \frac{1}{2^n}+\frac{y_n}{2^{n+1}}\). Now let 
    \[f(x) = 
    \begin{cases}
        f_n(x), \quad &x \in (\frac{1}{2^n},\frac{1}{2^{n-1}}], \ n \in \N , \\
        0, \quad &x = 0
    \end{cases}\]
    for \(x \in [0,1]\). Then \(f \in \mathcal{H}^+([0,1])\) and \(f(F((x_n))) = F((y_n))\). We will show that \(f \in \Lip([0,1])\). We have that \(f\) is piecewise linear. From the fact that \((x_n) \ell_\infty^\times(y_n)\), we get that there are upper and lower bounds on the slopes of the linear pieces of \(f\), and thus \(f \in \Lip([0,1])\).

    Now suppose that there is \(f \in \Lip([0,1])\) such that \(f(F((x_n)_{n \in \N})) = F((y_n)_{n \in \N})\). Since \(f\) preserves the order we have \(f(\frac{1}{2^n}+\frac{x_n}{2^{n+1}}) = \frac{1}{2^n}+\frac{y_n}{2^{n+1}}\) and \(f(\frac{1}{2^n}) = \frac{1}{2^n}\) for every \(n \in \N\). Since \(f\) is bi-Lipschitz there are \(C_1,C_2 \in (0,\infty)\) such that 
    \[C_1\frac{x_n}{2^{n+1}} \leq \frac{y_n}{2^{n+1}} \leq C_2 \frac{x_n}{2^{n+1}}.\] This implies \((x_n)\ell_\infty^\times(y_n)\).

 Together, we get that \(F\) is a reduction. 

    We will use the first part for the second reduction. For \((x_n) \in (0,1)^\N\) let \(H((x_n))\) be the unique increasing homeomorphism satisfying the following
    \begin{enumerate}
        \item \(H((x_n))(0) = 0\), \(H((x_n))(1) = 1\),
        \item \(H((x_n))(\frac{1}{2^n}) = \frac{1}{2^n}\), \(H((x_n))(\frac{1}{2^n}+\frac{1}{2^{n+1}}) = \frac{1}{2^n}+\frac{x_n}{2^{n+1}}\) for every \(n \in \N\),
        \item \(H((x_n))\) is linear between \(\frac{1}{2^{n}}\) and \(\frac{1}{2^n}+\frac{1}{2^{n+1}}\) for every \(n \in \N\),
        \item \(H((x_n))\) is linear between \(\frac{1}{2^n}+\frac{1}{2^{n+1}}\) and \(\frac{1}{2^{n-1}}\) for every \(n \in \N\).
    \end{enumerate}
    The map $H$ is clearly Borel. It remains to show that it is a reduction.
    Note that for every \((x_n) \in (0,1)^\N\) we have 
    \[H((x_n))(\{0\} \cup \{\frac{1}{2^n}, \frac{1}{2^n}+\frac{1}{2^{n+1}}; n \in \N\}) = F((x_n)).\]
    
     Let \((x_n),(y_n) \in (0,1)^\N\). Suppose there is \(f \in \Lip([0,1])\) such that \(f \circ H((x_n)) = H((y_n))\). Then \(f(F((x_n))) = F((y_n))\) and by the previous part we have \((x_n) \ell_\infty^\times (y_n)\).

    On the other hand, suppose that \((x_n) \ell_\infty^\times (y_n)\). Again, using the previous part, there is \(f \in \Lip([0,1])\) linear between the points of \(F((x_n))\) such that \(f(F((x_n))) = F((y_n))\). Since \(f\) is piecewise linear between the points of \(F((x_n))\) and composition of piecewise linear mappings is again piecewise linear, we get \(f \circ H((x_n)) = H((y_n))\). 
    
\end{proof}
\begin{Corollary}
\label{Corollary: Hypespace and shift for Lipschitz are ell infinity}
    We have
    \[E^H_{\Lip([0,1])} \Bsim \ell_\infty\] 
    and 
    \[E^{S}_{\Lip([0,1])} \Bsim  \ell_\infty.\]
\end{Corollary}
\begin{proof}
    This is a consequence of \Cref{Corollary: hyperspace and conjugation ar below ell_infty}, \Cref{Proposition: shift is below ell_infty}, \Cref{Theorem: Hyperspace and shift for lipschitz is above ell infty} and \Cref{Lemma: multiplicative ell_infty}.
\end{proof}
Rosendal {\cite[Theorem 24]{Cofinal_families_of_Borel_equivalence_relations_and_quasiorders}} has shown that the bi-Lipschitz classification of all compact metric spaces is Borel bireducible with \(\ell_\infty\). It is surprising that it is enough to consider the bi-Lipschitz classification of closed subsets of the interval.

It is not surprising that the equivalence relation \(E^C_{\Lip([0,1])}\) is of the same complexity. In fact, we will use the same method as in \Cref{Theorem: Hyperspace and shift for lipschitz is above ell infty}. However, this time we will use the special closed sets as fixed points of Lipschitz homeomorphisms. Conjugating elements of \(\Lip([0,1])\) is, in general, not easy. First, we show that if the considered functions behave well, they can be conjugated without any problems. This will then be used to show the reduction for \(E^C_{\Lip([0,1])}\).
\begin{Lemma}
\label{Lemma: conjugating piecewise linear lip maps}
    Let \(0 \leq a_1 < b_1 \leq 1\) and \(0 \leq a_2 < b_2 \leq 1\). Let 
    \[f_i(t) = 
    \begin{cases}
        2t-a_i, \quad &t \in [a_i,a_i + \frac{b_i-a_i}{3}] \\
        \frac{1}{2}t+\frac{1}{2}b_i, \quad &t \in [a_i + \frac{b_i-a_i}{3},b_i].        
    \end{cases}\]
    Then there is a homeomorphism \(\varphi: [a_1,b_1] \to [a_2,b_2]\) such that \(\Lip(\varphi) = \frac{b_2-a_2}{b_1-a_1}\), \(\Lip(\varphi^{-1}) = \frac{b_1-a_1}{b_2-a_2}\) and \(\varphi \circ f_1 = f_2 \circ \varphi.\)
\end{Lemma}
\begin{proof}
    Let \[\varphi(t) = (t -a_1)\cdot \frac{b_2-a_2}{b_1-a_1}+a_2\] for \(t \in [a_1,b_1]\). It can be easily checked that \(\varphi\) satisfies the required properties.
\end{proof}
\begin{Theorem}
\label{Theorem: conjugation for Lip is above l infinity}
    We have
    \[\ell^\times_\infty \Bleq E^C_{\Lip([0,1])}.\]
\end{Theorem}
\begin{proof}
    For \((x_n) \in (0,1)^\N\) define \(F((x_n)) \in \Lip([0,1])\) as follows:
    \[F((x_n))(t) = \begin{cases}
        0, \quad  &t = 0, \\
        2t-\frac{1}{2^n}, \quad &t \in [\frac{1}{2^n},\frac{1}{2^n}+\frac{x_n}{3\cdot2^{n+1}}], \ n \in \N, \\
        \frac{1}{2}t+\frac{1}{2}(\frac{1}{2^n}+\frac{x_n}{2^{n+1}}), \quad &t \in [\frac{1}{2^n}+\frac{x_n}{3\cdot2^{n+1}},\frac{1}{2^n}+\frac{x_n}{2^{n+1}}], \ n \in \N, \\
        2t-(\frac{1}{2^n}+\frac{x_n}{2^{n+1}}) , \quad &t \in [\frac{1}{2^n}+ \frac{x_n}{2^{n+1}},\frac{1}{2^n}+ \frac{x_n}{2^{n+1}}+\frac{1}{3}(\frac{1}{2^n}-\frac{x_n}{2^{n+1}})], \ n \in \N \\
        \frac{1}{2}t+\frac{1}{2}\frac{1}{2^{n-1}} , \quad & t \in [\frac{1}{2^n}+ \frac{x_n}{2^{n+1}}+\frac{1}{3}(\frac{1}{2^n}-\frac{x_n}{2^{n+1}}), \frac{1}{2^{n-1}}], \ n \in \N.        
    \end{cases}\]
    It can be easily shown that the mapping \(F\) is Borel and that \(F((x_n)) \in \Lip([0,1])\). Note that for every \(n \in \N\) on the intervals 
    \[\left[\frac{1}{2^n},\frac{1}{2^n}+\frac{x_n}{2^{n+1}}\right]\]
    and 
    \[\left[\frac{1}{2^n}+\frac{x_n}{2^{n+1}}, \frac{1}{2^{n-1}}\right]\]
    the function \(F((x_n))\) is defined exactly as the functions in \Cref{Lemma: conjugating piecewise linear lip maps}. We can also see that 
    \[\Fix(F((x_n))) = \{0,1\} \cup \{\frac{1}{2^n},\frac{1}{2^n}+\frac{x_n}{2^{n+1}}; n \in \N\}.\]
    We will show that \(F\) is a reduction.

    Let \((x_n),(y_n) \in (0,1)^\N\). Suppose there is \(\varphi \in \Lip([0,1])\) such that \(\varphi \circ  F((x_n)) = F((y_n))\circ \varphi\). Since \(\varphi\) has to map the set of fixed points of \(F((x_n))\) onto the set of fixed points of \(F((y_n))\), we can show in the same way as in the proof of \Cref{Theorem: Hyperspace and shift for lipschitz is above ell infty} that \((x_n) \ell_\infty^{\times}(y_n)\). Now suppose that \((x_n) \ell_\infty^\times (y_n)\). By \Cref{Lemma: conjugating piecewise linear lip maps} for every \(n \in \N\) we can find homeomorphisms 
    \[\varphi_1^n: \left[\frac{1}{2^n},\frac{1}{2^n}+\frac{x_n}{2^{n+1}}\right] \to \left[\frac{1}{2^n},\frac{1}{2^n}+\frac{y_n}{2^{n+1}}\right]\] and 
    \[\varphi^n_2: \left[\frac{1}{2^n}+\frac{x_n}{2^{n+1}},\frac{1}{2^{n-1}}\right] \to \left[\frac{1}{2^n}+\frac{y_n}{2^{n+1}},\frac{1}{2^{n-1}}\right]\]
    with the following properties.
    \begin{enumerate}
        \item For every \(n \in \N\) we have \(\Lip(\varphi^n_1) = \frac{y_n}{x_n}\), \(\Lip((\varphi^n_1)^{-1}) = \frac{x_n}{y_n}\).
        \item For every \(n \in \N\) we have \(\Lip(\varphi^n_2) = \frac{\frac{1}{2^n}-\frac{y_n}{2^{n+1}}}{\frac{1}{2^n}-\frac{x_n}{2^{n+1}}} = \frac{2-y_n}{2-x_n} < 2\), \(\Lip((\varphi^n_2)^{-1}) = \frac{2-x_n}{2-y_n} < 2\). 
        \item For every \(n \in \N\) we have \(\varphi^n_1 \circ F((x_n))|_{[\frac{1}{2^n},\frac{1}{2^n}+\frac{x_n}{2^{n+1}}]} = F((y_n))|_{[\frac{1}{2^n},\frac{1}{2^n}+\frac{y_n}{2^{n+1}}]}\circ \varphi^n_1\).
        \item For every \(n \in \N\) we have \(\varphi^n_2 \circ F((x_n))|_{[\frac{1}{2^n}+\frac{x_n}{2^{n+1}},\frac{1}{2^{n-1}}]} = F((y_n))|_{[\frac{1}{2^n}+\frac{y_n}{2^{n+1}},\frac{1}{2^{n-1}}]}\circ \varphi^n_2\).
    \end{enumerate}
    Let \(\varphi: [0,1] \to [0,1]\) be defined as follows:
    \[\varphi(t) = \begin{cases}
        0, \quad & t = 0, \\
        \varphi^n_1(t) , \quad & t \in [\frac{1}{2^n},\frac{1}{2^n}+\frac{x_n}{2^{n+1}}], \ n \in \N, \\
        \varphi^n_2(t), \quad & t \in [\frac{1}{2^n}+\frac{x_n}{2^{n+1}},\frac{1}{2^{n-1}}], \ n \in \N.
    \end{cases}\]
 Then \(\varphi \in \mathcal{H}^+([0,1])\), and by the properties of \(\varphi^n_1\) and \(\varphi^n_2\) we have \(\varphi \circ F((x_n)) = F((y_n)) \circ \varphi\). Since there is \(K > 0\) such that for every \(n \in \N\) we have 
    \[\frac{1}{K} \leq \left|\frac{x_n}{y_n} \right|\leq K\]
    we get \(\varphi \in \Lip([0,1])\).
\end{proof}
\begin{Corollary}
\label{Corollary: conjugatio for Lip is l infinity}
    We have 
    \[E^C_{\Lip([0,1])} \Bsim \ell_\infty.\]
\end{Corollary}
\begin{proof}
    This follows from \Cref{Proposition: shift is below ell_infty}, \Cref{Theorem: conjugation for Lip is above l infinity} and \Cref{Lemma: multiplicative ell_infty}.
\end{proof}

As a consequence of the shown reductions, we can recover an interesting structural property of the group of bi-Lipschitz homeomorphisms. This was already shown in \cite{Subgroups_of_Homeo_without_polish_topology} but using very different techniques.
\begin{Corollary}
    The group \(\Lip([0,1])\) is not Polishable.
\end{Corollary}
\begin{proof}
    Since \(\ell_\infty^\times \Bleq E^{H}_{\Lip([0,1])}\) and \(\ell_\infty^\times \Bsim \ell_\infty\) we get \(\ell_\infty \Bleq E^H_{\Lip([0,1])}\). From \Cref{Theorem: ell infty is not group action} we have that \(\ell_\infty\) is not Borel reducible to any orbit equivalence relation induced by a Borel action of a Polish group. Thus, \(\Lip([0,1])\) does not have a Polish group topology generating the same Borel structure as the supremum metric.
\end{proof}

\subsection{Diffeomorphisms}

When working with diffeomorphisms in the context of the considered actions, there is an issue that is not present in the case of Lipschitz functions and absolutely continuous functions, and that is the issue of extension. Although a theory of extension of differentiable functions \cite{Differentibal_functions_defined_on_closed_sets_I, Analytic_extensions_of_differnetiable_functions_defined_in_closed_sets} exists, using it in a Borel way is not an easy task. To avoid this issue, we will often work at the level of derivatives. Thus, instead of creating homeomorphisms on a smaller set and then extending them to the whole space, which is exactly what one does when considering, for instance, the complexity of conjugation on the whole homeomorphism group, we define a continuous function and then obtain the diffeomorphism by integration. This creates some other problems, such as ensuring that the obtained homeomorphism works for the equivalence relation.  

We show that when considering the subgroup of diffeomorphisms, the complexity of the induced equivalence relations is incomparable with \(E_{S_\infty}\). Thus, we need to show that the 
equivalence relations are not classifiable by countable structures and that \(E_{S_\infty}\) cannot be reduced to the equivalence relations. To show non-classifiability by countable structures, we again use a multiplicative version of a well known equivalence relation. This is again natural, since derivatives are defined as quotients. To show the other direction we use some structural properties of the group \(\Dif([0,1])\).
\begin{Definition}
    Let \(c_0^\times\) be the equivalence relation on \((0,1)^\N\) defined by 
    \[(x_n) c_0^\times (y_n) \iff \lim_{n \to \infty}\left(\frac{x_n}{y_n}\right) = 1\]
\end{Definition}
Similarly to \(\ell_\infty^\times\), we find that the complexity of \(c_0^\times\) is essentially the same as the complexity of \(c_0\).
\begin{Remark*}
    It can be easily shown that \(\R^\N / c_0 \Bsim (-\infty,0)^\N/c_0\).
\end{Remark*}
\begin{Lemma}
\label{Lemma: E_c_1 is not classifiable by countable structures} 
    We have 
    \[c_0^\times \Bsim c_0\]
    and thus \(c_0^\times\) is not classifiable by countable structures.
\end{Lemma}
\begin{proof}
    By the preceding remark, it is enough to show \(c_0^\times \Bsim (-\infty,0)^\N/c_0\). The function \((0,1)^\N \ni (x_n) \mapsto (\log(x_n)) \in (-\infty,0)^\N\) is a Borel isomorphism of \(c_0^\times\) and \((-\infty,0)^\N/c_0\).
\end{proof}

To show that \(c_0^\times \Bleq E^H_{\Dif([0,1])}\) we use a very similar strategy as for the case of Lipschitz functions. As already mentioned, the fact that we work with diffeomorphisms causes some issues. We will solve this by first defining a piecewise linear homeomorphism that satisfies what we need and then changing the derivative in a controlled way to make it continuous, while keeping the properties important for the equivalence. Doing this is very easy geometrically; see \Cref{fig:graph}.
\begin{figure}[t]
    \centering
    \includegraphics[width=0.7\linewidth]{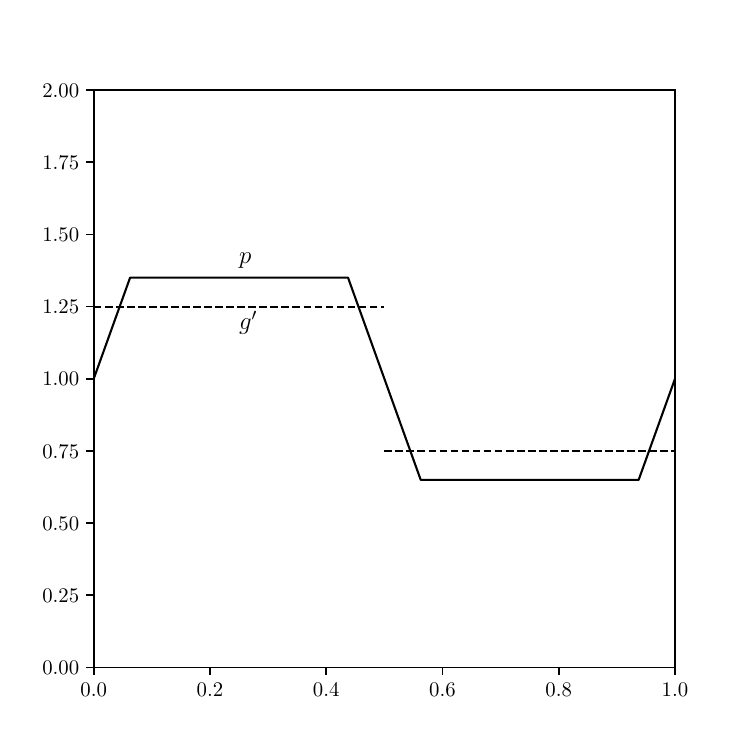}
    \caption{A derivative \(g'\) of a piecewise linear mapping \(g\) and a controlled perturbation \(p\)  to make it continuous while having the same integral over the interval.} 
    \label{fig:graph}
\end{figure}
\begin{Lemma}
\label{Lemma: aproximating picewise linear functions by difeo}
    Let \(0 \leq a < b \leq 1\) and \(x,y \in (a,b)\). Then there exists \(f \in \Dif([a,b])\) such that \(f(x) = y\) and \(f'_+(a) = f'(x) = f'_-(b) = 1\) and for every \(t \in [a,b]\) we have 
    \[|f'(t)-1| \leq 2 \max\left\{\left|1-\frac{y-a}{x-a}\right|,\left|1-\frac{b-y}{b-x}\right|\right\}.\]
\end{Lemma}
\begin{proof}
    Let \(g \in \mathcal{H}^+([a,b])\) be such that \(g(x) = y\) and \(g\) is linear on \((a,x)\) and on \((x,b)\). Then we have that \(g'(t) = \frac{y-a}{x-a}\) for \(t \in (a,x)\) and \(g'(t) = \frac{b-y}{b-x}\) for \(t \in (x,b)\). It is easy to see that we can find a continuous function \(p: [a,b] \to (0,\infty)\) such that 
    \begin{align*}
        &p(a) = 1, \ p(b) = 1, \ p(x) = 1, \\
        &\sup_{t \in (a,b) \setminus \{x\}}\abs{p(t) - g'(t)} \leq \max\left\{\left|1-\frac{y-a}{x-a}\right|,\left|1-\frac{b-y}{b-x}\right|\right\}, \\
        &a+ \int_{a}^x p(t) = y, \text{ and } \ a+\int_{a}^b p(t) = b .
    \end{align*}
    We can then define \[f(t) = a+\int_a^tp(s)\] to get the desired diffeomorphism.
\end{proof}
We can now show that, similarly to the case of Lipschitz functions, we have \(c_0^\times \Bleq E^H_{\Dif([0,1])}\). The idea of the proof is the same as the proof of \Cref{Theorem: Hyperspace and shift for lipschitz is above ell infty}, but for the reasons already mentioned, it is more technical.
\begin{Theorem}
\label{Theorem: Dif classification on interval contains c_0}
    We have
    \[c^\times_0 \Bleq E^H_{\Dif([0,1])}.\]
\end{Theorem}
\begin{proof}
    Let \(F: (0,1)^\N \to \mathcal{F}([0,1])\) be defined as follows: \[F((x_n)) = \{0\} \cup \{\frac{1}{2^n}; n \in \N\} \cup \{\frac{1}{2^n}+\frac{x_n}{2^{n+1}}; n \in \N\}.\] This is clearly a Borel map. We will show that this is a reduction. To this end, let \((x_n),(y_n) \in (0,1)^\N\).
    
 Suppose there is \(f \in \Dif([0,1])\) such that \(f(F((x_n))) = F((y_n))\). Since \(f \in \mathcal{H}^+([0,1])\) it follows that for every \(n \in \N\) we have
    \begin{equation}
    \label{equation: fixed points}
        f(\frac{1}{2^n}) = \frac{1}{2^n}
    \end{equation}
    and 
    \begin{equation}
    \label{equation: non fixed}
        f(\frac{1}{2^n}+\frac{x_n}{2^{n+1}}) = \frac{1}{2^n}+\frac{y_n}{2^{n+1}}.
    \end{equation}
    From \eqref{equation: fixed points} we get \(f_+'(0) = 1\) and from \eqref{equation: non fixed} and the continuity of \(f'\) we have \[1 = f'_+(0) = \lim_{n \to \infty}\frac{f(\frac{1}{2^n}+\frac{x_n}{2^{n+1}})-f(\frac{1}{2^n})
    }{\frac{1}{2^n}+\frac{x_n}{2^{n+1}}-\frac{1}{2^n}} = \lim_{n \to \infty}\frac{y_n}{x_n}\] and thus \((x_n) c^\times_0(y_n)\).
    
    Now, suppose that \((x_n)c^\times_0(y_n)\). By \Cref{Lemma: aproximating picewise linear functions by difeo} for every \(n \in \N\) we can find \(f_n \in \Dif([\frac{1}{2^n},\frac{1}{2^{n-1}}])\) such that the following holds: 
    \begin{align*}
        &f_n(\frac{1}{2^n}+\frac{x_n}{2^{n+1}}) = \frac{1}{2^n}+\frac{y_n}{2^{n+1}} \\
        &(f_n)'_+(\frac{1}{2^n}) = (f_n)'_-(\frac{1}{2^{n-1}}) = (f_n)'(\frac{1}{2^n}+\frac{x_n}{2^{n+1}})=1 \\
        &\sup_{t \in (\frac{1}{2^n},\frac{1}{2^{n-1}})}\abs{(f_n)'(t)-1} \leq 2\max\left\{\left|1-\frac{y_n}{x_n}\right|,\left|1-\frac{2-y_n}{2-x_n}\right|\right\}.
    \end{align*}
    Since \(\lim_{n \to \infty}\frac{y_n}{x_n} = 1\) we also have \(\lim_{n \to \infty}\frac{2-y_n}{2-x_n}=1\) and so 
    \[\lim_{n \to \infty}\sup_{t \in [\frac{1}{2^n},\frac{1}{2^{n-1}}]}\abs{f_n'(t) -1} = 0.\] Thus, if we define
    \[f(x) = 
    \begin{cases}
        f_n(x), \quad &x\in (\frac{1}{2^n},\frac{1}{2^{n-1}}] \\
        0, \quad & x = 0
    \end{cases}\]
    for \(x \in [0,1]\) we have \(f\in \Dif([0,1])\) and \(f ( F((x_n))) = F((y_n))\).
\end{proof}
\begin{Corollary}
    The orbit equivalence relation \(E_{\Dif([0,1])}^H\) is not classifiable by countable structures.
\end{Corollary}
\begin{proof}
    This follows from \Cref{Lemma: E_c_1 is not classifiable by countable structures} and \Cref{Theorem: Dif classification on interval contains c_0}.
\end{proof}
It turns out that the equivalence \(c_0^\times\) can also be reduced to the conjugation action of \(\Dif([0,1])\). We will employ a similar strategy as in \Cref{Theorem: Dif classification on interval contains c_0} however, we need to be more careful in the construction. The main reason is that conjugating diffeomorphisms by diffeomorphisms is a very complicated task. To ensure that the derivative of the conjugating homeomorphisms exists and is continuous, the conjugated diffeomorphisms are required to satisfy some relations on the level of derivatives, which are not easy to describe. For a survey on conjugating interval diffeomorphisms see \cite{Conjugating_diffeo_survey}.
\begin{Theorem}
\label{Theorem: Conjugation of Dif is above c_0}
    We have 
    \[c_0^\times \Bleq E_{\Dif([0,1])}^C.\]
\end{Theorem}
\begin{proof}
    We will construct a sequence of functions \((f_n)\), and for every sequence \((x_n) \in (0,1)^\N\), a sequence of functions \(\varphi^n_{(x_n)}\) defined on consecutive intervals such that if we glue together the functions \(\varphi^n_{(x_n)}\circ f_n \circ (\varphi^n_{(x_n)})^{-1}\), we get a diffeomorphism. Furthermore, for two different sequences \((x_n),(y_n) \in (0,1)^\N\), the glued functions \(\varphi^n_{(x_n)}\circ f_n \circ (\varphi^n_{(x_n)})^{-1}\) and \(\varphi^n_{(y_n)}\circ f_n \circ (\varphi^n_{(y_n)})^{-1}\) will be conjugated if and only if \((x_n) c^\times_0 (y_n)\). This will be done by using fixed points similarly as in \Cref{Theorem: Dif classification on interval contains c_0}. 

    For \((x_n)\in (0,1)^\N\) and \(n \in \N\) let \(\psi^n_{(x_n)}\) be a positive continuous function defined on \([\frac{1}{2^n},\frac{1}{2^{n-1}}]\) in the following way:
    \begin{enumerate}
        \item \(\psi^n_{(x_n)}(\frac{1}{2^n}) = 1\), \(\psi^n_{(x_n)}(\frac{1}{2^{n-1}}) = 1\),
        \item for \(t \in [\frac{1}{2^n}+\frac{1}{2^{n+2}},\frac{1}{2^n}+\frac{2}{2^{n+2}}]\) we have \(\psi^n_{(x_n)}(t) = 1\),
        \item \(\psi^n_{(x_n)}(\frac{1}{2^n}+\frac{3}{2^{n+2}}) = 2-x_n\),
        \item \(\psi^n_{(x_n)}\) is linear between \(\frac{1}{2^n}+\frac{2}{2^{n+2}}\) and \(\frac{1}{2^n}+\frac{3}{2^{n+2}}\),
        \item \(\psi^n_{(x_n)}\) is linear between \(\frac{1}{2^n}+\frac{3}{2^{n+2}}\) and \(\frac{1}{2^{n-1}}\),
        \item \[\int_{\frac{1}{2^n}}^{\frac{1}{2^n}+\frac{1}{2^{n+2}}}\psi^n_{(x_n)}(t) \der t = \frac{x_n}{2^{n+2}}.\]
    \end{enumerate}
    A simple calculation shows that for \((x_n) \in (0,1)^\N\) and \(n \in \N\)  we have 
    \[\int_{\frac{1}{2^n}}^{\frac{1}{2^{n-1}}}\psi^n_{(x_n)}(t) \der t = \frac{1}{2^n}.\]
    Note that for any \((x_n) \in (0,1)^\N\) and any \(n \in \N\)
    \[1 < 2-x_n<2.\]
    By this, we have that for a fixed \(n \in \N\), the functions
    \[\psi^n_{(x_n)}|_{[\frac{1}{2^n}+\frac{1}{2^{n+2}},\frac{1}{2^{n-1}}]}, \ (x_n) \in (0,1)^\N\]
    are uniformly equicontinuous. Thus, for every \(n \in \N\) we can find \(f_n \in \Dif([\frac{1}{2^n},\frac{1}{2^{n-1}}])\) such that 
    \begin{enumerate}
        \item \((f_n)'_+(\frac{1}{2^n}) = (f_n)'_-(\frac{1}{2^{n-1}}) = 1\),
        \item \(f_n(t) = t  \iff t \in [\frac{1}{2^n},\frac{1}{2^n}+\frac{1}{2^{n+2}}] \cup \{\frac{1}{2^{n-1}}\}\),
        \item \(f_n(t) \geq t\) for \(t\in [\frac{1}{2^n},\frac{1}{2^{n-1}}]\),
        \item \(|f'_n(t) -1| < \frac{1}{2^n}\) for \(t\in (\frac{1}{2^n},\frac{1}{2^{n-1}})\),
        \item for every \((x_n) \in (0,1)^\N\) and  \(t \in [\frac{1}{2^n},\frac{1}{2^{n-1}}]\)
        \[\left|\frac{\psi_{(x_n)}^n(f_n(t))}{\psi^n_{(x_n)}(t)}-1\right| < \frac{1}{2^n}.\]
    \end{enumerate}
    Satisfying the first four items can be done easily. From the equicontinuity of the functions
    \[\psi^n_{(x_n)}|_{[\frac{1}{2^n}+\frac{1}{2^{n+2}},\frac{1}{2^{n-1}}]}, \ (x_n) \in (0,1)^\N,\]
    we can find \(f_n\) close enough to the identity such that \(5.\) is satisfied as well.

    For \((x_n) \in (0,1)^\N\), \(n \in \N\), and \(t \in [\frac{1}{2^n},\frac{1}{2^{n-1}}]\) define 
    \[\varphi^n_{(x_n)}(t) = \frac{1}{2^n}+\int_\frac{1}{2^n}^{t}\psi^n_{(x_n)}(s) \der s.\]
    It follows that \(\varphi^n_{(x_n)} \in \Dif([\frac{1}{2^n},\frac{1}{2^{n-1}}])\). For every \((x_n)\in (0,1)^\N\) and any \(n \in \N\), we also have \(\varphi^n_{(x_n)}(\frac{1}{2^n}+\frac{1}{2^{n+2}}) = \frac{1}{2^n}+\frac{x_n}{2^{n+2}}\). Now, for \((x_n) \in (0,1)^\N\), \(n \in \N\), and \(t \in [\frac{1}{2^n},\frac{1}{2^{n-1}}]\) let
    \[g^n_{(x_n)}(t) =\varphi^n_{(x_n)} \circ f_n \circ (\varphi^n_{(x_n)})^{-1}(t)\]
    then \(g^n_{(x_n)} \in \Dif([\frac{1}{2^n},\frac{1}{2^{n-1}}])\). Define
    \[g_{(x_n)}(t) = 
    \begin{cases}
        g^n_{(x_n)}(t), \quad &t \in [\frac{1}{2^n},\frac{1}{2^{n-1}}], \ n \in \N, \\
        0, \quad &t = 0.
    \end{cases}\]
    It is clear that for every \((x_n) \in (0,1)^\N\) we have \(g_{(x_n)} \in \mathcal{H}^+([0,1])\). We need to show that \(g_{(x_n)}\) and \((g_{(x_n)})^{-1}\) have continuous derivatives. For this, it is enough to show that \(g_{(x_n)}\) has a continuous non-zero derivative, but for this, it is enough to show that the derivative is continuous at \(0\) from the right.
    For \(n \in \N\) and \(t \in [\frac{1}{2^n},\frac{1}{2^{n-1}}]\) we have
    \[g_{(x_n)}'(t) = \frac{(\varphi^n_{(x_n)})'(f_n((\varphi^n_{(x_n)})^{-1}(t)))}{(\varphi^n_{(x_n)})'((\varphi^n_{(x_n)})^{-1}(t))}f_n'((\varphi^n_{(x_n)})^{-1}(t)).\]
    Using the properties of \(f_n\), we have 
    \begin{align*}
        |g_{(x_n)}'(t)-1| &\leq \left| \left(\frac{(\varphi^n_{(x_n)})'(f_n((\varphi^n_{(x_n)})^{-1}(t)))}{(\varphi^n_{(x_n)})'((\varphi^n_{(x_n)})^{-1}(t))}-1 \right) f_n'((\varphi^n_{(x_n)})^{-1}(t))\right| + \\  & \quad+ \Bigg|f_n'((\varphi^n_{(x_n)})^{-1}(t))-1\Bigg| \\
        & \leq \left| \left(\frac{(\varphi^n_{(x_n)})'(f_n((\varphi^n_{(x_n)})^{-1}(t)))}{(\varphi^n_{(x_n)})'((\varphi^n_{(x_n)})^{-1}(t))}-1 \right)\left( f_n'((\varphi^n_{(x_n)})^{-1}(t))-1\right)\right| + \\ & \quad+\left|\frac{(\varphi^n_{(x_n)})'(f_n((\varphi^n_{(x_n)})^{-1}(t)))}{(\varphi^n_{(x_n)})'((\varphi^n_{(x_n)})^{-1}(t))}-1 \right|  + \frac{1}{2^n} \\
        & \leq \frac{1}{2^n}\frac{1}{2^n}+\frac{1}{2^n}+\frac{1}{2^n} \leq \frac{3}{2^n}
    \end{align*}
    Thus, we have \((g_{(x_n)})'_+(0) = 1 \), and \(g_{(x_n)}'\) is continuous.

    We will show that the mapping \((0,1)^\N \ni (x_n) \mapsto g_{(x_n)}\) is a reduction. The fact that this mapping is Borel can be shown by standard methods, and the proof will be omitted.

    Let \((x_n), (y_n) \in (0,1)^\N\). Suppose there is \(\varphi \in \Dif([0,1])\) such that \(g_{(y_n)} = \varphi \circ g_{(x_n)}\circ  \varphi^{-1}\). Since we have 
    \[g_{(x_n)}(t) = t \iff t \in \bigcup_{n \in \N}\left[\frac{1}{2^n},\frac{1}{2^n}+\frac{x_n}{2^{n+2}}\right] \cup \{0,1\},\]
    it must be true that 
    \begin{equation}
    \label{equation: 3}
     \varphi\left(\frac{1}{2^n}\right) = \frac{1}{2^n}, \ n \in \N
    \end{equation}
    and 
    \begin{equation}
    \label{equation: 4}
    \varphi(\frac{1}{2^n}+\frac{x_n}{2^{n+2}}) = \frac{1}{2^n}+\frac{y_n}{2^{n+2}}, \ n \in \N
    \end{equation}
    Using the same method as in the proof of \Cref{Theorem: Dif classification on interval contains c_0} we get \(\varphi'_+(0) = 1\) from \eqref{equation: 3} and \[\frac{y_n}{x_n} \to 1\] from \eqref{equation: 3} and \eqref{equation: 4}. Thus, \((x_n) c_0^\times (y_n)\).
    
    On the other hand, suppose \((x_n) c^\times_0 (y_n)\). For every \(n \in \N\) we have 
    \[g^n_{(y_n)} = (\varphi^n_{(x_n)} \circ (\varphi_{(y_n)}^n)^{-1})^{-1} \circ g^n_{(x_n)} \circ (\varphi^n_{(x_n)} \circ (\varphi_{(y_n)}^n)^{-1}).\]
    For every \(n \in \N\) let \(\eta_n: [\frac{1}{2^n},\frac{1}{2^n}+\frac{y_n}{2^{n+2}}] \to [\frac{1}{2^n},\frac{1}{2^n}+\frac{x_n}{2^{n+2}}]\) be a \(C^1\) homeomorphism such that 
    \((\eta_n)'(t) > 0\) for every \(t \in [\frac{1}{2^n},\frac{1}{2^n}+\frac{y_n}{2^{n+2}}]\),
    \[(\eta_n)_+'\left(\frac{1}{2^n}\right) = (\eta_n)_-'\left(\frac{1}{2^n}+\frac{y_n}{2^{n+2}}\right) = 1\]
    and
    \[|(\eta_n)'(t)-1|\leq 2\left|1-\frac{x_n}{y_n}\right|.\]
    We will show that the mapping \(\varphi\), defined by 
    \[\varphi(t) = 
    \begin{cases}
        \eta_n(t), \quad &t \in [\frac{1}{2^n},\frac{1}{2^n}+\frac{y_n}{2^{n+2}}], \ n \in \N, \\
        (\varphi^n_{(x_n)} \circ ((\varphi_{(y_n)}^n)^{-1}))(t), \quad &t \in [\frac{1}{2^n}+\frac{y_n}{2^{n+2}},\frac{1}{2^{n-1}}], \ n \in \N, \\
        0, &t = 0,
    \end{cases}\]
    is in \(\Dif([0,1])\) and conjugates \(g_{(x_n)}\) to \(g_{(y_n)}\). The fact that 
    \[g_{(y_n)} = \varphi^{-1} \circ g_{(x_n)} \circ \varphi\]
    is clear.    

    To show that \(\varphi \in \Dif([0,1])\), it is again enough to show that the derivative of \(\varphi\) is continuous at \(0\) from the right. Note that for \(t \in [\frac{1}{2^n}+\frac{y_n}{2^{n+2}},\frac{1}{2^{n-1}}]\) we have
    \[\varphi'(t) = \frac{(\varphi^n_{(x_n)})'((\varphi^n_{(y_n)})^{-1}(t))}{(\varphi^n_{(y_n)})'((\varphi^n_{(y_n)})^{-1}(t))}\]
    and also 
    \[(\varphi^n_{(y_n)})^{-1}\left(\left[\frac{1}{2^n}+\frac{y_n}{2^{n+2}},\frac{1}{2^{n-1}}\right]\right) = \left[\frac{1}{2^n}+\frac{1}{2^{n+2}},\frac{1}{2^{n-1}}\right] \]
    We will show 
    \[ \lim_{n \to \infty}\max\left\{\sup_{t \in [\frac{1}{2^n}+\frac{1}{2^n+2},\frac{1}{2^{n-1}}]}\left|\frac{(\varphi^n_{(x_n)})'(t)}{(\varphi^n_{(y_n)})'(t)}-1\right|, \sup_{t \in [\frac{1}{2^n},\frac{1}{2^n}+\frac{y_n}{2^{n+2}}]}|\varphi'(t)-1|\right\}= 0.\]

    Let \(1 > \varepsilon > 0\). We can find \(n_0 \in \N\) such that for any \(n > n_0\)
    \[\left|\frac{x_n}{y_n}-1\right| < \varepsilon/2.\]
    Then, for \(n > n_0\) we can estimate.
    For \(t \in [\frac{1}{2^n}+\frac{1}{2^{n+2}},\frac{1}{2^{n}}+\frac{2}{2^{n+2}}]\), we have 
    \begin{align*}
        \frac{(\varphi^n_{(x_n)})'(t)}{(\varphi^n_{(y_n)})'(t)}=1.
    \end{align*}
    For \(t \in [\frac{1}{2^n}+\frac{2}{2^{n+2}},\frac{1}{2^{n}}+\frac{3}{2^{n+2}}]\) we have
    \begin{align*}
        \frac{(\varphi^n_{(x_n)})'(t)}{(\varphi^n_{(y_n)})'(t)} &= \frac{(1-x_n)2^{n+2}(t-\frac{1}{2^n}-\frac{1}{2^{n+1}})+1}{(1-y_n)2^{n+2}(t-\frac{1}{2^n}-\frac{1}{2^{n+1}})+1} \\
        &\leq \frac{(1-(1-\varepsilon)y_n)2^{n+2}(t-\frac{1}{2^n}-\frac{1}{2^{n+1}})+1}{(1-y_n)2^{n+2}(t-\frac{1}{2^n}-\frac{1}{2^{n+1}})+1} \\
        &= 1+\frac{\varepsilon y_n2^{n+2}(t-\frac{1}{2^n}-\frac{1}{2^{n+1}})}{(1-y_n)2^{n+2}(t-\frac{1}{2^n}-\frac{1}{2^{n+1}})+1} \\
        &\leq 1+\varepsilon
    \end{align*}
    Similarly, we can get 
    \[\frac{(\varphi^n_{(x_n)})'(t)}{(\varphi^n_{(y_n)})'(t)}  > 1-\varepsilon\]
    and also
    \[\left|\frac{(\varphi^n_{(x_n)})'(t)}{(\varphi^n_{(y_n)})'(t)}-1\right| < \varepsilon\]
    for \(t \in [\frac{1}{2^n}+\frac{3}{2^{n+2}},\frac{1}{2^{n-1}}]\).
    We also have 
    \[|\varphi'(t) -1| < \varepsilon\] 
    for \(t \in [\frac{1}{2^n},\frac{1}{2^n}+\frac{y_n}{2^{n+2}}]\) from the properties  of \(\eta_n\).
    Together we get
    \[\max\left\{\sup_{t \in [\frac{1}{2^n}+\frac{1}{2^{n+2}},\frac{1}{2^{n-1}}]}\left|\frac{(\varphi^n_{(x_n)})'(t)}{(\varphi^n_{(y_n)})'(t)}-1\right|, \sup_{t \in [\frac{1}{2^n},\frac{1}{2^n}+\frac{y_n}{2^{n+2}}]}|\varphi'(t)-1|\right\} < \varepsilon\]
    for \(n > n_0\).
\end{proof}

To show that the equivalence relation \(E_{S_\infty}\) cannot be reduced to the hyperspace or conjugation equivalence relation of the diffeomorphism group we will use a structural property of the diffeomorphism group. Recall that a Polish group is called \textit{CLI} if it admits complete left invariant metric. Hjorth showed that \(E_{S_\infty}\) cannot be reduced to an equivalence relation induced by a continuous action of any CLI group.
\begin{Theorem}[{\cite[Theorem 0.5]{Orbit_cardinals}}]
\label{Theorem: CLI groups are not above E_S_infty}
    Let \(G\) be a CLI group acting continuously on a Polish space \(X\). Then 
    \[E_{S_\infty} \not \Bleq E^X_{G}.\]
\end{Theorem}
\begin{Proposition}
\label{Proposition: Dif is CLI}
    The group \(\Dif([0,1])\) is CLI.
\end{Proposition}
\begin{proof}
    Consider the following metric
    \[\rho(f,g) = \|\log(f')-\log(g')\|_\infty\] 
    for \(f,g \in \Dif([0,1])\). We will show that this metric is a complete right invariant metric generating the same topology as the metric \(\ddif\). It can be easily seen that \(\rho\) is a right invariant metric.

    To show completeness we can use completeness of \(\mathcal{C}([0,1])\) and the fact that the set
    \[F = \left\{u \in \mathcal{C}([0,1]); \int_0^1e^{u(t)} \der t =1\right\}\]
    is closed in \(\mathcal{C}([0,1])\).

    To argue that \(\rho\) induces the same topology as \(\ddif\) it is enough to observe that the mapping \(g \mapsto \log(g')\) is a homeomorphism between \(\Dif([0,1])\) and \(F\).

    To get a complete left invariant metric we can consider the metric \(\hat{\rho}(f,g) = \rho(f^{-1},g^{-1})\) for \(f,g \in \Dif([0,1])\).
\end{proof}
\begin{Corollary}
\label{Corollary: hyperspace and conjugation for diff are not above E_S_infty}
    We have 
    \[E_{S_\infty} \not\Bleq E^{H}_{\Dif([0,1])}\]
    and
    \[E_{S_\infty} \not\Bleq E^{C}_{\Dif([0,1])}.\]
\end{Corollary}
\begin{proof}
    This follows from \Cref{Proposition: Dif is CLI} and \Cref{Theorem: CLI groups are not above E_S_infty}
\end{proof}

It turns out that the hyperspace action of the group \(\Dif([0,1])\) is different from the hyperspace action of the full group \(\mathcal{H}^+([0,1])\), not only from the point of view of classification complexity, but it also has very different dynamical properties.
\begin{Proposition}
    Every orbit of the hyperspace action of \(\Dif([0,1])\) is meager.
\end{Proposition}
\begin{proof}
    We have \(\Dif([0,1]) \subset \Lip([0,1])\). Thus, it is enough to show that every orbit of the hyperspace action of \(\Lip([0,1])\) is meager. It can be shown by standard use of Baire category that 
    \[\mathcal{C} =\{F; \lambda(F) = 0, \ F  \approx 2^\N,\  F \cap\{0,1\} = \emptyset\}\]
    is comeager in \(\mathcal{F}([0,1])\); see {\cite[Section 8.B]{Classical_descriptive_set_theory}} for details. We also have
    \[\Lip([0,1]) \cdot  \mathcal{C} = \mathcal{C}.\]
    Thus, it is enough to show that for every \(F \in \mathcal{C}\) the orbit of \(F\) is meager. We show that for every \(F \in \mathcal{C}\) and \(K > 0\), the set 
    \[\{h(F); h \in \Lip_K([0,1])\}\]
    is compact and has empty interior in \(\mathcal{F}([0,1])\). To this end, pick \(F \in \mathcal{C}\) and \(K > 0\). The fact that 
    \[\{h(F); h \in \Lip_K([0,1])\}\]
    is compact follows by \Cref{Lemma: Lipcshitz functions are compact}. To show that it has empty interior, it is enough to note that every non-empty open set \(\mathcal{U} \subset \mathcal{F}([0,1])\) in the Vietoris topology contains a closed set of positive measure and thus \(\mathcal{U} \setminus\mathcal{C} \neq \emptyset\).
\end{proof}

It turns out that even the shift action of the diffeomorphisms is not classifiable by countable structures. For this, we use the notion of turbulence and \Cref{Thoerem: Hjorth turbulenc theorem}. The same has been shown by Solecki in \cite{Polish_group_topologies} for the shift action of the group \(\AC([0,1])\). In fact, our proof is very similar to the one given by Solecki.

\begin{Proposition}
\label{Proposition: Shift of Diff is turbulent}
    The shift action of \(\Dif([0,1])\) is turbulent.
\end{Proposition}
\begin{proof}
    The group \(\Dif([0,1])\) is dense in \(\mathcal{H}^+([0,1])\). To see this note that piecewise linear homeomorphisms are dense in \(\mathcal{H}^+([0,1])\) and each piecewise linear homeomorphism can be approximated by elements of \(\Dif([0,1])\). Since \(\Dif([0,1])\) is dense but not equal to \(\mathcal{H}^+([0,1])\) it is meager by \Cref{Theorem: proper subsets of Polish groups are meager}. Thus, every orbit of the shift action is dense and meager. 

    It remains to show that every local orbit is somewhere dense. Since we are dealing with the shift action, it is enough to show that every local orbit of the identity is somewhere dense. Let \(U \subset \mathcal{H}^+([0,1])\) and \(V \subset \Dif([0,1])\) be open such that \(\id \in U\) and \(\id \in V\). We can find \(\delta > 0\) such that \(B_\infty(\id, \delta) \subset U\). Let \(h \in B_\infty(\id,\delta)\). Let \(\varepsilon >0\) be arbitrary. Since \(\Dif([0,1])\) is dense in \(\mathcal{H}^+([0,1])\), we can find \(f \in \Dif([0,1])\) such that \(\|f-h\|_\infty < \varepsilon\) and \(\|f - \id\|_\infty < \delta\).

    For \(t \in [0,1]\) define
    \[h_t = (1-t)\id + t f.\]
    Then the map \((s,t) \mapsto h_s \circ h_t^{-1}\) is continuous from \([0,1]^2\) into \(\Dif([0,1])\). Since \([0,1]^2\) is compact, the map is uniformly continuous. Choose a fine partition \(0 = t_0<t_1< \dots <t_n = 1\) so that 
    \[h_{t_{i+1}}\circ h_{t_i}^{-1} \in V.\]

    Then we have \(h_{t_i} \in \Dif([0,1])\), since the sum of two differentiable functions with positive derivatives is differentiable and has a positive derivative. We also have \(\|h_{t_i}-\id\|_\infty < \delta\), thus \(h_{t_i} \in U\).

    Now, let \(g_{i} = h_{t_{i+1}} \circ h_{t_{i}}^{-1}\) for \(i = 0, \dots, n-1\). Observe that for every \(i = 0, \dots, n-1\) we have \(g_i \in V \) and \(g_i \circ \dots \circ g_0 \circ \id = h_{t_{i+1}}\); in particular, for \(i = n-1\) we get \(g_{n-1} \circ \dots \circ g_0 \circ \id = h_{t_n} = f\).

    This implies that the \(U \text{-} V\)-local orbit of \(\id\) is dense in \(B_\infty(\id,\delta)\).
\end{proof}

\subsection{Absolutely continuous homeomorphisms}

So far, there has been a very nice correspondence between the space of the derivatives of the considered homeomorphisms and the equivalence relations that are reducible to the various induced orbit equivalence relations (\(\ell_\infty\) for bounded derivatives, \(c_0\) for continuous derivatives). Thus, it would appear that a natural candidate for reductions for the equivalence relations induced by \(\AC([0,1])\) is some type of multiplicative \(\ell_1\) equivalence, since absolutely continuous functions have derivatives in \(L_1([0,1])\).  However, this is not the case, and it turns out that the natural equivalence relation to use in reductions is the bi-absolute continuity of measure. It was shown by Solecki \cite{Polish_group_topologies} that the bi-absolute continuity of measures is actually almost the same as the shift equivalence relation of \(\AC([0,1])\), see \cite{Polish_group_topologies} for details. 

The reason why absolutely continuous homeomorphisms are related to the bi-absolute continuity of measures is the following classical result. We denote by \(\lambda\) the Lebesgue measure on \([0,1]\).
\begin{Theorem}[{\cite[Theorem 7.18]{Real_and_complex_analysis}}]
\label{Theorem: Homeo is AC iff Luzin N property}    
Let \(h \in \mathcal{H}^+([0,1])\). Then \(h \in \AC([0,1])\) if and only if 
\[\lambda(E) = 0 \iff \lambda(h(E)) = 0\]
for every Borel (compact) E.
\end{Theorem}
Solecki \cite{Polish_group_topologies} has shown that the bi-absolute continuity of measures is not classifiable by countable structures.
\begin{Definition}
    Let \(X\) be a compact space and \(P\) be the set of Borel probability measures on \(X\). Let \(\equiv\) be the equivalence relation on \(P\) defined by
    \[\mu \equiv \nu \iff \mu \ll \nu, \ \nu \ll \mu\]
    for \(\mu, \nu \in P\).
\end{Definition}
\begin{Theorem}[{\cite[Lemma 2.7]{Polish_group_topologies}}]
\label{Theorem: measure equivalence is not classifiable}
    Let \(N\) be the set of non-atomic, locally positive probability Borel measures on the Cantor space. Then \(\equiv|_N\) is not classifiable by countable structures.
\end{Theorem}

This means that our strategy for the proofs has to change significantly. Embedding sequences in a special way is not enough to obstruct bi-absolute continuity by \Cref{Theorem: Homeo is AC iff Luzin N property}. Instead, we will embed Cantor sets with a given measure into the interval in a way that one can be mapped to the other by an element of \(\AC([0,1])\) if and only if the measures on the Cantor sets are bi-absolutely continuous. 

\begin{Theorem}
\label{Theorem: AC eqivalence on interval is not classifiable} 
    Let \(N\) be the set of non-atomic,  locally positive, probability Borel measures on the Cantor space. Then we have
    \[\equiv|_N \Bleq E^H_{\AC([0,1])}.\]
    Thus, \( E^H_{\AC([0,1])}\) is not classifiable by countable structures.
\end{Theorem}
\begin{proof}

    We will use \Cref{Theorem: measure equivalence is not classifiable}. Let \(M\) denote the set of measures in \(N\) divided by \(2\). Note that dividing by \(2\) does not change bi-absolute continuity of measures.

    For every measure \( \mu \in M\), we construct a closed subset of the interval \(F(\mu)\) that will contain a Cantor set \(C^\mu \approx 2^\N\) such that \(\lambda|_{C^\mu} = \mu\) and whenever we have \(\mu, \nu \in M\) and \(h \in \mathcal{H}^+([0,1])\) satisfying \(h(F(\mu)) = F(\nu)\) then \(h(C^\mu) = C^\nu\) and \(h|_{C^\mu} = \id(2^\N)\).
    
    Let \(1/4 >\varepsilon > 0\) and \(\mu \in M\). We will use induction to construct closed intervals \(A^{\mu}_{s} \subset [0,1]\) for \(s \in 2^{<\N}\) such that the following holds
    \begin{itemize}
        \item if \(m \leq n\), \(s \in 2^{m}\) and \(r \in 2^{n}\) such that \(s \preceq r \) then \(A^{\mu}_r\subset A^{\mu}_s\) otherwise \(A^{\mu}_r \cap A^{\mu}_s = \emptyset\),
        \item \(\lambda(A^{\mu}_s) = \mu(s \times 2^\N)+\varepsilon/2^{2n}\) for \(s \in 2^n\).
    \end{itemize}

    Let \(A^{\mu}_\emptyset\) be the closed interval centered around \(1/2\) such that \(\lambda(A^{\mu}_{\emptyset}) = 1/2 + \varepsilon \). Suppose we constructed the sets \(A^{\mu}_s\) for every \(s \in 2^{\leq n}\) for some \(n \in \N\). Let \(s \in 2^{n}\) suppose \(A^{\mu}_s = [a_s,b_s]\) for some \(a_s,b_s \in [0,1]\) such that \(a_s < b_s \) and \(|a_s-b_s|= \mu(s \times 2^\N) + \varepsilon/2^{2n}\). We can find \(c_s,d_s \in [0,1]\) such that \(a_s < c_s < d_s < b_s\), \(|a_s-c_s| = \mu(s^\smallfrown 0 \times 2^\N) + \varepsilon/2^{2n+2}\) and \(|d_s-b_s| = \mu(s^\smallfrown 1 \times 2^\N) + \varepsilon/2^{2n+2}\). We can then let \(A^{\mu}_{s^\smallfrown 0} = [a_s,c_s]\) and \(A^{\mu}_{s^\smallfrown 1} = [d_s,b_s]\). It is clear from the construction that the sets \(A^{\mu}_{s^\smallfrown 0}, A^{\mu}_{s^\smallfrown 1}\) have the desired properties.

    Let \(\psi: 2^{<\N} \to \N\) be a bijection. For \(s \in 2^{<\N}\) let \(B^{\mu}_s = A^{\mu}_{s^\smallfrown 0}\cup A^{\mu}_{s^\smallfrown 1} \cup Q^{\mu}_s\) where \(Q^\mu_s\) is a set of \(\psi(s)\) many points evenly spaced in the gaps between the sets \(A^{\mu}_{s^\smallfrown 0}\) and \(A^{\mu}_{s^\smallfrown 1}\). Now, let \(F^{\mu}_n = \bigcup_{s \in 2^n}B^{\mu}_s \cup \bigcup_{s \in 2^{<n}} Q^{\mu}_s\) for \(n \in \N\). We have \(F^\mu_{n+1} \subset F^\mu_n\) for every \(n \in \N\). From the construction, it is clear that the mapping \(\hat{F}: M \to \mathcal{F}([0,1])^\N\) defined by \(\hat{F}(\mu) = (F^{\mu}_n)_{n \in \N}\) is Borel. Then it follows that the mapping \(F: M \to \mathcal{F}([0,1])\) such that \(\mu \mapsto \bigcap_{n \in \N}F^{\mu}_n\) is Borel.
    
    Let \(C^{\mu}\) be the Cantor set defined by the intervals \(A_s^{\mu}\) for \(s \in 2^{<\N}\), let \(i^{\mu}\) be the canonical homeomorphism given by the construction between \(2^\N\) and \(C^\mu\) and let \(Q^\mu = \bigcup_{s \in 2^{<\N}}Q^\mu_s\). Then \(F(\mu) = C^\mu \cup Q^{\mu}\) and \(\closure{Q^{\mu}} = F(\mu)\). We also have \(\mu  = \lambda|_{C^\mu}i^{\mu}\), this follows immediately from the construction, since the measures \(\mu\) and \(\lambda|_{C^\mu}i^{\mu}\) are equal on the basic clopen sets.

    We also have the additional property that if we have \(\mu, \nu \in M\) and a homeomorphism \(h \in \mathcal{H}^+([0,1])\) such that \(h(F(\mu)) = F(\nu)\) then we have \(h(C^\mu)= C^\nu\) and \((i^\nu)^{-1} \circ h|_{C^\mu} \circ i^{\mu} = \id(2^\N)\). This follows since the set \(Q^\mu\), respectively \(Q^\nu\), is a maximal set of isolated points in \(F(\mu)\), respectively \(F(\nu)\) and the number of points between any clopen basic intervals is distinct.

    Now we have all the ingredients needed to show that \(F\) is a reduction. Let \(\mu,\nu \in M\). First, suppose \(\mu \equiv \nu \). We can find a homeomorphism \(h \in \mathcal{H}^+([0,1])\) such that \(h(F(\mu)) = F(\nu)\) and such that for every Borel set \(E \subset [0,1] \setminus F(\mu)\) we have \(\lambda(E) = 0 \iff \lambda(h(E)) = 0\). This is easy to do, as we can map any maximal open interval from the set \([0,1] \setminus F(\mu)\) onto the corresponding interval in \([0,1] \setminus F(\nu)\) by a linear mapping. This correspondence is clear from the construction and the fact that null-sets are preserved follows by the linearity of the mapping. This mapping can be uniquely extended onto the whole interval and it will have the desired properties. The fact that \(\mu \equiv \nu\) is equivalent to saying that \(\id(2^\N)\) maps \(\mu\)-null sets onto \(\nu\)-null sets. Since \((i^\nu)^{-1} \circ h|_{C^\mu} \circ i^{\mu} = \id(2^\N)\), \(\nu  = \lambda|_{C^\nu}i^{\nu}\) and \(\mu  = \lambda|_{C^\mu}i^{\mu}\) we get that \( h \in \AC([0,1])\). On the other hand, suppose there is \(h \in \AC([0,1])\) such that \(h(F(\mu))=F(\nu)\). Again, by using \((i^\nu)^{-1} \circ h|_{C^\mu} \circ i^{\mu} = \id(2^\N)\), \(\nu  = \lambda|_{C^\nu}i^{\nu}\), \(\mu  = \lambda|_{C^\mu}i^{\mu}\) we get that \(\mu \equiv \nu\) if and only if \(\id(2^\N)\) maps \(\mu\)-null sets onto \(\nu\)-null sets. Since \(h \in \AC([0,1])\), we get \(\mu \equiv \nu\). Thus, \(F\) is a reduction and we are done.
\end{proof}

We are able to show not only that \(E^H_{\AC([0,1])}\) is not classifiable by countable structures, but also that it is above \(E_{S_\infty}\). The main reason for this is that extending bi-absolutely continuous homeomorphisms can be easily done by using \Cref{Theorem: Homeo is AC iff Luzin N property}. If \(F,G\) are closed subsets of the interval and \(f \in \mathcal{H}^+([0,1])\) maps \(F\) to \(G\) in such a way that it maps \(\lambda\)-null subsets of \(F\) onto \(\lambda\)-null subsets of \(G\) we can change \(f\) to be linear on the gaps of the sets to obtain a bi-absolutely continuous homeomorphism.

\begin{Proposition}
\label{Proposition: hyperspace on AC contains S_infty}
    We have 
    \[E_{S_\infty} \Bleq E^H_{\AC([0,1])}.\]
\end{Proposition}
\begin{proof}
    We will use \Cref{Theorem: order isomorphisms on Q is E S infinity}. Let \(C \subset(0,1)\) be a Cantor set of Lebesgue measure zero. Let \(\mathcal{U}\) be the collection of all maximal open intervals in the complement of \(C\) having both endpoints in \(C\). For every \(U \in \mathcal{U}\) let \(I_U\) be a closed interval such that \(I_U \subset U\). Then \(C \cup \bigcup_{U \in \mathcal{U}}I_U\) is compact. We also have that \(\mathcal{U}\) with the natural induced order is order isomorphic to \(\Q\). Fix \(\varphi: \Q \to \mathcal{U}\) an order isomorphism. For \(A \subset \Q \) let \(F(A) = C \cup \bigcup_{U \in \varphi(A)} I_U\). This map is easily seen to be Borel. We will show that this is a reduction.

    To this end, let \(A,B \subset \Q\). First, suppose \(A\) and \(B\) are such that there is an order isomorphism \(\psi\) of \(\Q\) such that \(\psi(A) = B\). Let \(\phi\) be the order isomorphism of \(\mathcal{U}\) induced by \(\psi\). For \(U \in \varphi(A)\) let \(h_U: U \to \phi(U)\) be a piecewise linear increasing homeomorphism such that \(h_U(I_U) = I_{\phi(U)}\) and for \(U \in \mathcal{U} \setminus \varphi(A)\) let \(h_U: U \to \phi(U)\) be a linear increasing homeomorphism. For \(x \in [0,1]\setminus C\) let 
    \[\hat{h}(x) = \begin{cases}
        x, \quad &x \notin \bigcup \mathcal{U} \\
        h_U(x) , \quad &x \in U, \ U \in \mathcal{U}.
    \end{cases}\]
    It can be easily seen that \(\hat{h}\) has a unique extension to a homeomorphism \(h \in \mathcal{H}^+([0,1])\).
    By the \(\sigma\)-additivity of measures and the fact that linear functions are bi-absolutely continuous, we get that \(h\) is an increasing homeomorphism that satisfies the assumptions of \Cref{Theorem: Homeo is AC iff Luzin N property}. Thus \(h \in \AC([0,1])\) and by the construction \(h(F(A)) = F(B)\). 

    Now, suppose there is \(h \in \AC([0,1])\) such that \(h(F(A)) = F(B)\). Since \(h\) is a homeomorphism, we have \(\{h(I_U); U \in \varphi(A)\} = \{I_V; V \in \varphi(B)\}\), since the intervals \(I_U\) for \(U \in \varphi(A)\) respectively \(I_V\) for \(V \in \varphi(B)\) are exactly the nondegenerate components of \(F(A)\) respectively \(F(B)\). We also have \(h(C) = C\). From this it follows that \(h(U) \in \varphi(B)\) for any \(U \in \varphi(A)\). This means that \(h\) induces an order isomorphism of \(\mathcal{U}\) that maps \(\varphi(A)\) onto \(\varphi(B)\). This in turn induces an order isomorphism of \(\Q\) that maps \(A\) onto \(B\).
\end{proof}
\begin{Corollary} 
\label{Corollary: AC equivalence on interval is above E_{S_ifinity}}
    We have \[E_{S_\infty} <_B E^H_{\AC([0,1])}.\]
\end{Corollary}
\begin{proof}
    This follows by \Cref{Theorem: AC eqivalence on interval is not classifiable}, \Cref{Theorem: measure equivalence is not classifiable} and \Cref{Proposition: hyperspace on AC contains S_infty}.
\end{proof}
Since extending bi-absolutely continuous homeomorphisms is an easy task, we can say a lot also about the conjugation action. We start with a lemma saying that the conjugation of elements of \(\AC([0,1])\) by elements of \(\AC([0,1])\) is in some way similar to the conjugation on the group \(\mathcal{H}^+([0,1])\). The conjugation essentially depends only on the set of fixed points and the order structure of the interval. 
\begin{Lemma}
\label{Lemma: conjugating above diagonal AC homeo}
    Let \(f,g \in \AC([0,1])\) be such that \(f,g \geq \id\) and both \(f\) and \(g\) do not have any fixed points in \((0,1)\). Then there is \(\varphi \in \AC([0,1])\) such that \(\varphi \circ f = g \circ \varphi\).
\end{Lemma}
\begin{proof}
    Pick \(a,b \in (0,1)\). Let \(\psi\) be an increasing bi-absolutely continuous homeomorphism mapping  the interval \([a,f(a))\) onto the interval \([b, g(b))\). Define \(\varphi\) as follows
    \[\varphi(x) = \begin{cases}
        x ,&\quad x \in \{0,1\}, \\
        (g^{n}\circ\psi\circ f^{-n})(x) ,&\quad x \in [f^{n}(a), f^{n+1}(a)), n \in \Z.
    \end{cases}\]
    It is easy to see that \(\varphi\) is an increasing bi-absolutely continuous homeomorphism conjugating \(f\) to \(g\). 
\end{proof}
\begin{Remark}
\label{Remark: conjugating AC homeo on other intervals}
    Let \(0\leq a<b \leq1\) and \(0 \leq c<d \leq 1\), \(f \in \mathcal{H}^+([a,b])\) and \(g \in \mathcal{H}^+([c,d])\) such that \(f\) and \(g\) are bi-absolutely continuous, \(f\) does not have any fixed points in \((a,b)\), \(g\) does not have any fixed points in \((c,d)\), \(f \geq \id\) and \(g \geq \id\). Let \(\psi_1\) be an increasing bi-absolutely continuous homeomorphism from \([0,1]\) onto \([a,b]\) and \(\psi_2\) be an increasing bi-absolutely continuous homeomorphism from \([0,1]\) onto \([c,d]\). By \Cref{Lemma: conjugating above diagonal AC homeo} we can find \(\varphi \in \AC([0,1])\) such that 
    \[\varphi \circ (\psi_1^{-1}\circ f \circ \psi_1) \circ \varphi^{-1} = (\psi_2^{-1}\circ g \circ \psi_2).\]
    Then \(\psi_2 \circ \varphi \circ \psi_{1}^{-1}\) conjugates \(f\) to \(g\).

    We can, using the same method, prove \Cref{Lemma: conjugating above diagonal AC homeo} for \(f,g \leq \id\).
\end{Remark}

We show that the conjugation and hyperspace equivalence relations for the group \(\AC([0,1])\) are of the same complexity. Note that by \Cref{Theorem: equivalences of homemorphisms ar E S infinity} the same behavior occurs for the group \(\mathcal{H}^+([0,1])\) as well. In fact, we use a similar method to show bireducibility for the space \(\AC([0,1])\) as is used for the group \(\mathcal{H}^+([0,1])\). 
\begin{Theorem}
    We have 
    \[E^H_{\AC([0,1])} \Bleq E^C_{\AC([0,1])}.\]
\end{Theorem}
\begin{proof}
    Fix elements \(\varphi_1, \varphi_2 \in \AC([0,1])\) such that \(\varphi_1\) and \(\varphi_2\) are fixed point free in \((0,1)\) and \(\varphi_2 \leq \id\) and \(\varphi_1 \geq \id\). For any open interval \(I \subset [0,1]\) denote \(a_I\) affine homeomorphism from \((0,1)\) to \(I\). Let \(\rho(t) = t/3 + 1/3\) for \(t \in [0,1]\). For \(F \in \mathcal{F}([0,1])\) 
    let 
    \[\tilde{F} = \begin{cases}
        \rho(F), \quad &F \cap \{0,1\} = \emptyset, \\
        \rho(F) \cup [0,\frac{1}{3}], \quad &F \cap \{0,1\} = \{0\}, \\
        \rho(F) \cup [\frac{2}{3},1], \quad &F \cap \{0,1\} = \{1\}, \\ 
        \rho(F) \cup [0,\frac{1}{3}] \cup [\frac{2}{3},1], \quad &F\cap \{0,1\} = \{0,1\}.
    \end{cases}\]
    Let \(\Phi(F) \in \AC([0,1])\) be defined in the following way
    \[
    \Phi(F)(x) = 
    \begin{cases}
        x, \quad &x \in \{0,1\} \cup (\tilde{F} \cap [\frac{1}{3}, \frac{2}{3}]) \\
        a_A \circ \varphi_1 \circ a_A^{-1}(x), \quad &x \in A, A \subset (0,1) \setminus \tilde{F}, \ A \text{ maximal open interval } \\
        a_{(0,1/3)} \circ \varphi_2 \circ a_{(0,1/3)}^{-1}(x), \quad &x \in [0,\frac{1}{3}] \text{ if \([0,\frac{1}{3}] \subset \tilde{F}\)} \\
        a_{(2/3,1)} \circ \varphi_2 \circ a_{(2/3,1)}^{-1}(x), \quad &x \in [\frac{2}{3},1] \text{ if \([\frac{2}{3},1] \subset \tilde{F}\)}.
    \end{cases}
    \]
    We will show that \(\Phi\) is a reduction. Note that for every \(F \in \mathcal{F}([0,1])\) we have
    \[\Fix(\Phi(F)) \cap \left[\frac{1}{3},\frac{2}{3}\right] = \rho(F).\]

    It is easy to see that \(\Phi\) is a Borel mapping into \(\AC([0,1])\) with the supremum metric. But then it follows that it is also Borel as a mapping into \(\AC([0,1])\) with the metric \(\dac\).

    Let \(F_1, F_2 \in \mathcal{F}([0,1])\). Suppose there is \(\psi \in \AC([0,1])\) such that \(\psi \circ \Phi(F_1) \circ \psi^{-1} = \Phi(F_2)\). Then \(\psi\) maps the set of fixed points of \(\Phi(F_1)\) onto the set of fixed points of \(\Phi(F_2)\) and from the properties of \(\Phi(F_1)\) and \(\Phi(F_2)\) we get
    \[\psi\left(\Fix(\Phi(F_1)) \cap \left[\frac{1}{3},\frac{2}{3}\right]\right) = \Fix(\Phi(F_2)) \cap \left[\frac{1}{3},\frac{2}{3}\right].\]
    Thus, \(\rho^{-1} \circ \psi \circ \rho\) is a homeomorphism between \(F_1\) and \(F_2\)
    that maps \(\lambda\)-null sets of \(F_1\) to \(\lambda\)-null sets of \(F_2\). We also have \(0 \in F_1 \iff 0 \in F_2\) and \(1 \in F_1 \iff 1 \in F_2\). Thus, this homeomorphism can be extended in an affine way over every component of \([0,1] \setminus F_1\) to get \(h \in \AC([0,1])\) with \(h(F_1) = F_2\).

    On the other hand, suppose there is \(\psi \in \AC([0,1])\) such that \(\psi(F_1) = F_2\). For \(t \in [0,1]\) let
    \[\eta(t) = \begin{cases}
        \rho\circ\psi \circ \rho^{-1}(t), \quad& t \in [\frac{1}{3},\frac{2}{3}], \\
        t , \quad& t \in [0, \frac{1}{3}] \cup [\frac{2}{3},1].
    \end{cases}\]
    Then \(\eta \in \AC([0,1])\) and \(\eta\) is a homeomorphism between \(\tilde{F_1}\) and \(\tilde{F_2}\). For every maximal open interval \(A\) of \((0,1) \setminus \tilde{F_1}\) we can use \Cref{Lemma: conjugating above diagonal AC homeo} and \Cref{Remark: conjugating AC homeo on other intervals} to find an increasing homeomorphism \(\eta_A: A \to \eta(A)\) that is bi-absolutely continuous and conjugating \(\Phi(F_1)|_A\) to \(\Phi(F_2)|_{\eta(A)}\). Let \(\varphi\) be defined as follows
    \[\varphi(x) = 
    \begin{cases}
    \eta(x), &\quad x \in \tilde{F_1} \cup\{0,1\}\\
    \eta_A(x), &\quad x \in A, \ A \subset (0,1) \setminus \tilde{F_1} \text{ maximal open.}
    \end{cases}\]
    Clearly, we have \(\varphi \in \mathcal{H}^+([0,1])\) and \(\varphi\) conjugates \(\Phi(F_1)\) to \(\Phi(F_2)\). It remains to show that \(\varphi \in \AC([0,1])\). This can be shown by \Cref{Theorem: Homeo is AC iff Luzin N property}. Let \(E \subset [0,1]\) be Borel. Then we have 
    \begin{align*}
        \lambda(E) = 0& \\
         \iff &\lambda(E \cap \tilde{F_1}) = 0 \text{ and } \lambda(A \cap E) = 0\ \forall A \subset (0,1) \setminus \tilde{F_1} \ A \text{ maximal open}  \\
        \iff &\lambda(\varphi(E \cap \tilde{F_1})) = 0 \text{ and } \lambda(\varphi(A \cap E)) = 0 \ \forall A \subset (0,1) \setminus \tilde{F_1} \ A \text{ maximal open}   \\
         \iff &\lambda(\varphi(E) \cap \tilde{F_2}) = 0 \text{ and } \lambda(A \cap \varphi(E)) = 0 \ \forall A \subset (0,1) \setminus \tilde{F_2} \, A \text{ maximal open} \\
        \iff &\lambda(\varphi(E)) = 0.
    \end{align*}
    We used that \(\varphi|_{\tilde{F_1}}\) is bi-absolutely continuous and \(\varphi|_{A}\) for every \(A \subset (0,1) \setminus \tilde{F_1}\) maximal open is bi-absolutely continuous as well.
\end{proof}
\begin{Theorem}
\label{Theorem: AC conjugation is below AC hyperspace}
    We have \[E^C_{\AC([0,1])} \Bleq E^{H}_{\AC([0,1])}.\]
\end{Theorem}
\begin{proof}
    For \(f \in \AC([0,1])\) and any maximal open interval \(U \subset [0,1] \setminus \partial \operatorname{Fix}(f)\), with the boundary taken in \(\R\), we define a set \(I_U\) by the following rules:
    \begin{enumerate}
        \item if \(f|_U = \id|_U\) let \(I_U = \emptyset\),
        \item if \(f|_U > \id|_U\) let \(I_U\) be the closed subinterval of $U$ such that \(\diam(U) = 3\diam(I_U)\) and \(I_U\) is centered in \(U\),
        \item if \(f|_U < \id|_U\) let \(I_U\) be the union of two closed disjoint intervals \(I_1,I_2\subseteq U\) such that \(\diam (U) = 3\diam (I_1) = 3\diam(I_2)\), and the intervals \(I_1,I_2\) are placed so that the three complementary gaps in \(U\) are equal.
    \end{enumerate}
    For \(f \in \AC([0,1])\) let 
    \[F(f) = \partial \operatorname{Fix}(f) \cup \bigcup_{\substack{U \subset [0,1]\setminus \partial \Fix(f) \\ U \text{ maximal open interval}}} I_U.\]
    It can be easily shown that \(F\) is Borel and \(F(f) \in \mathcal{F}([0,1])\) for every \(f \in \AC([0,1])\). We will show that \(F\) is a reduction.
    
    To this end let \(f_1,f_2 \in \AC([0,1])\). Suppose there is \(\varphi \in \AC([0,1])\) such that \(\varphi \circ f_1 \circ \varphi^{-1} = f_2\). Then we have that \(\varphi(\Fix(f_1)) = \Fix(f_2)\). This implies that \(\varphi(\partial(\Fix(f_1))) = \partial \Fix(f_2)\) and for every maximal open interval \(U \subset [0,1] \setminus \partial \Fix(f_1)\) we have that \(\varphi(U)\) is a maximal open interval of \([0,1] \setminus \partial\Fix(f_2)\) and 
    \[\operatorname{sign}(f_1|_U-\id|_U) = \operatorname{sign}(f_2|_{\varphi(U)}-\id|_{\varphi(U)}).\] Let \(\psi\) be defined as follows:
    \[\psi(t) = 
    \begin{cases}
        \varphi(t), \quad & t \in \partial \Fix(f_1),\\
        t, \quad & t \in \{0,1\}, \\
        (t-a)\cdot\frac{\varphi(b)- \varphi(a)}{b-a}+\varphi(a), \quad &t  \in (a,b), 
        \begin{array}{c}
         (a,b) \text{ maximal open interval} \\
            \text{ of } [0,1] \setminus \partial \Fix(f_1).
        \end{array}
    \end{cases}\]
    In other words, \(\psi\) agrees with \(\varphi\) on \(\partial \Fix(f_1)\) and is extended linearly on the gaps. It easily follows by \Cref{Theorem: Homeo is AC iff Luzin N property} that \(\psi \in \AC([0,1])\). Since for any maximal open interval \(U \subset [0,1]\setminus \partial \Fix(f_1)\) the sets \(I_U\) and \(I_{\varphi(U)}\) are the same up to the position and linear stretching we get \(\psi(F(f_1)) = F(f_2)\). On the other hand, suppose there is \(\psi \in \AC([0,1])\) such that \(\psi(F(f_1)) = F(f_2)\). The points of \(F(f_1)\) respectively \(F(f_2)\) that are not contained in any interval are exactly the points of \(\partial \Fix(f_1)\) respectively \( \partial \Fix(f_2)\). Thus, we get \(\psi(\partial \Fix(f_1)) = \partial \Fix(f_2)\). For any maximal open interval \(U \subset [0,1] \setminus \partial \Fix(f_1)\) we have 
    \[\operatorname{sign}(f_1|_U-\id|_U) = \operatorname{sign}(f_2|_{\psi(U)}-\id|_{\psi(U)}),\] since 
    \[\operatorname{sign}(f_1|_U-\id|_U)\] 
    depends on the number of intervals in \(U\). For any maximal open interval \(U \subset [0,1] \setminus \partial \Fix(f_1)\) we can use \Cref{Lemma: conjugating above diagonal AC homeo} to find a bi-absolutely continuous homeomorphism \(\varphi_U: U \to \psi(U)\) such that 
    \[\varphi_U \circ f_1|_U \circ (\varphi_U)^{-1} = f_2|_{\psi(U)}.\] 
    Define \(\varphi \in \AC([0,1])\) as follows:
    \[\varphi(x) = \begin{cases}
        \psi(x), \quad &x \in \partial\Fix(f_1) \\
        x , \quad &x \in \{0,1\} \\
        \varphi_U(x), \quad &x \in U, \ U \text{ is a maximal open interval of } [0,1] \setminus \partial\Fix(f_1).
    \end{cases}\]
    It again follows by \Cref{Theorem: Homeo is AC iff Luzin N property} that \(\varphi \in \AC([0,1])\) and by the definition of \(\varphi_U\) we get \(\varphi \circ f_1 \circ \varphi^{-1} = f_2\).
\end{proof}
\begin{Corollary}
\label{Corollary: AC cojugation and hyperspace are the same}
    We have \[E^C_{\AC([0,1])} \Bsim E^{H}_{\AC([0,1])}.\]
\end{Corollary}
Although the equivalences induced by \(\AC([0,1])\) are different from the equivalences induced by the group \(\mathcal{H}^+([0,1])\), we can still find some similarities. In both the hyperspace action and the conjugation action of \(\mathcal{H}^+([0,1])\) there are comeager orbits and the same is true for the group \(\AC([0,1])\). In fact, comeager orbits for \(\AC([0,1])\) can be constructed in a very similar fashion as for the actions of the group \(\mathcal{H}^+([0,1])\). See {\cite[Theorem 5.3]{Turbulence_amalgamation_and_generic_automorphisms_of_homogeneous_structures}} for a description of the comeager conjugacy class in \(\mathcal{H}^+([0,1])\).
\begin{Proposition}
    Both the hyperspace action and the conjugation action of \(\AC([0,1])\) have a comeager orbit.
\end{Proposition}
\begin{proof}
    Since Cantor sets of zero measure not intersecting the endpoints form a comeager set in \(\mathcal{F}([0,1])\) and since any homeomorphism between Cantor sets in the interval that respects the order can be extended linearly on the gaps, we get that the following set is a comeager orbit in the hyperspace action
    \[\mathcal{O}_H=\{C \in \mathcal{F}([0,1]); \lambda(C) = 0, \ C \approx 2^\N, C \cap \{0,1\} = \emptyset\}.\]

    We will show that the following set is a comeager orbit in the conjugation action 
    \begin{align*}
        \mathcal{O}_C = \{f \in \AC([0,1]);& \Fix(f) \approx 2^\N, \ \lambda(\Fix(f)) = 0,\\
        &\text{for every \(x \in \Fix(f)\) and every \(\varepsilon > 0\)} \\
        &\text{there are \(x_1,x_2 \in B(x,\varepsilon):\) }\ f(x_1) > x_1, \ f(x_2) < x_2\}.
    \end{align*}
    The fact that the set \(\mathcal{O}_C\) forms a single orbit can be easily shown by a standard back-and-forth argument and by using \Cref{Lemma: conjugating above diagonal AC homeo}. The fact that this set is \(G_\delta\) can also be easily shown. We will only show that this set is dense.
    
    Let \(f \in \AC([0,1])\) and \(\varepsilon > 0\). By bi-absolute continuity and uniform continuity of \(f\) and by the compactness of \([0,1]\), we can find finitely many numbers \(0 = a_1 < b_1 < a_2  < b_2 < \dots < a_n < b_n = 1\) such that 
    \[[a_1,b_1) \cup (a_n,b_n] \cup \bigcup_{1<i<n} (a_i,b_i) \supset \Fix(f),\] 
    \[\int_{\bigcup_{1\leq i \leq n} (a_i,b_i)}|f'-1| < \varepsilon\]
    and  for every 
    \[x \in [a_1,b_1) \cup (a_n,b_n] \cup \bigcup_{1<i<n} (a_i,b_i)\]
    we have
    \[|f(x) - x| < \varepsilon.\]
    
    Without loss of generality, we may assume that for every \(i \in \{1, \dots, n\}\) every interval \([a_i,b_i]\) contains at least one fixed point in its interior. For \(i \in \{1, \dots, n\}\) let \(c_i\) be the smallest fixed point in the interior of \([a_i,b_i]\) and \(d_i\) be the largest fixed point in the interior of  \([a_i,b_i]\). Note that \(c_1 = 0\) and \(d_n = 1\). Using continuity and making the intervals smaller if necessary, we can assume that for
    \[x \in [d_1,b_1) \cup (a_n,c_n] \cup \bigcup_{1 < i < n}( (a_i,c_i] \cup [d_i,b_i))\]
    we have
    \[|f(x)-x|< \frac{\varepsilon}{n}\]
    and
    \[\sup_{1 \leq i \leq n}(|a_i-c_i|+|d_i-b_i| + |f(a_i)-f(c_i)|+|f(d_i)-f(b_i)|) < \frac{\varepsilon}{n}.\]
    For \(i \in \{2, \dots, n\}\) pick \(u_i \in [0,1]\) such that \(\max\{a_i,f(a_i)\} < u_i < c_i\) and for \(i \in \{1, \dots, n-1\}\) pick \(v_i \in [0,1]\) such that \(d_i < v_i < \min\{b_i, f(b_i)\}\). Also, let \(u_1 = 0\) and \(v_n = 1\). 
    For \(i \in \{2, \dots, n\}\) let \(\varphi^1_i: [a_i,u_i]\to [f(a_i),u_i]\) be a linear increasing homeomorphism, and for \(i \in \{1, \dots, n-1\}\) let \(\varphi^3_i: [v_i,b_i]\to [v_i,f(b_i)]\) be a linear increasing homeomorphism. For \(i \in \{1, \dots, n\}\) let \(\varphi^2_i: [u_i,v_i] \to [u_i,v_i]\) be a bi-absolutely continuous homeomorphism such that for every \(i \in \{1, \dots, n\}\) we have
    \[\|\varphi^2_i-\id\|_\infty + \int_{u_i}^{v_i}|(\varphi^2_i)'-1| < \frac{\varepsilon}{n},\]
    \[\Fix(\varphi^2_i) \approx 2^\N, \ \lambda(\Fix(\varphi^2_i)) = 0\]
    and for every \(x \in \Fix(\varphi^2_i)\) and every \(\eta > 0\) there are \(x_1,x_2 \in B(x,\eta)\) such that \( \varphi^2_i(x_1) > x_1\) and \( \varphi^2_i(x_2) < x_2\).
    For \(x \in [0,1]\) define
    \[\varphi(x) = 
    \begin{cases}
        f(x), \quad & x \notin [a_1,b_1) \cup (a_n,b_n] \cup \bigcup_{1<i<n} (a_i,b_i),\\
        \varphi^1_i(x), \quad &x \in [a_i,u_i], \ i \in \{2, \dots, n\},\\
        \varphi^2_i(x), \quad &x \in [u_i,v_i], \ i \in \{1, \dots, n\},\\
        \varphi^3_i(x), \quad &x \in [v_i,b_i], \ i \in \{1, \dots, n-1\}.\\
    \end{cases}\]
    
    We can use the \(\sigma\)-additivity of the Lebesgue measure and \Cref{Theorem: Homeo is AC iff Luzin N property} to show that \(\varphi \in \AC([0,1])\). Clearly, we also have \(\varphi \in \mathcal{O}_C\). We need to calculate \(\dac(f,\varphi)\). For
    \[x \notin [a_1,b_1) \cup (a_n,b_n] \cup \bigcup_{1<i<n} (a_i,b_i)\] 
    we have \(f(x) = \varphi(x)\).
    For \(i \in \{2, \dots, n\}\) we have 
    \begin{align*}
        \|\varphi^1_{i}-f|_{[a_i,u_i]}\|_\infty + \int_{a_i}^{u_i}|f'-(\varphi^1_i)'| &\leq \sup_{r,s \in [f(a_i),c_i]}|r-s| + f(u_i) - f(a_i) \\ &\quad+ \varphi^1_i(u_i)-\varphi^1_i(a_i) \\
        &\leq \frac{\varepsilon}{n} + 2(f(c_i)- f(a_i)) \leq \frac{3\varepsilon}{n}.
    \end{align*}
    For \(i \in \{1, \dots, n-1\}\) we have
    \begin{align*}
        \|\varphi^3_{i}-f|_{[v_i,b_i]}\|_\infty + \int_{v_i}^{b_i}|f'-(\varphi^3_i)'| &\leq \sup_{r,s \in [d_i,f(b_i)]}|r-s| + f(b_i) - f(v_i) \\ &\quad+ \varphi^3_i(b_i)-\varphi^3_i(v_i) \\
        &\leq \frac{\varepsilon}{n} + 2(f(b_i)- f(d_i)) \leq \frac{3\varepsilon}{n}.
    \end{align*}
    We also have
    \begin{align*}
        &\|(\varphi-f)|_{\bigcup_{1 \leq i \leq n}(u_i,v_i)}\|_\infty + \int_{\bigcup_{1 \leq i \leq n}(u_i,v_i)} |\varphi'-f'| \\
        &\qquad\leq \|(\varphi-\id)|_{\bigcup_{1 \leq i \leq n}(u_i,v_i)}\|_\infty + \|(f-\id)|_{\bigcup_{1 \leq i \leq n}(u_i,v_i)}\|_\infty \\
        &\qquad\quad+\sum_{i=1}^{n} \int_{u_i}^{v_i} |(\varphi^2_i)'-1| + \int_{\bigcup_{1 \leq i \leq n}(u_i,v_i)} |f'-1| \leq \varepsilon + \varepsilon + \varepsilon + \varepsilon = 4\varepsilon.
    \end{align*}
    Together, using the \(\sigma\)-additivity of the Lebesgue measure, we get \(\dac(f,\varphi) \leq 10\varepsilon\). Since \(\varepsilon\) was arbitrary, we get that \(\mathcal{O}_C\) is dense in \(\AC([0,1])\).
\end{proof}
\section[Absolutely continuous homeomorphisms on other spaces]{Absolutely continuous homeomorphisms \\ on other spaces}

In this section, we investigate groups of bi-absolutely continuous homeomorphisms on the Hilbert cube and the Cantor space. It is worth noting that the groups depend on the measures chosen on the spaces. Although there are natural choices of probability Borel measures on both the Cantor space and the Hilbert cube, they are not as canonical as the Lebesgue measure on the interval. We first establish that on both the Cantor space and the Hilbert cube the choice of the measure in fact does not matter, and if the measures have some reasonable regularity properties, we end up with isomorphic groups. To do this on the Hilbert cube we use results of \cite{Homeomorphic_measures_in_the_Hilbert_cube}, where it is established that any two locally positive, non-atomic Borel probability measures are homeomorphic. This result does not hold on the Cantor space. However, we show that any two such measures are bi-absolutely continuous up to a homeomorphism of one of them. 

We then investigate mainly the hyperspace action of these groups. We show that the induced equivalence relations demonstrate a similar behavior as in the case of the interval. This behavior is even more surprising for the Cantor space, since any homeomorphism group of the Cantor set is naturally a subgroup (not necessarily closed) of the permutation group. Nevertheless, we show that the hyperspace action is strictly more complex than any orbit equivalence relation of the permutation group can be. Note that on the Hilbert cube, the hyperspace action cannot actually increase its complexity, since \(E^H_{\mathcal{H}([0,1]^\N)}\) is already of the highest possible complexity by \Cref{Theorem: Hyperspace relation on Hilber cube is universal}. We show that the complexity of \(E^H_{\AC([0,1]^\N)}\) remains the same, that is, \(E^H_{\AC([0,1]^\N)}\) is a complete orbit equivalence relation.

First we will establish some general results on bi-absolutely continuous homeomorphisms. 
\begin{Definition}
    Let \(X\) be a compact space and \(\mu\) be a Borel measure on \(X\). We say that a homeomorphism \(h \in \mathcal{H}(X)\) is bi-absolutely continuous with respect to \(\mu\) if for every \(E \subset X\) Borel we have \[\mu(E) = 0 \iff \mu(h(E)) = 0.\]
    We denote by \(\AC(X,\mu)\) the set of all homeomorphisms of \(X\) that are bi-absolutely continuous with respect to \(\mu\).
\end{Definition}
By \Cref{Theorem: Homeo is AC iff Luzin N property} this is a possible generalization of bi-absolutely continuous homeomorphisms on the interval to other compact spaces with a measure.
We will only focus on non-atomic locally positive measures.
\begin{Definition}
    Let \(X\) be a compact space, \(\mu\) be a Borel measure on \(X\) and \(h \in \mathcal{H}(X)\). We will denote by \(\mu h\) the measure on \(X\) such that \(\mu h(E) = \mu(h(E))\) for every \(E\subset X\) Borel.
\end{Definition}
\begin{Remark*}
    Let \(X\) be a compact space and \(\mu\) be a Borel measure on \(X\). For a homeomorphism \(h \in \mathcal{H}(X)\) to be an element of \(\AC(X,\mu)\) is equivalent to 
    \[\mu h \ll \mu  \text{ and } \mu h^{-1} \ll \mu.\]
\end{Remark*}

It is easy to see that for any compact space \(X\) and any finite Borel measure \(\mu\) the set \(\AC(X,\mu)\) is a Borel subset of \(\mathcal{H}(X)\). To show this, one can use that \(h \in \AC(X,\mu)\) if and only if the measures \(\mu h \) and \(\mu\) are bi-absolutely continuous, and the fact that bi-absolute continuity of measures can be described by a Borel condition. But it turns out that more is true and the group \(\AC(X,\mu)\) is actually a Polishable subgroup of \(\mathcal{H}(X)\). Group topologies for bi-absolutely continuous homeomorphisms on compact spaces and manifolds have also been studied in \cite{Group_topologies_on_AC_homeo}. There, it has been established that the group is Polishable for any compact metrizable space \(X\). We use a very similar metric to the one used in \cite{Group_topologies_on_AC_homeo} but we use a different method for the proof. We include the proof for the sake of completeness. We actually use an approach very similar to the one used by Solecki in \cite{Polish_group_topologies} to show Polishability of the group \(\AC([0,1])\). We will use the following lemma, the proof of which is essentially included in the proof of {\cite[Lemma 2.4]{Polish_group_topologies}}.
\begin{Lemma}
\label{Lemma: convergence of inverse of derivatives}
    Let \(\mu\) be a finite Borel measure on a compact space \(X\). Let \(f_n \in \AC(X,\mu)\) be a sequence and \(f \in \AC(X,\mu)\). Suppose \(d_\infty(f_n,f) \to 0\) and \(\|\mu f_n - \mu f\| \to 0\). Then \(\|\mu f^{-1}_n - \mu f^{-1}\| \to 0\).
\end{Lemma}
\begin{Proposition}
\label{Proposition: AC homeomorphisms are Polish group}
    Let \(\mu\) be a finite Borel measure on a compact space \(X\). The metric 
    \[d_{\AC}(f,g) = d_\infty(f,g)+ \|\mu f - \mu g \|\]
    generates a Polish group topology that is stronger than the uniform convergence topology. 
\end{Proposition}
\begin{proof}
    We will use the Radon-Nikodým derivative. For any \(h \in \AC(X,\mu)\) the measures \(\mu h\) and \(\mu h^{-1}\) are absolutely continuous with respect to \(\mu\). Thus, by the Radon-Nikodým theorem {\cite[Theorem 6.10 (b)]{Real_and_complex_analysis}} there are two functions 
    \[\frac{\der \mu h}{\der \mu}, \frac{\der \mu h^{-1}}{\der \mu} \in L_1(X,\mu)\]
    such that for any Borel set \(E \subset X\)
    \[\mu h (E) = \int_E \frac{\der \mu h}{\der \mu}\der \mu \]
    and 
    \[\mu h^{-1} (E) = \int_E \frac{\der \mu h^{-1}}{\der \mu}\der \mu.\]
    Note that for \(g,h \in \AC(X,\mu)\) we have
    \[\|\mu g - \mu h \|= \int_X\left| \frac{\der \mu h}{\der \mu}- \frac{\der \mu g}{\der \mu}\right| \der \mu\]
    and thus
    \[d_{\AC}(g,h) = d_\infty(g,h)+\int_X\left|\frac{\der\mu  g}{\der \mu}-\frac{\der \mu h}{\der \mu}\right|\der\mu.\]
    Consider the following metric
    \begin{align*}
        \rho_{\AC}(g,h)  &= d_\infty(g,h)+d_\infty(g^{-1},h^{-1}) \\ 
        &\quad +\int_X\left|\frac{\der\mu  g}{\der \mu}-\frac{\der \mu h}{\der \mu}\right|\der\mu+ \int_X\left|\frac{\der \mu g^{-1}}{\der \mu}-\frac{\der \mu h^{-1}}{\der \mu}\right|\der\mu.
    \end{align*}
    First we show that \(\rho_{\AC}\) is a complete separable metric and that it generates the same topology as \(d_{\AC}\).

    It can be easily seen that \(\rho_{\AC}\) generates a stronger topology than \(d_{\AC}\). To show that the topology generated by \(d_{\AC}\) is stronger than the topology generated by \(\rho_{\AC}\) it is enough to show that if \(f_n \in \AC(X,\mu)\) is a sequence and \(f \in \AC(X,\mu)\) such that
    \[f_n \overset{d_{\AC}}{\to}f\]
    then
    \[f_n \overset{\rho_{\AC}}{\to}f.\]
    This follows by \Cref{Lemma: convergence of inverse of derivatives}.

    For completeness of the metric \(\rho_{\AC}\) it is enough to show that the set 
    \begin{align*}
        \bigg\{&(h_1,h_2,f_1,f_2) \in \mathcal{H}(X)^2 \times L_1(X, \mu)^2 ; h_1^{-1} = h_2,  \\  
        &\forall K \in \mathcal{F}(X) \int_K f_1\der \mu = \mu h_1(K) \int_K f_2\der \mu = \mu h_2(K)\bigg\}  
    \end{align*}
    is closed in the space \(\mathcal{C}(X,X)^2 \times L_1(X, \mu)^2\). To show this, it is enough to consider \(h_n,h \in \mathcal{H}(X)\) and \(f_n,f  \in L_1(X,\mu)\) such that \(h_n \overset{d_\infty}{\to} h\), \(f_n \overset{L_1}{\to} f \) and for every \(K\subset X\) compact and every \(n \in \N\) we have \[\int_K f_n =  \mu h_n (K),\]
    and show that this implies
    \[\int_Kf \der \mu= \mu h (K)\]
    for every compact \(K \subset X\). For the inverses we can make the same argument.

    Fix \(K \subset X\) compact. For every \(U\) open such that \(U \supset h(K)\) there is \(n_0 \in \N\) such that for all \(n > n_0\) we have \(h_n(K) \subset U\). This and the outer regularity of \(\mu\) imply 
    \[\int_Kf \der \mu \leq  \mu h (K).\]
    Now pick \(V \supset K\) open  and \(L \subset h(V)\) compact. There exists \(n_0 \in \N\) such that for all \(n > n_0\) we have \(L \subset h_n(V)\). This and the inner regularity of \(\mu\) imply
    \[\int_Vf \der \mu \geq  \mu h (V).\]
    Connecting the last two inequalities we get 
    \[\int_V f \der \mu \geq  \mu h (V) \geq \mu h (K) \geq \int_K f \der \mu\]
    for \(V \supset K\) open.
    This by continuity of integration implies
    \[\int_K f \der \mu = \mu h (K).\]
    Since \(K\) was arbitrary, we get 
    \[\frac{\der \mu h }{\der \mu} = f.\]

    Separability follows easily since both \(\mathcal{H}(X)\) with the supremum metric and \(L_1(X,\mu)\) are separable.

    Since the topology generated by \(\rho_{\AC}\) and \(d_{\AC}\) is Polish, to show that multiplication and inverse are continuous, it is enough to show that multiplication is separately continuous {\cite[Corollary 9.15 and Remark following it]{Classical_descriptive_set_theory}}. To establish this it is enough to show that inverse is continuous and multiplication in the first variable is continuous. But it can be easily seen that inverse is \(\rho_{\AC}\) isometry and multiplication in the first variable is \(d_{\AC}\) isometry.

    The fact that \(d_{\AC}\) generates a stronger topology than the uniform convergence is clear.
    \end{proof}
    Note that the metric \(d_{\AC}\) defined on \(\AC([0,1],\lambda) \cap \mathcal{H}^+([0,1])\) in \Cref{Proposition: AC homeomorphisms are Polish group} is actually the same metric as on \(\AC([0,1])\). Also note that the metric \(d_{\AC}\) on \(\AC(X,\mu)\) is right invariant.
    \begin{Remark}
        The fact that the metric defined in \Cref{Proposition: AC homeomorphisms are Polish group} on \(\AC(X,\mu)\) is stronger than the supremum metric means that the inclusion mapping \(i:{\AC}(X, \mu)\hookrightarrow \mathcal{H}(X)\) is continuous. Since it is injective, by \cite[Theorem 15.1]{Classical_descriptive_set_theory} we get that images of Borel sets in \(\AC(X,\mu)\) are Borel in \(\mathcal{H}(X)\). This in turn implies that open sets in \(\AC(X,\mu)\) are Borel in \(\mathcal{H}(X)\). Since the topology of \(\AC(X,\mu)\) is stronger, we get that \(\AC(X,\mu)\) has the same Borel structure as inherited from \(\mathcal{H}(X)\). Thus \(\AC(X,\mu)\) is actually a Polishable subgroup of \(\mathcal{H}(X)\).
    \end{Remark}
\subsection{Hilbert cube}
In this section, we show that the hyperspace equivalence relation of bi-absolutely continuous homeomorphisms on the Hilbert cube is a complete orbit equivalence relation. First, we show that the group and complexity of the hyperspace action do not depend on the measure. For this we use results of \cite{Homeomorphic_measures_in_the_Hilbert_cube} where it is shown that any two non-atomic, locally positive Borel probability measures on the Hilbert cube can be mapped to each other by a homeomorphism and that many of the homogeneity properties of the Hilbert cube remain valid even with respect to any non-atomic, locally positive Borel measure. We will denote by \(Q\) the Hilbert cube \([0,1]^\N\).

\begin{Theorem}[{\cite[Theorem 1]{Homeomorphic_measures_in_the_Hilbert_cube}}]
\label{Theorem: homeomorphic measures on the Hilbert cube}
    Let \(\mu,\nu\) be two non-atomic, locally positive probability Borel measures on \(Q\). Then there exists a homeomorphism \(h \in \mathcal{H}(Q)\) such that \(\mu h = \nu\).
\end{Theorem}

In the study of the properties of the Hilbert cube, an important notion is the notion of a \(Z\)-set. We will not go into details, but intuitively \(Z\)-sets are in some sense very small subsets of the Hilbert cube. For details, see {\cite[Chapter 5]{Infinite_dimensiona_topology}}. The only important fact for us is that the set \(\{0\} \times [0,1]^\N \subset [0,1]^\N\) is a \(Z\)-set in \(Q\) that is homeomorphic to the Hilbert cube. It turns out that many of the properties of \(Z\)-sets remain valid even with respect to any non-atomic, locally positive Borel measure.

\begin{Theorem}[{\cite[Theorem 4]{Homeomorphic_measures_in_the_Hilbert_cube}}]
\label{Theorem: Extension of measure preserving homeomorphims of the Hilber cube}
    Let \(\mu\) be a non-atomic, locally positive probability Borel measure on \( Q\). If \(A,B\) are two \(Z\)-sets in \(Q\) and \(h\) is a \(\mu\) preserving homeomorphism between \(A\) and \(B\) then \(h\) can be extended to a \(\mu\) preserving homeomorphism of \(Q\).
\end{Theorem}
We can use these results to show that the group \(\AC(Q,\mu)\) in fact does not depend on the choice of the measure.
\begin{Proposition}
\label{Proposition: All absolutely continuous homeomorphism groups on the Hilbert cube are isomorphic}
    Let \(\mu,\nu\) be two non-atomic, locally positive probability Borel measures on \(Q\). Then \(\AC(Q,\mu)\) and \(\AC(Q,\nu)\) are topologically isomorphic.
\end{Proposition}
\begin{proof}
    By \Cref{Theorem: homeomorphic measures on the Hilbert cube} we can find \(h \in \mathcal{H}(Q)\) such that \(\nu  = \mu h\). Then for any \(g \in \AC(Q,\mu)\), \(f \in \AC(Q,\nu)\) and \(E \subset Q\) Borel we have the following
    \[\nu(E) = 0 \iff \mu(h(E)) = 0 \iff \mu(g(h(E))) = 0  \iff \nu(h^{-1}(g(h(E)))) = 0\]
    and 
    \[\mu(E) = 0 \iff \nu(h^{-1}(E)) = 0 \iff \nu(f(h^{-1}(E))) = 0 \iff \mu(h(f(h^{-1}(E)))) = 0.\]
    Thus, the map \(g \mapsto h^{-1}\circ g \circ h \) is a topological isomorphism between \(\AC(Q,\mu)\) and \(\AC(Q,\nu)\).
\end{proof}
\Cref{Proposition: All absolutely continuous homeomorphism groups on the Hilbert cube are isomorphic} and its proof imply that the complexity of the hyperspace equivalence relation is the same for any choice of non-atomic, locally positive probability Borel measure on \(Q\). This allows us to fix one Borel measure and one group of bi-absolutely continuous homeomorphisms for our investigation. As we already mentioned, there is a natural choice of a probability Borel measure on the Hilbert cube.

\begin{Notation}
    We will denote by \(\lambda^\N\) the measure on \(Q\) that is the product of Lebesgue measures on \([0,1]\) and we will denote by \(\AC(Q)\) the group \(\AC(Q, \lambda^\N)\).
\end{Notation}

As mentioned before, it is not possible for \(E^{H}_{\AC(Q)}\) to be of strictly higher complexity than \(E^H_{\mathcal{H}(Q)}\). However, considering the group \(\AC(Q)\) does not lower the complexity.
\begin{Theorem}
\label{Theorem: AC Hyperspace on Hilbert cube is complete}
    We have 
    \[E^H_{\AC(Q)} \Bsim E_{G_\infty}.\]
\end{Theorem}
\begin{proof}
    The reduction \(E^H_{\AC(Q)} \Bleq E_{G_\infty}\) follows from the fact that \(E^H_{\AC(Q)}\) is an orbit equivalence relation induced by a Borel action of a Polish group. Thus, we only need to show \(E^H_{\AC(Q)} \Bgeq E_{G_\infty}\). We will use \Cref{Theorem: Hyperspace relation on Hilber cube is universal}.

    Let \(Z \subset Q\) be a \(Z\)-set homeomorphic to the Hilbert cube such that \(\lambda^{\N}(Z)= 0\). There exists a homeomorphism \(\tilde{F}\) between \(Q\) and \(Z\). Let \(F\) be the mapping from \(\mathcal{F}(Q)\) to \(\mathcal{F}(Z) \subset \mathcal{F}(Q)\) defined by \(G \mapsto \{\tilde{F}(x); x \in G\}\). We will show that \(F\) is the desired reduction. 

    Let \(A,B \in \mathcal{F}(Q)\). If there is \(h \in \AC(Q)\) such that \(h(F(A)) = F(B)\) then \(\tilde{F}^{-1}|_{F(B)} \circ h \circ \tilde{F}|_A\) is a homeomorphism between \(A\) and \(B\). On the other hand, suppose that there is a homeomorphism \(\tilde{g}_0\) between \(A\) and \(B\). Since \(F\) is induced by \(\tilde{F}\) there exists a homeomorphism \(g\) between \(F(A)\) and \(F(B)\). Since \(F(A)\) and \(F(B)\) are \(Z\)-sets of zero measure \(g\) is measure preserving. By \Cref{Theorem: Extension of measure preserving homeomorphims of the Hilber cube} we can extend \(g\) to a measure preserving homeomorphism \(e\). Since \(e\) is measure preserving, we have \(e \in \AC(Q)\) and thus 
    \[F(A) E^H_{\AC(Q)} F(B). \qedhere\]
\end{proof}

Since \Cref{Theorem: AC Hyperspace on Hilbert cube is complete} implies that \(\AC(Q)\) induces a complete orbit equivalence relation, a very natural question to ask is whether the group \(\AC(Q)\) is a universal Polish group. We will show that this is not the case. Note that the group \(\mathcal{H}(Q)\) is known to be a universal Polish group {\cite[Theorem 9.18]{Classical_descriptive_set_theory}}. Recall that a \textit{unitary representation} of a topological group is any homomorphism into \(U(H)\) the space of unitary operators on some nonzero Hilbert space \(H\) that is continuous with respect to the strong (equivalently, weak) operator topology. A unitary representation is said to be faithful if it is injective. For a more detailed exposition on unitary representations see \cite{A_course_in_abstract_harmonic_analysis}.
\begin{Theorem}
\label{Theorem: AC(Q) is not universal}
    The group \(\AC(Q)\) is not a universal Polish group.
\end{Theorem}
\begin{proof}
    Megrelishvili \cite{No_repre_for_Homeo} showed that there are Polish groups that do not admit a faithful unitary representation. Thus, it is enough to show that the group \(\AC(Q)\) admits a faithful unitary representation.

    For \(h \in \AC(Q)\) let \[\frac{\der \lambda^\N h}{\der \lambda^\N} \in L_1(Q,\lambda^\N)\] be the Radon-Nikodým derivative. Note that for every \(h \in \AC(Q)\) the derivative \(\frac{\der \lambda^\N h}{\der \lambda^\N}\) is positive \(\lambda^\N\) a.e. For \(\varphi \in L_2(Q,\lambda^\N)\) and \(h \in \AC(Q)\) define \(U_h \in U(L_2(Q,\lambda^\N))\) in the following way 
    \[U_h(\varphi)(x) = \left(\frac{\der \lambda^\N h}{\der \lambda^\N}(x)\right)^{1/2} \varphi(h(x))\] for \(x \in Q\). We will show that the mapping \(\pi: h \mapsto U_{h^{-1}}\) is a faithful unitary representation of \(\AC(Q)\).

    First, we show that for every \(h \in \AC(Q)\) the operator \(U_{h^{-1}}\) is a unitary operator on \(L_2(Q,\lambda^\N)\). It is straightforward to check that \(U_{h^{-1}}\) is a linear operator on \(L_2(Q,\lambda^\N)\). Let \(\varphi \in L_2(Q,\lambda^\N)\). By an easy application of the substitution rule we get
    {\allowdisplaybreaks
    \begin{align*}
        \int_Q |U_{h^{-1}}(\varphi)(x)|^2 \der \lambda^\N = \int_Q |\varphi(h^{-1}(x))|^2 \frac{\der \lambda^\N h^{-1}}{\der \lambda^\N}(x) \der \lambda^\N = \int_Q |\varphi(x)|^2 \der \lambda^\N.
    \end{align*}
    Thus, \(U_{h^{-1}}\) is a linear isometry of \(L_2(Q,\lambda^\N)\) it is also onto since \((U_{h^{-1}})^{-1} = U_{h}\). Thus, it is a  unitary operator. To see that the mapping \(\pi\) is a group homomorphism it is enough to note that for \(h,g \in \AC(Q)\) we have
    \[\frac{\der \lambda^\N g \circ h}{\der \lambda^\N} = \frac{\der \lambda^\N g}{\der \lambda^\N} \circ h \cdot \frac{\der \lambda^\N h}{\der \lambda^\N}.\]

    To check continuity it is enough to check continuity at \(\id \in \AC(Q)\), since \(\pi\) is a homomorphism. Let \(h_n \in \AC(Q)\) be a sequence such that \(h_n \overset{d_{\AC}}{\to}\id\) then also \((h_n)^{-1} \overset{d_{\AC}}{\to}\id\). Let \(\varphi \in C(Q)\). We have 
    \begin{align*}
    \int_Q |U_{(h_n)^{-1}}(\varphi)(x)& - \varphi(x)|^2 \der \lambda^\N \\
    \qquad &\leq 2\int_Q \left|\left(\frac{\der \lambda^\N (h_n)^{-1}}{\der \lambda^\N}(x) \right)^{1/2}(\varphi((h_n)^{-1}(x))- \varphi(x))\right|^2 \der \lambda^\N \\ 
    &\qquad+2\int_Q |\varphi(x)|^2\left|\left( \frac{\der \lambda^\N (h_n)^{-1}}{\der \lambda^\N}(x) \right)^{1/2}-1\right|^2\der\lambda^\N \\
    &\leq 2\|\varphi \circ (h_n)^{-1}- \varphi\|_\infty^2  +2\|\varphi\|^2_\infty\int_Q \left|\left( \frac{\der \lambda^\N (h_n)^{-1}}{\der \lambda^\N}(x) \right)^{1/2}-1\right|^2\der\lambda^\N \\
    &\leq 2\|\varphi \circ (h_n)^{-1}- \varphi\|_\infty^2 + 2 \|\varphi\|_\infty^2\int_Q\left|1- \frac{\der \lambda^\N (h_n)^{-1}}{\der \lambda^\N}(x) \right| \der\lambda^\N
    \end{align*}}
    where we used the fact that 
    \[\int_Q \frac{\der \lambda^\N (h_n)^{-1}}{\der \lambda^\N}(x)\der \lambda^\N = \int_Q 1 \der \lambda^\N = 1.\]
    The last term converges to \(0\) since \((h_n)^{-1} \overset{\dac }{\to} \id\). Since continuous functions are dense in \(L_2(Q,\lambda)\) we can use this and the fact that \(U_{(h_n)^{-1}}\) are isometric linear operators to get 
    \[\|U_{(h_n)^{-1}} \varphi - \varphi\|_{L_2(Q,\lambda^\N)} \to 0\]
    for every \(\varphi \in L_2(Q, \lambda^\N)\).

    To show injectivity suppose there is \(h \in \AC(Q)\), \(h \neq \id\) such that \(U_h(\varphi) = \varphi\) for every \(\varphi \in L_2(Q,\lambda^\N)\). There exists an open set \(B \subset Q\) such that \(h(B) \cap B = \emptyset\). Let \(\varphi \in L_2(Q,\lambda^\N)\) be the characteristic function of \(B\). Then
    \[ U_h(\varphi)(x) - \varphi(x) = -1\]
    for every \(x \in B\). Since \(\lambda^\N(B) > 0\) we get \(U_h \neq \id\) and \(\pi\) is injective.
\end{proof}

It was asked by Sabok whether a non-universal Polish group can induce a complete orbit equivalence relation \cite[Subsection 4.1]{open_questions}. This was first answered positively in \cite{COMPLETE_ORBIT_EQUIVALENCE_RELATIONS_AND_NON-UNIVERSAL_POLISH_GROUPS}, where a non-universal Polish group inducing a complete orbit equivalence relation is constructed. However, the construction is quite involved. The group \(\AC(Q)\) is a quite natural group that has the same property. We offer another fairly easy alternative. For any compact metric space \(K\) there exists a metrizable compactification \(\tilde{K}\) of \(\N\) such that \(\tilde{K} \setminus \N\) is homeomorphic to \(K\), see \cite{Compactifiactions_of_N} for details.
\begin{Theorem}
    Let \(\tilde{Q}\) be a compactification of \(\N\) as in {\cite[Proposition 2.3]{Compactifiactions_of_N}} such that \(\tilde{Q} \setminus \N \approx Q\).
    Then the group \(\mathcal{H}(\tilde{Q})\) is totally disconnected. Thus, it is a non-universal Polish group inducing a complete orbit equivalence relation.
\end{Theorem}  
\begin{proof}
    By {\cite[Proposition 2.3]{Compactifiactions_of_N}} any homeomorphism of \(Q\) can be extended to a homeomorphism of \(\tilde{Q}\). This easily implies \(E^H_{\mathcal{H}(Q)} \Bleq E^H_{\mathcal{H}(\tilde{Q})}\). Since \(E^H_{\mathcal{H}(Q)}\) is a complete orbit equivalence relation we get that \(E^H_{\mathcal{H}(\tilde{Q})}\) is as well. On the other hand, the restriction mapping \(\mathcal{H}(\tilde{Q}) \to \mathcal{H}(\N) = S_\infty\) is a continuous injection {\cite[Proposition 2.4]{Compactifiactions_of_N}} and thus \(\mathcal{H}(\tilde{Q})\) is totally disconnected, since \(S_\infty\) is totally disconnected. Since there are non-trivial connected Polish groups, the group \(\mathcal{H}(\tilde{Q})\) is not a universal Polish group.
\end{proof}

\subsection{Cantor space}
In this section, we investigate the hyperspace and the shift equivalence relations of bi-absolutely continuous homeomorphisms of the Cantor space. Again, we first establish that as long as the considered measures are non-atomic, locally positive probability Borel measures, the obtained groups are isomorphic and the complexities of the considered equivalences are the same as well. However, when compared to the case of the Hilbert cube, we will not be able to obtain a parallel with \Cref{Theorem: homeomorphic measures on the Hilbert cube}. The reason is that being
equal is too restrictive for measures on the Cantor space. The countable set of clopen sets offers a countable set of values that serve as an invariant that has to be preserved by two measures that can be mapped onto each other. However, we can relax the requirement of being equal to obtain a weaker analogue of \Cref{Theorem: homeomorphic measures on the Hilbert cube}. We start with a technical lemma used in the proof. 
\begin{Lemma}
\label{Lemma: good partitioning}
Let \(\mu,\nu\) be two non-atomic, locally positive probability Borel measures on the Cantor set. Let \(A,B\) be non-empty clopen subsets of the Cantor set such that \(\frac{\mu(A)}{\nu(B)} \in (1/2,2)\). Suppose that \(B\) is partitioned into clopen sets \(B_i\), \(i \in \{0, \dots,n\}\) for some \(n \in \N\). Then we can find a partition \(A_i\), \(i \in \{0, \dots n\}\) of \(A\) into clopen set such that \(\frac{\mu(A_i)}{\nu(B_i)} \in (1/2,2)\) for every \(i \in \{0, \dots, n\}\). 
\end{Lemma}
\begin{proof}
    We can find \(\varepsilon > 0\) such that
    \[\frac{\mu(A)}{\nu(B)}\cdot\left(1+\varepsilon\frac{\nu(B)}{\nu(B_n)}\right) \in (1/2,2),\]
    \[\frac{\mu(A)}{\nu(B)}\cdot\left(1-\varepsilon\frac{\nu(B)}{\nu(B_n)}\right) \in (1/2,2),\]
    \[\frac{\mu(A)}{\nu(B)}\cdot\left(1+\varepsilon\right) \in (1/2,2)\]
    and 
    \[\frac{\mu(A)}{\nu(B)}\cdot\left(1-\varepsilon\right) \in (1/2,2).\]
    Using the Sierpinski's theorem for non-atomic measures we can find a Borel set \(\tilde{A}_0\) such that
    \[\frac{\mu(\tilde{A}_0)}{\mu(A)} = \frac{\nu(B_0)}{\nu(B)}.\]
    By regularity of \(\mu\) we can find clopen subset \(A_0\) of \(A\) such that 
    \[\left|\frac{\frac{\mu(A_0)}{\mu(A)}}{\frac{\nu(B_0)}{\nu(B)}}-1\right|<\frac{\varepsilon}{n}.\]
    Proceeding by induction we can for \(i \in \{0,\dots,n-1\}\) find disjoint clopen subsets \(A_i\) of \(A\) such that 
    \[\left|\frac{\frac{\mu(A_i)}{\mu(A)}}{\frac{\nu(B_i)}{\nu(B)}}-1\right|<\frac{\varepsilon}{n}.\]
    Then for \(i \in \{0,\dots,n-1\}\) we have that \[\frac{\mu(A_i)}{\nu(B_i)} = \frac{\mu(A)}{\nu(B)}\cdot \frac{\frac{\mu(A_i)}{\mu(A)}}{\frac{\nu(B_i)}{\nu(B)}}\in (1/2,2).\] We can then set \[A_n = A \setminus (A_0 \cup \dots \cup A_{n-1})\] and from the choice of \(\varepsilon\) we get \[\frac{\mu(A_n)}{\nu(B_n)} \in (1/2,2).\]
    Taking \(\varepsilon\) small enough, we can ensure that all the sets \(A_i\), \(i \in \{0,\dots, n\}\) are non-empty.
\end{proof}
\begin{Proposition}
\label{Proposition: equivalent measures}
    Let \(\mu,\nu\) be two non-atomic, locally positive probability Borel measures on the Cantor set. Then there exists \(h \in \mathcal{H}(2^\N)\) such that \(\mu \ll \nu h \) and \(\nu h \ll \mu\).
\end{Proposition}
\begin{proof}
     Using a back-and-forth argument and induction, for \(i \in \N\) we will construct two sequences of clopen covers \(\mathcal{A}^\nu_i = \{A^\nu_{i,j}; j \in \{0, \dots ,n_i\}\}\) and \(\mathcal{A}^\mu_i = \{A^\mu_{i,j}; j \in \{0, \dots ,n_i\}\}\) where \(n_i \in \N\) such that the following holds:
    \begin{itemize}
        \item for every \(i \in \N\) the sets in \(\mathcal{A}^\nu_i\) respectively \(\mathcal{A}^\mu_i\) are pairwise disjoint,
        \item the collections \[\mathcal{A}^\nu = \{A; A \in \mathcal{A}^\nu_i \text{ for some \(i \in \N\)}\},\] \[\mathcal{A}^\mu = \{A; A \in \mathcal{A}^\mu_i\text{ for some \(i \in \N\)} \}\] form bases of open sets,
        \item for every \(i \in \N\), \(j \in \{0, \dots, n_{i+1}\}\) and \(l \in \{0, \dots, n_i\}\) we have \(A^\mu_{i+1,j} \subset A^{\mu}_{i,l}\) if and only if \(A^\nu_{i+1,j} \subset A^{\nu}_{i,l}\),
        \item for every \(i \in \N\) and \(j \in \{0, \dots n_i\}\) we have \[\frac{\mu(A^\mu_{i,j})}{\nu(A^\nu_{i,j})} \in (\frac{1}{2},2).\]
        
    \end{itemize}
    Let \(\mathcal{B}_i = \{s \times 2^\N; s \in 2^{i}\}\) for \(i \in \N\).

    Let \(n_1 = 1\) and \(A_{1,j}^\mu = j \times 2^\N\), \(j \in 2\). By \Cref{Lemma: good partitioning} we can find clopen sets \(A_{1,0}^{\nu}, A_{1,1}^{\nu}\) such that \(A^\nu_{1,0} \cup A^\nu_{1,1} = 2^\N\) and
    \[\frac{\mu(A_{1,j}^\mu)}{\nu(A_{1,j}^\nu)} \in (1/2,2), \quad j \in 2.\] 

    We can find \(n > 0\) and disjoint collections \(\mathcal{S}^\nu_j \subset \mathcal{B}_n\), \(j \in 2\) such that \[\bigcup \mathcal{S}^\nu_j = A^\nu_{1,j}, \quad j \in 2.\] Let \(n_2 = 2^n-1\) and let \(A_{2,j}^{\nu}\), \(j \in \{0,\dots, 2^n-1\}\) be an enumeration of \(\mathcal{S}^\nu_0 \cup \mathcal{S}^\nu_1\). Again, by applying \Cref{Lemma: good partitioning} twice we can find clopen sets \(A_{2,j}^{\mu}\) such that 
    \[\frac{\mu(A_{2,j}^\mu)}{\nu(A_{2,j}^\nu)} \in (1/2,2), \quad j \in \{0,\dots, 2^n-1\},\] 
    \[\bigcup_{j \in \{0, \dots 2^n-1\}}A^\mu_{2,j} = 2^\N\]
    and for every \(j \in \{0,\dots,2^n-1\}\) every set \(A^\mu_{2,j}\) is contained in the set \(A^\mu_{1,l}\) for the \(l \in 2\) with \(A^{\nu}_{2,j}\subset A^{\nu}_{1,l}\).

    Then we can proceed by clear induction. By increasing \(n\) in every step we can make sure that we get bases at the end.

    We can then define \(h\) as the unique homeomorphism that for every \(i \in \N\) maps the clopen set \(A_{i,j}^\mu\) onto the clopen set \(A_{i,j}^\nu\) for every \(j \in \{0, \dots, n_i\}\). Since
    \[\frac{\mu(A^{\mu}_{i,j})}{\nu(A^{\nu}_{i,j})} \in (1/2,2)\] for every \(i \in \N\) and \(j \in \{0, \dots, n_i\}\) we get \(\mu \ll \nu h\) and \(\nu h \ll \mu\).
\end{proof}
\begin{Remark*}
    We can easily apply \Cref{Proposition: equivalent measures} to non-probability finite measures by normalizing them and then applying \Cref{Proposition: equivalent measures}. This does not change bi-absolute continuity of measures.
\end{Remark*}
\begin{Proposition}
\label{Proposition: isomorphic AC groups for Cantor}
        Let \(\mu,\nu\) be two non-atomic, locally positive probability Borel measures on \(2^\N\). Then \(\AC(2^\N,\mu)\) and \(\AC(2^\N,\nu)\) are topologically isomorphic.
\end{Proposition}
\begin{proof}
    By \Cref{Proposition: equivalent measures} we can find \(h \in \mathcal{H}(2^\N)\) such that \(\nu \ll \mu h\) and \(\mu h \ll \nu\). Then the mapping \(g \mapsto h^{-1}\circ g \circ h\) is a topological isomorphism between \(\AC(2^\N, \mu)\) and \(\AC(2^\N,\nu)\).
\end{proof}

The isomorphism from the proof of \Cref{Proposition: isomorphic AC groups for Cantor} of the different groups can be used to show that the complexities of the considered equivalences do not depend on the measure. As for the Hilbert cube, there is a natural measure on the Cantor space that will be considered in the sequel.

\begin{Notation}
    We will denote by \(\lambda^0\) the measure on \(2^\N\) that is the infinite product of equally distributed Bernoulli measures on the set \(2\) and we will denote by \(\AC(2^\N)\) the group \(\AC(2^\N,\lambda^0)\).
\end{Notation}

In \cite{Group_topologies_on_AC_homeo} de la Nuez González investigated what topologies can be considered on the space \(\AC(2^\N,\mu)\). He showed that for a special class of measures there is no group topology on \(\AC(2^\N,\mu)\) that lies strictly between the topology generated by the supremum metric and the topology generated by \(\dac\). We can use \Cref{Proposition: isomorphic AC groups for Cantor} to extend his result to all non-atomic, locally positive probability measures on the Cantor set.
\begin{Proposition}
    Let \(\mu\) be a non-atomic, locally positive probability measure on \(2^\N\). Then there is no group topology on \(\AC(2^\N,\mu)\) that is strictly between the topology generated by the supremum metric and the topology generated by \(\dac\).
\end{Proposition}
\begin{proof}
    By \Cref{Proposition: isomorphic AC groups for Cantor} and its proof there exists an algebraic isomorphism \(\psi: \AC(2^\N,\mu) \to \AC(2^\N)\) that is a homeomorphism with respect to \(\dac\) as well as the supremum metric. Suppose that there is a group topology \(\tau\) on \(\AC(2^\N,\mu)\) that is strictly in between the topology generated by the supremum metric and the topology generated by \(\dac\). Then \(\psi(\tau)\) is a group topology on \( \AC(2^\N)\). Since \(\psi\) is a topological isomorphism with respect to the topology generated by \(\dac\) as well as the topology generated by the supremum metric we get that \(\psi(\tau)\) is a topology on \(\AC(2^\N)\) that is strictly in between the topology generated by the supremum metric and the topology generated by \(\dac\). By {\cite[Theorem B]{Group_topologies_on_AC_homeo}} there is no such group topology on \(\AC(2^\N)\).
\end{proof}

To investigate the hyperspace equivalence relation, we again need some kind of extension theorem at hand. Similarly as for the Hilbert cube, there are extension theorems available for the Cantor set when considered without any measure. We will show that there is a parallel even with respect to a non-atomic, locally positive probability Borel measure. A crucial notion for extension theorems for the Cantor set is the notion of a Knaster-Reichbach covering.

\begin{Definition}[Knaster-Reichbach covering]
    Let \(X,Y\) be zero-dimensional spaces and let \(A\) be a closed nowhere dense subset of \(X\) and \(B\) be a closed nowhere dense subset of \(Y\). Suppose that \(h: A \to B\) is a homeomorphism, \(\mathcal{U} = \{U_i; i \in \N\}\) is a covering of \(X \setminus A\) by disjoint non-empty clopen subsets of \(X\) and \(\mathcal{V}=\{V_i; i \in \N\}\) is a covering of \(Y \setminus B\) by disjoint non-empty clopen subsets of \(Y\). Then \(\{(U_i,V_i); i \in \N\}\) is a Knaster-Reichbach covering or KR-covering for \((X\setminus A, Y \setminus B, h)\) if, whenever \(h_i: U_i \to V_i\) is a bijection for each \(i \in \N\), the combination mapping \(\tilde{h} = h \cup \bigcup_{i \in \N}h_i : X \to Y\) is continuous at the points of \(A\) and \(\tilde{h}^{-1}\) is continuous at the points of \(B\).
\end{Definition}
It turns out that Knaster-Reichbach coverings of nowhere dense closed sets always exist. 
\begin{Lemma}[{\cite[Lemma 3.2.2]{Homogeneous_zero_dim_absolute_Borel_sets}}]
\label{Lemma: existence of KR-covering}
    Let \(X\) and \(Y\) be zero-dimensional spaces, and let \(A\) and \(B\) be non-empty closed nowhere dense subsets of \(X\) and \(Y\), respectively. If \(h:A \to B\) is a homeomorphism, then there exists a KR-covering for \((X\setminus A, Y \setminus B, h)\).
\end{Lemma}
\Cref{Lemma: existence of KR-covering} can be used to show an extension theorem for homeomorphisms on the Cantor space. We will use it to prove an extension theorem for the Cantor set with a measure.
\begin{Proposition}
\label{Proposition: AC homeo extension on Cantor}
    Let \(A\) and \(B\) be closed, non-empty and nowhere dense subsets of \(2^\N\). Suppose there is a homeomorphism \(h: A \to B\) such that for every Borel \(E \subset A\) we have \(\lambda^0(E) = 0 \iff \lambda^0(h(E)) = 0 \). Then there is a homeomorphism \(\hat{h} \in \AC(2^\N)\) extending \(h\).
\end{Proposition}
\begin{proof}
    By \Cref{Lemma: existence of KR-covering} there exists a KR-covering \(\{(U_i,V_i); i \in \N\}\) for \((2^\N\setminus A , 2^\N\setminus B, h)\). Note that \(U_i \approx 2^\N\) and \(V_i \approx 2^\N\) for every \(i \in \N\) since they are clopen in \(2^\N\). For every \(i \in \N\) the measures \(\lambda^0|_{U_i}\) and \(\lambda^0|_{V_i}\) are non-atomic and locally positive on \(U_i\) and \(V_i\), respectively. By \Cref{Proposition: equivalent measures} for every \(i \in \N\) we can find a homeomorphism \(h_i: U_i \to V_i\) such that for every \(E \subset U_i\) Borel we have \(\lambda^0(E) = 0 \iff \lambda^0(h_i(E)) = 0\). The combination mapping \(\hat{h} = h \cup \bigcup_{i \in \N}h_i\) is the desired homeomorphism. 

    The fact that \(\hat{h}\) is a homeomorphism follows, since \(\{(U_i,V_i); i \in \N\}\) is a KR-covering  for \((2^\N\setminus A , 2^\N\setminus B, h)\). The fact that \(\hat{h}\) is bi-absolutely continuous with respect to \(\lambda^0\) follows from \(\sigma\)-additivity of measures.
\end{proof}
\begin{Remark}
    The specific choice of the measure in \Cref{Proposition: AC homeo extension on Cantor} does not play any important role in the proof and we could, in the same way, prove the result for any non-atomic, locally positive probability Borel measure on the Cantor space.
\end{Remark}

The case of the Cantor set is in some sense very similar to the case of the interval. We will be able to show that \(E^H_{\AC(2^\N)}\) is strictly more complex than \(E_{S_\infty}\) and we will use a very similar strategy as we used on the interval. Note that it can be shown that \(E^H_{\mathcal{H}(2^\N)}\) and \(E^H_{\mathcal{H}^+([0,1])}\) are of the same complexity; see \cite{Hjorth} for details.

\begin{Theorem}
\label{Theorem: AC equivalence on Cantor is above S_infinity}
    We have 
    \[E_{S_\infty} \Bleq E^H_{\AC(2^\N)}.\]
\end{Theorem}
\begin{proof}
    We will use \Cref{Theorem: Hyperspace relation on Cantor space is universal S_infty}.
    We can find a nowhere dense set \(C \subset 2^\N\) such that \(C \approx 2^\N\) and \(\lambda^0(C) = 0\). Let \(\tilde{F}\) be a homeomorphism between \(2^\N\) and \(C\) and let \(F\) be the mapping from \(\mathcal{F}(2^\N)\) to \(\mathcal{F}(C) \subset \mathcal{F}(2^\N)\) defined by \(G \mapsto \{\tilde{F}(x); x \in G\}\) for \(G \in \mathcal{F}(2^\N)\). It is easy to see that \(F\) is continuous. We will show that \(F\) is a reduction. 

    Let \(A,B \in \mathcal{F}(2^\N)\). If \(A \approx B\) then we have \(F(A) \approx F(B)\) and since both sets are nowhere dense and have \(\lambda^0\) measure \(0\), we can use \Cref{Proposition: AC homeo extension on Cantor}  to find \(h \in {\AC}(2^\N)\) such that \(h(F(A)) = F(B)\).
    
    On the other hand, if there is \(h \in \AC(2^\N)\) such that \(h(F(A)) = F(B)\) then \(\tilde{F}^{-1}\circ h \circ \tilde{F}|_{A}\) is a homeomorphism between \(A\) and \(B\). Thus, \(F\) is a reduction.
\end{proof}

To show the anti-classification we again use the equivalence \(\equiv\).

\begin{Theorem}
\label{Theorem: AC eqivalence on Cantor is not classifiable}
     Let \(N\) be the set of non-atomic,  locally positive, probability Borel measures on the Cantor space. Then we have
    \[\equiv|_{N} \Bleq E^H_{\AC(2^\N)}.\]
    Thus, \(E^H_{\AC(2^\N)}\) is not classifiable by countable structures.
\end{Theorem}
\begin{proof}
We start with the following claim. 
\begin{Claim}
\label{Claim: Borel assignment}
    Let \(\nu \in N\) be fixed. Then there is a Borel mapping \(H:N \to \mathcal{H}(2^\N)\) such that for every \(\mu \in \N\) we have \(\mu(E) = 0 \iff \nu(H(\mu)(E))=0\) for every Borel set \(E \subset 2^\N\).
\end{Claim}
\begin{claimproof}
    Let \(\mathcal{B}\) be the countable family of all non-empty clopen subsets of \(2^\N\). Fix an enumeration \(\mathcal{S} = (S_i)_{i=1}^{\infty}\) of all finite sequences of \(\mathcal{B}\) that form a disjoint partition of \(2^\N\). Let \(\mu \in N\). We will build finite clopen partitions \(\mathcal{P}_k = (P_{k,i})_{i=1}^{m_k}\) and \(\mathcal{Q}_k= (Q_{k,i})_{i=1}^{m_k}\) with diameters of the elements of the partitions tending to zero such that for every \(k \in \N\), \(i \in \{1, \dots, m_k\}\) and \(l \in \{1, \dots, m_{k+1}\}\) we have 
    \[\frac{\mu(P_{k,i})}{\nu(Q_{k,i})} \in (1/2,2)\]
    and
    \[P_{k+1,l} \subset P_{k,i} \iff Q_{k+1,l} \subset Q_{k,i}.\]
    Let \(\mathcal{Q}_1 = \mathcal{P}_1 = (2^\N)\). Let \(k > 1\). If \(k\) is even, let \(\mathcal{P}_{k}\) be the first element of \(\mathcal{S}\) that is a common refinement of \(\mathcal{P}_{k-1}\) and basic clopen sets of the form \(s \times 2^\N\) for \(s \in 2^k\). Let \(Q_{k}\) be the first element of \(\mathcal{S}\) that refines \(\mathcal{Q}_{k-1}\) in a way corresponding to the refinement \(\mathcal{P}_k\) of \(\mathcal{P}_{k-1}\) and satisfies the ratio condition. Such a sequence exists by \Cref{Lemma: good partitioning}. For \(k\) odd we exchange the roles of \(\mathcal{P}_k\) and \(\mathcal{Q}_k\). There exists a unique homeomorphism \(H(\mu)\) sending \(P_{k,i}\) to \(Q_{k,i}\) for all \(k \in \N\) and \(i \in \{1, \dots, m_k\}\). This homeomorphism has the desired property. It can be shown by the standard methods that the mapping \(\mu \mapsto H(\mu)\) is Borel.
    
\end{claimproof}
Fix \(\mathcal{B} = \{B_i; i \in \N\}\) a base of open sets for \(2^\N\) consisting of clopen sets. For every \(B \in \mathcal{B}\) let \(h_B: 2^\N \to B\) be a homeomorphism such that \(\lambda^0(E) = 0 \iff \lambda^0(h_B(E)) = 0\). Such a homeomorphism exists by \Cref{Proposition: equivalent measures}. For \(i \in \N\) let \(X_i \subset 2^\N\) be a nowhere dense closed set such that \(\lambda^0(X_i) = 0\) and consisting of a Cantor set \(K_i\) along with a countable set \(T_i\) such that every point of \(T_i\) has Cantor-Bendixon rank at most \(i\) and points with Cantor-Bendixon rank exactly \(i\) are dense along \(K_i\). Let \(C\) be a nowhere dense subset of \(2^\N\) such that \(C \approx 2^\N\) and \(\lambda^0(C) \geq 1/2\) and for every \(U \subset 2^\N\) open if \(U \cap C \neq \emptyset\) then \(\lambda^0(U\cap C) > 0\). Let \(e: 2^\N \to C\) be a homeomorphism such that for every \(E\subset 2^\N\) Borel we have 
\[\lambda^0(E) = 0 \iff \lambda^0(e(E)) = 0.\]
Such a homeomorphism exists by \Cref{Proposition: equivalent measures}. Let \(\hat{H}: N \to \mathcal{H}(2^\N)\) be the mapping from \Cref{Claim: Borel assignment} for the measure \(\lambda^0\) and let \(H = e \circ \hat{H}\). Then for every \(\mu \in N\) and \(E \subset 2^\N\) Borel we have 
\[\mu(E) = 0 \iff \lambda^0(H(\mu)(E)) = 0.\]

Let \(D \) be a countable dense subset of \(2^\N\) and \(\{d_i; i \in \N\}\) be an enumeration of \(D\) such that every point is repeated infinitely many times. Let \(\mu \in N\). For \(i \in \N\) let \(B_i^\mu\) be the first element of \(\mathcal{B}\) disjoint from \(C\) and \(B^\mu_{1}, \dots, B^\mu_{i-1}\) and such that \(B^\mu_i \subset B(H(\mu)(d_i),1/i)\). Let \(X_i^\mu = h_{B^\mu_i}(X_i)\) and \(K_i^\mu = h_{B^\mu_i}(K_i)\). Define 
\[F(\mu) = C \cup \bigcup_{i \in \N} X_i^\mu.\]
Then \(F(\mu)\) is closed nowhere dense. It can be shown by standard methods that the mapping \(\mu \mapsto F(\mu)\) is a Borel mapping from \(N\) to \(\mathcal{F}(2^\N)\). We will show that it is a reduction. 

Let \(\mu, \nu \in N\). If there is \(h \in \AC(2^\N)\) such that \(h(F(\mu)) = F(\nu)\). Then from the properties of the sets \(X_i\) we have \(h(K_i^\mu) = K_i^\nu\). By this we also get \(h(H(\mu)(d_i)) = H(\nu)(d_i)\), which implies \(h|_{C} = H(\nu) \circ (H(\mu))^{-1}\). For every \(E \in 2^\N\) Borel we have 
\begin{align*}
    \mu(E) = 0 &\iff \lambda^0(H(\mu)(E)) = 0 \iff \lambda^0(H(\nu) \circ H(\mu)^{-1} \circ H(\mu) )(E) = 0 \\
    &\iff  \nu(E) = 0.
\end{align*}
Thus, \(\mu \equiv \nu\).

On the other hand, suppose that \(\mu \equiv \nu\). For \(x \in F(\mu)\) let 
\[\hat{h}(x) = \begin{cases}
    (H(\nu) \circ (H(\mu))^{-1})(x), \quad &x \in C \\
    (h_{B^\nu_i} \circ (h_{B^\mu_i})^{-1})(x), \quad  &x \in X^\mu_i, \ i \in \N.
\end{cases}\]
It can be easily shown that \(\hat{h}\) is a homeomorphism between \(F(\mu)\) and \(F(\nu)\). We will show that it preserves \(\lambda^0\) null sets. Let \(E \in F(\mu)\). Since \(\lambda^0(X^\mu_i) = \lambda^0(X^\nu_i) = 0 \) for every \(i \in \N\), it is enough to consider \(E \subset C\). Then
\begin{align*}
    \lambda^0(E) = 0 &\iff \mu(H(\mu)^{-1}(E)) = 0\iff \nu(H(\mu)^{-1}(E)) = 0 \\ 
    &\iff \lambda^0((H(\nu) \circ (H(\mu))^{-1})(E))= 0.
\end{align*}
By \Cref{Proposition: AC homeo extension on Cantor} there is \(h \in \AC(2^\N)\) extending \(\hat{h}\). Thus \(F(\mu) E^{H}_{\AC(2^\N)} F(\nu)\). Together we get that \(F\) is a reduction.
\end{proof}
\begin{Corollary}
\label{Corollary: AC equivalence on Cantor is strictly above S_infinity}
    We have \(E_{S_\infty} <_B E^{H}_{\AC(2^\N)}\).
\end{Corollary}
\begin{proof}
    This follows by \Cref{Theorem: AC equivalence on Cantor is above S_infinity}, \Cref{Theorem: AC eqivalence on Cantor is not classifiable} and
    \Cref{Theorem: measure equivalence is not classifiable}.
\end{proof}

It can be shown that the shift action of \(\AC(2^\N)\) is turbulent. The technique for showing it is again similar to the one used by Solecki in \cite{Polish_group_topologies}.

\begin{Theorem}
    The shift action of \(\AC(2^\N)\) is turbulent.
\end{Theorem}
\begin{proof}
    To show that \(\AC(2^\N)\) is dense in \(\mathcal{H}(2^\N)\) we will use \Cref{Proposition: equivalent measures}. Pick \(f \in \mathcal{H}(2^\N)\) and let \(\varepsilon > 0\) be arbitrary. We can find a disjoint partition \(V_1, \dots, V_n\) of \(2^\N\) into clopen sets such that 
    \[\max_{1 \leq i \leq n}\diam f(V_i) \leq \varepsilon.\]
    Since all sets \(V_i\) are clopen, we have \(V_i \approx 2^\N \approx f(V_i)\). By \Cref{Proposition: equivalent measures}, for every \(1 \leq i \leq n\) we can find a homeomorphism \(g_i\) between \(V_i\) and \(f(V_i)\) such that for every  Borel \(E \subset V_i\) we have 
    \[\lambda^0(E) = 0 \iff \lambda^0(g_i(E)) = 0.\]
    Now, let
    \[g(x) = g_i(x), \ x \in V_i, \ i \in \{1,\dots, n\}.\]
    By the additivity of the measure, we get \(g \in \AC(2^\N)\) and we have 
    \begin{align*}
        \sup_{x \in 2^\N}d(g(x),f(x)) &\leq \max_{1 \leq i \leq n}\sup_{x \in V_i}d(g_i(x),f(x))  \\ &\leq \max_{1 \leq i \leq n} \sup_{y_1,y_2 \in f(V_i)}d(y_1,y_2) \leq \max_{1 \leq i \leq n} \diam f(V_i) \leq \varepsilon
    \end{align*}
    Since \(\AC(2^\N)\) is not equal to \(\mathcal{H}(2^\N)\), by \Cref{Theorem: proper subsets of Polish groups are meager} we get that \(\AC(2^\N)\) is meager in \(\mathcal{H}(2^\N)\). Thus, every orbit of the shift action is meager and dense.

    It remains to show that every local orbit is somewhere dense. Again, since we are dealing with the shift action, it is enough to show that every local orbit of the identity is somewhere dense. Let \(U \subset \mathcal{H}(2^\N)\) be open such that \(\id \in U\) and \(V \subset \AC(2^\N)\) be open such that \(\id \in V\). We can find \(\delta_1 > 0\) and \(\delta_2 > 0\) such that \(B_\infty(\id,\delta_1) \subset U\) and \(B_{\AC}(\id,\delta_2) \subset V\). For every \(x \in 2^\N\) there exists a clopen neighborhood \(V_x\) of \(x\) such that \(\lambda^0(V_x) < \delta_2/3\). Thus, we can find a clopen disjoint partition \(V_0,\dots, V_n\) of \(2^\N\) such that \(\lambda^0(V_i) < \delta_2/3\) for every \(i \in \{0,\dots,n\}\). Let \(\delta_3 = \min_{i\neq j}d(V_i,V_j)\) and let \(0 < \delta < \min\{\delta_1,\delta_2/3,\delta_3\}\). 
    
    Pick \(h \in B_\infty(\id,\delta)\). Let \(\varepsilon > 0\). Since \(\AC(2^\N)\) is dense in \(\mathcal{H}(2^\N)\) we can find \(f \in \AC(2^\N)\) such that \(d_\infty(f,h) < \varepsilon\) and \(f \in B_\infty(\id,\delta)\). Since \(f \in B_\infty(\id,\delta_3)\), we have \(f(V_i) = V_i\) for every \(i \in \{0,\dots, n\}\). For \(i \in \{0, \dots,n\}\) let 
    \[g_i(x) =\begin{cases}
        f(x), \quad &x \in V_i \\
        x, \quad &x \in 2^\N \setminus V_i.
    \end{cases}\]
    Since \(f \in \AC(2^\N)\) we have that \(g_i \in \AC(2^\N)\) for every \(i \in \{0,\dots, n\}\). For \(i \in \{0, \dots,n\}\) let \(h_i = g_i \circ \dots \circ g_0 \circ \id\). Then we have \(h_n = f\) and for every \(i \in \{0, \dots,n\}\) we have 
    \[d_\infty(\id,h_i) \leq d_\infty(\id, f) < \delta.\]
    It remains to show that for every \(i \in \{0, \dots,n\}\) we have \(g_i \in B_{\AC}(\id, \delta_2)\). For every \(i \in \{0, \dots, n\}\) we have 
    \begin{align*}
        d_{\AC}(g_i, \id) &= d_\infty(g_i, \id) +  \int_{2^\N}\left|\frac{\der \lambda^0 g_i}{\der \lambda^0}-1\right|\der \lambda^0  \\
        &\leq \delta + \lambda^0(V_i) + \int_{V_i}\frac{\der \lambda^0 g_i}{\der \lambda^0}\der \lambda^0 \\
        &= \delta +  \lambda^0(g_i(V_i)) + \lambda^0(V_i) \leq \delta + 2\delta_2/3 < \delta_2.
    \end{align*}
    Thus, we have that the \(U \text{-} V\)-local orbit of \(\id\) is dense in \(B_\infty(\id,\delta)\) and we are done.
\end{proof}
 
\section{Questions}
In this section, we gather some open questions arising from our investigation.

We do not know whether the equivalences induced by \(\Dif([0,1])\) are analytic non-Borel. Note that any orbit equivalence relation induced by a Borel action of a Polish group is an analytic set. Also note that if for an analytic equivalence relation \(E\) we have \(E_{S_\infty} \Bleq E\) then \(E\) is a complete analytic set in particular it is non-Borel. Thus, \(E^H_{\AC}\) and \(E^C_{\AC}\) are complete analytic sets and, in particular, non-Borel.
\begin{Question}
    Is \(E^H_{\Dif([0,1])}\) Borel?
\end{Question}
\begin{Question}
    Is \(E^C_{\Dif([0,1])}\) Borel?
\end{Question}

When it comes to the equivalences induced by \(\AC([0,1])\) the natural question to ask is how high the complexity of the studied equivalences is. 
\begin{Question}
    Is it true that \(E_{G_\infty} \not\Bleq E^H_{\AC([0,1])}\)?
\end{Question}
However, there seems to be a lack of theory for proving anti-classification results with the complete orbit equivalence relation for equivalences that are known to be strictly above \(E_{S_\infty}\). The equivalence \(E^H_{\AC([0,1])}\) seems to be a reasonable candidate to lie strictly between \(E_{S_\infty}\) and \(E_{G_\infty}\). Intuitively, it seems that the structure of the interval could provide enough rigidity to prevent the reduction of \(E_{G_\infty}\) and, as is apparent from the proofs, the equivalence \(E^H_{\AC([0,1])}\) seems to be very close to \(E_{S_\infty}\).

The same question can be asked for the equivalences induced by \(\AC(2^\N)\).
\begin{Question}
    Is it true that \(E_{G_\infty} \not\Bleq E^H_{\AC(2^\N)}\)?
\end{Question}
\section*{Acknowledgment}
The authors acknowledge the use of artificial intelligence tools in obtaining the
results of this paper. The AI assistants ChatGPT and Claude were used for discussing the proofs of \Cref{Proposition: Dif is CLI}, \Cref{Theorem: AC(Q) is not universal} and \Cref{Theorem: AC eqivalence on Cantor is not classifiable} and for proofreading the manuscript. The authors take full responsibility for the content of the manuscript.

\bibliographystyle{abbrv}
\bibliography{bibliography}

@article {Subgroups_of_Homeo_without_polish_topology,
    AUTHOR = {Cohen, Michael P. and Kallman, Robert R.},
     TITLE = {{${\rm PL}_+(I)$} is not a {P}olish group},
   JOURNAL = {Ergodic Theory Dynam. Systems},
  FJOURNAL = {Ergodic Theory and Dynamical Systems},
    VOLUME = {36},
      YEAR = {2016},
    NUMBER = {7},
     PAGES = {2121--2137},
      ISSN = {0143-3857,1469-4417},
   MRCLASS = {22A05 (37E05 57S05)},
  MRNUMBER = {3568974},
MRREVIEWER = {Eliza\ Jab\l o\'nska},
       DOI = {10.1017/etds.2015.13},
       URL = {https://doi.org/10.1017/etds.2015.13},
}

@book {Kelley_General_topology,
    AUTHOR = {Kelley, John L.},
     TITLE = {General topology},
    SERIES = {Graduate Texts in Mathematics},
    VOLUME = {No. 27},
      NOTE = {Reprint of the 1955 edition [Van Nostrand, Toronto, Ont.]},
 PUBLISHER = {Springer-Verlag, New York-Berlin},
      YEAR = {1975},
     PAGES = {xiv+298},
   MRCLASS = {54-XX},
  MRNUMBER = {370454},
}

@article {Homeomorphic_measures_in_the_Hilbert_cube,
    AUTHOR = {Oxtoby, John C. and Prasad, Vidhu S.},
     TITLE = {Homeomorphic measures in the {H}ilbert cube},
   JOURNAL = {Pacific J. Math.},
  FJOURNAL = {Pacific Journal of Mathematics},
    VOLUME = {77},
      YEAR = {1978},
    NUMBER = {2},
     PAGES = {483--497},
      ISSN = {0030-8730,1945-5844},
   MRCLASS = {28A35 (28C15)},
  MRNUMBER = {510936},
MRREVIEWER = {A.\ H.\ Stone},
       URL = {http://projecteuclid.org/euclid.pjm/1102806462},
}

@article {The_complexity_of_the_homeomorphism_relation_between_compact_metric_spaces,
    AUTHOR = {Zielinski, Joseph},
     TITLE = {The complexity of the homeomorphism relation between compact
              metric spaces},
   JOURNAL = {Adv. Math.},
  FJOURNAL = {Advances in Mathematics},
    VOLUME = {291},
      YEAR = {2016},
     PAGES = {635--645},
      ISSN = {0001-8708,1090-2082},
   MRCLASS = {03E15 (46L35 54E45)},
  MRNUMBER = {3459026},
MRREVIEWER = {Dominique\ Lecomte},
       DOI = {10.1016/j.aim.2015.11.051},
       URL = {https://doi.org/10.1016/j.aim.2015.11.051},
}

@incollection {Polish_group_topologies,
    AUTHOR = {Solecki, S\l{}awomir},
     TITLE = {Polish group topologies},
 BOOKTITLE = {Sets and proofs ({L}eeds, 1997)},
    SERIES = {London Math. Soc. Lecture Note Ser.},
    VOLUME = {258},
     PAGES = {339--364},
 PUBLISHER = {Cambridge Univ. Press, Cambridge},
      YEAR = {1999},
      ISBN = {0-521-63549-7},
   MRCLASS = {54H11 (22A05 54E35)},
  MRNUMBER = {1720580},
MRREVIEWER = {Manuel\ Sanchis},
}

@book {Homogeneous_zero_dim_absolute_Borel_sets,
    AUTHOR = {van Engelen, A. J. M.},
     TITLE = {Homogeneous zero-dimensional absolute {B}orel sets},
    SERIES = {CWI Tract},
    VOLUME = {27},
 PUBLISHER = {Stichting Mathematisch Centrum, Centrum voor Wiskunde en
              Informatica, Amsterdam},
      YEAR = {1986},
     PAGES = {iv+133},
      ISBN = {90-6196-303-6},
   MRCLASS = {54H05},
  MRNUMBER = {851765},
MRREVIEWER = {Gabriel\ Debs},
}

@book {Real_and_complex_analysis,
    AUTHOR = {Rudin, Walter},
     TITLE = {Real and complex analysis},
   EDITION = {Third},
 PUBLISHER = {McGraw-Hill Book Co., New York},
      YEAR = {1987},
     PAGES = {xiv+416},
      ISBN = {0-07-054234-1},
   MRCLASS = {00A05 (26-01 30-01 46-01)},
  MRNUMBER = {924157},
}

@book {Classical_descriptive_set_theory,
    AUTHOR = {Kechris, Alexander S.},
     TITLE = {Classical descriptive set theory},
    SERIES = {Graduate Texts in Mathematics},
    VOLUME = {156},
 PUBLISHER = {Springer-Verlag, New York},
      YEAR = {1995},
     PAGES = {xviii+402},
      ISBN = {0-387-94374-9},
   MRCLASS = {03E15 (03-01 03-02 04A15 28A05 54H05 90D44)},
  MRNUMBER = {1321597},
MRREVIEWER = {Jakub\ Jasi\'nski},
       DOI = {10.1007/978-1-4612-4190-4},
       URL = {https://doi.org/10.1007/978-1-4612-4190-4},
}

@article {Differentibal_functions_defined_on_closed_sets_I,
    AUTHOR = {Whitney, Hassler},
     TITLE = {Differentiable functions defined in closed sets. {I}},
   JOURNAL = {Trans. Amer. Math. Soc.},
  FJOURNAL = {Transactions of the American Mathematical Society},
    VOLUME = {36},
      YEAR = {1934},
    NUMBER = {2},
     PAGES = {369--387},
      ISSN = {0002-9947,1088-6850},
   MRCLASS = {58C25},
  MRNUMBER = {1501749},
       DOI = {10.2307/1989844},
       URL = {https://doi.org/10.2307/1989844},
}

@article {Analytic_extensions_of_differnetiable_functions_defined_in_closed_sets,
    AUTHOR = {Whitney, Hassler},
     TITLE = {Analytic extensions of differentiable functions defined in
              closed sets},
   JOURNAL = {Trans. Amer. Math. Soc.},
  FJOURNAL = {Transactions of the American Mathematical Society},
    VOLUME = {36},
      YEAR = {1934},
    NUMBER = {1},
     PAGES = {63--89},
      ISSN = {0002-9947,1088-6850},
   MRCLASS = {26B05},
  MRNUMBER = {1501735},
       DOI = {10.2307/1989708},
       URL = {https://doi.org/10.2307/1989708},
}

@book {Gao,
    AUTHOR = {Gao, Su},
     TITLE = {Invariant descriptive set theory},
    SERIES = {Pure and Applied Mathematics (Boca Raton)},
    VOLUME = {293},
 PUBLISHER = {CRC Press, Boca Raton, FL},
      YEAR = {2009},
     PAGES = {xiv+383},
      ISBN = {978-1-58488-793-5},
   MRCLASS = {03-02 (03E15 22E05 28A05 28D05 37A20 54A35)},
  MRNUMBER = {2455198},
MRREVIEWER = {\'Etienne\ Matheron},
}

@article {Cofinal_families_of_Borel_equivalence_relations_and_quasiorders,
    AUTHOR = {Rosendal, Christian},
     TITLE = {Cofinal families of {B}orel equivalence relations and
              quasiorders},
   JOURNAL = {J. Symbolic Logic},
  FJOURNAL = {The Journal of Symbolic Logic},
    VOLUME = {70},
      YEAR = {2005},
    NUMBER = {4},
     PAGES = {1325--1340},
      ISSN = {0022-4812,1943-5886},
   MRCLASS = {03E15 (06A07 54D45 54H05)},
  MRNUMBER = {2194249},
MRREVIEWER = {Marek\ Balcerzak},
       DOI = {10.2178/jsl/1129642127},
       URL = {https://doi.org/10.2178/jsl/1129642127},
}

@article{COMPLETE_ORBIT_EQUIVALENCE_RELATIONS_AND_NON-UNIVERSAL_POLISH_GROUPS, 
title={COMPLETE ORBIT EQUIVALENCE RELATIONS AND NON-UNIVERSAL POLISH GROUPS},
DOI={10.1017/jsl.2026.10207}, 
journal={The Journal of Symbolic Logic}, 
author={Ding, Longyun and Li, Ruiwen and Peng, Bo},
year={2026}, 
pages={1–7}}

@article {Compactifiactions_of_N,
    AUTHOR = {Tsankov, Todor},
     TITLE = {Compactifications of {$\Bbb N$} and {P}olishable subgroups of
              {$S_\infty$}},
   JOURNAL = {Fund. Math.},
  FJOURNAL = {Fundamenta Mathematicae},
    VOLUME = {189},
      YEAR = {2006},
    NUMBER = {3},
     PAGES = {269--284},
      ISSN = {0016-2736,1730-6329},
   MRCLASS = {54H05 (54F50 54H15)},
  MRNUMBER = {2213623},
MRREVIEWER = {Rainer\ L\"owen},
       DOI = {10.4064/fm189-3-4},
       URL = {https://doi.org/10.4064/fm189-3-4},
}

@incollection {The_complexity_and_the_structure_and_classification_of_dynamical_systems,
    AUTHOR = {Foreman, Matthew},
     TITLE = {The complexity and the structure and classification of
              dynamical systems},
 BOOKTITLE = {Ergodic theory},
    SERIES = {Encycl. Complex. Syst. Sci.},
     PAGES = {529--576},
 PUBLISHER = {Springer, New York},
      YEAR = {[2023] \copyright 2023},
      ISBN = {978-1-0716-2387-9; 978-1-0716-2388-6},
   MRCLASS = {37A20 (03E15 22C05 22E40 37B10 37C15)},
  MRNUMBER = {4647089},
       DOI = {10.1007/978-1-0716-2388-6\_726},
       URL = {https://doi.org/10.1007/978-1-0716-2388-6_726},
}

@article {Computing_the_complexity_of_the_relation_of_isometry_between_separable_Banach_spaces,
    AUTHOR = {Melleray, Julien},
     TITLE = {Computing the complexity of the relation of isometry between
              separable {B}anach spaces},
   JOURNAL = {MLQ Math. Log. Q.},
  FJOURNAL = {MLQ. Mathematical Logic Quarterly},
    VOLUME = {53},
      YEAR = {2007},
    NUMBER = {2},
     PAGES = {128--131},
      ISSN = {0942-5616,1521-3870},
   MRCLASS = {03E15 (37A20 46B04)},
  MRNUMBER = {2308492},
MRREVIEWER = {Tam\'as\ M\'atrai},
       DOI = {10.1002/malq.200610032},
       URL = {https://doi.org/10.1002/malq.200610032},
}

@article {The_complexity_of_the_classification_problems_of_finite-dimensional_continua,
    AUTHOR = {Chang, Cheng and Gao, Su},
     TITLE = {The complexity of the classification problems of
              finite-dimensional continua},
   JOURNAL = {Topology Appl.},
  FJOURNAL = {Topology and its Applications},
    VOLUME = {267},
      YEAR = {2019},
     PAGES = {106876, 18},
      ISSN = {0166-8641,1879-3207},
   MRCLASS = {03E15 (54B05 54F15 54H05)},
  MRNUMBER = {4000128},
MRREVIEWER = {Alexander\ Yurievich\ Shibakov},
       DOI = {10.1016/j.topol.2019.106876},
       URL = {https://doi.org/10.1016/j.topol.2019.106876},
}

@article {Actions_by_the_classical_Banach_spaces,
    AUTHOR = {Hjorth, G.},
     TITLE = {Actions by the classical {B}anach spaces},
   JOURNAL = {J. Symbolic Logic},
  FJOURNAL = {The Journal of Symbolic Logic},
    VOLUME = {65},
      YEAR = {2000},
    NUMBER = {1},
     PAGES = {392--420},
      ISSN = {0022-4812,1943-5886},
   MRCLASS = {03E15 (46B99 54H15)},
  MRNUMBER = {1782128},
MRREVIEWER = {Marek\ Balcerzak},
       DOI = {10.2307/2586545},
       URL = {https://doi.org/10.2307/2586545},
}

@book {Hjorth,
    AUTHOR = {Hjorth, Greg},
     TITLE = {Classification and orbit equivalence relations},
    SERIES = {Mathematical Surveys and Monographs},
    VOLUME = {75},
 PUBLISHER = {American Mathematical Society, Providence, RI},
      YEAR = {2000},
     PAGES = {xviii+195},
      ISBN = {0-8218-2002-8},
   MRCLASS = {03E15 (22A05 37A20 37B99 54H05 54H20)},
  MRNUMBER = {1725642},
MRREVIEWER = {Miroslav\ Repick\'y},
       DOI = {10.1090/surv/075},
       URL = {https://doi.org/10.1090/surv/075},
}

@article {Turbulence_amalgamation_and_generic_automorphisms_of_homogeneous_structures,
    AUTHOR = {Kechris, Alexander S. and Rosendal, Christian},
     TITLE = {Turbulence, amalgamation, and generic automorphisms of
              homogeneous structures},
   JOURNAL = {Proc. Lond. Math. Soc. (3)},
  FJOURNAL = {Proceedings of the London Mathematical Society. Third Series},
    VOLUME = {94},
      YEAR = {2007},
    NUMBER = {2},
     PAGES = {302--350},
      ISSN = {0024-6115,1460-244X},
   MRCLASS = {03E15 (37B05)},
  MRNUMBER = {2308230},
MRREVIEWER = {Tam\'as\ M\'atrai},
       DOI = {10.1112/plms/pdl007},
       URL = {https://doi.org/10.1112/plms/pdl007},
}

@article {Conjugating_diffeo_survey,
    AUTHOR = {O'Farrell, A. G. and Roginskaya, M.},
     TITLE = {Conjugacy of real diffeomorphisms. {A} survey},
   JOURNAL = {Algebra i Analiz},
  FJOURNAL = {Rossi\u iskaya Akademiya Nauk. Algebra i Analiz},
    VOLUME = {22},
      YEAR = {2010},
    NUMBER = {1},
     PAGES = {3--56},
      ISSN = {0234-0852},
   MRCLASS = {37E05 (37C15 54H20)},
  MRNUMBER = {2641079},
MRREVIEWER = {Ian\ Short},
       DOI = {10.1090/S1061-0022-2010-01130-0},
       URL = {https://doi.org/10.1090/S1061-0022-2010-01130-0},
}

@book {Infinite_dimensiona_topology,
    AUTHOR = {van Mill, Jan},
     TITLE = {The infinite-dimensional topology of function spaces},
    SERIES = {North-Holland Mathematical Library},
    VOLUME = {64},
 PUBLISHER = {North-Holland Publishing Co., Amsterdam},
      YEAR = {2001},
     PAGES = {xii+630},
      ISBN = {0-444-50557-1},
   MRCLASS = {57N20 (46T10 54C35 57N17 58D15)},
  MRNUMBER = {1851014},
MRREVIEWER = {Taras\ Banakh},
}

@article{open_questions,
  title={Open questions in descriptive set theory and dynamical systems},
  author={Buzzi, J{\'e}r{\^o}me and Chandgotia, Nishant and Foreman, Matthew and Gao, Su and Garc{\'\i}a-Ramos, Felipe and Gorodetski, Anton and Maitre, Fran{\c{c}}ois Le and Rodr{\'\i}guez-Hertz, Federico and Sabok, Marcin},
  journal={arXiv preprint arXiv:2305.00248},
  year={2023}
}

@article {Completness_of_Boolean_algebras,
    AUTHOR = {Camerlo, Riccardo and Gao, Su},
     TITLE = {The completeness of the isomorphism relation for countable
              {B}oolean algebras},
   JOURNAL = {Trans. Amer. Math. Soc.},
  FJOURNAL = {Transactions of the American Mathematical Society},
    VOLUME = {353},
      YEAR = {2001},
    NUMBER = {2},
     PAGES = {491--518},
      ISSN = {0002-9947,1088-6850},
   MRCLASS = {03E15 (03C15 06E15)},
  MRNUMBER = {1804507},
MRREVIEWER = {Miroslav\ Repick\'y},
       DOI = {10.1090/S0002-9947-00-02659-3},
       URL = {https://doi-org.ezproxy.is.cuni.cz/10.1090/S0002-9947-00-02659-3},
}

@article {No_repre_for_Homeo,
    AUTHOR = {Megrelishvili, Michael G.},
     TITLE = {Every semitopological semigroup compactification of the group
              {$H_+[0,1]$} is trivial},
   JOURNAL = {Semigroup Forum},
  FJOURNAL = {Semigroup Forum},
    VOLUME = {63},
      YEAR = {2001},
    NUMBER = {3},
     PAGES = {357--370},
      ISSN = {0037-1912,1432-2137},
   MRCLASS = {54H15 (22A20 43A65 54D35)},
  MRNUMBER = {1851816},
MRREVIEWER = {Michael\ Tka\v cenko},
       DOI = {10.1007/s002330010076},
       URL = {https://doi.org/10.1007/s002330010076},
}

@book {A_course_in_abstract_harmonic_analysis,
    AUTHOR = {Folland, Gerald B.},
     TITLE = {A course in abstract harmonic analysis},
    SERIES = {Textbooks in Mathematics},
   EDITION = {Second},
 PUBLISHER = {CRC Press, Boca Raton, FL},
      YEAR = {2016},
     PAGES = {xiii+305 pp.+loose errata},
      ISBN = {978-1-4987-2713-6},
   MRCLASS = {43-01 (22-01 42-01 46-01)},
  MRNUMBER = {3444405},
MRREVIEWER = {D.\ L.\ Salinger},
}

@article {Orbit_cardinals,
    AUTHOR = {Hjorth, Greg},
     TITLE = {Orbit cardinals: on the effective cardinalities arising as
              quotient spaces of the form {$X/G$} where {$G$} acts on a
              {P}olish space {$X$}},
   JOURNAL = {Israel J. Math.},
  FJOURNAL = {Israel Journal of Mathematics},
    VOLUME = {111},
      YEAR = {1999},
     PAGES = {221--261},
      ISSN = {0021-2172,1565-8511},
   MRCLASS = {03E15 (03E45 54E50)},
  MRNUMBER = {1710740},
MRREVIEWER = {Klaas\ Pieter\ Hart},
       DOI = {10.1007/BF02810686},
       URL = {https://doi.org/10.1007/BF02810686},
}

@article {Group_topologies_on_AC_homeo,
    AUTHOR = {de la Nuez Gonz\'alez, J.},
     TITLE = {Group topologies on groups of absolutely bicontinuous
              homeomorphisms},
   JOURNAL = {Pacific J. Math.},
  FJOURNAL = {Pacific Journal of Mathematics},
    VOLUME = {344},
      YEAR = {2026},
    NUMBER = {1},
     PAGES = {103--143},
      ISSN = {0030-8730,1945-5844},
   MRCLASS = {57S05},
  MRNUMBER = {5116676},
       DOI = {10.2140/pjm.2026.344.103},
       URL = {https://doi.org/10.2140/pjm.2026.344.103},
}

@article{generation_and_genericity_for_AC,
  title={Generation and genericity of the group of absolutely continuous homeomorphisms of the interval},
  author={Ihli, Dakota Thor},
  journal={arXiv preprint arXiv:2203.02854},
  year={2022}
}
\end{document}